\documentclass[11pt,twoside]{amsart}

      \usepackage{amssymb}
      \usepackage{graphicx}
      \usepackage{dcolumn,indentfirst,cite,color,hyperref}

      \theoremstyle{plain}
     \newtheorem{thm}{Theorem}[section]
\newtheorem{lem}[thm]{Lemma}

\newtheorem{conj}[thm]{Conjecture}
\newtheorem{cor}[thm]{Corollary}
\newtheorem{pro}[thm]{Proposition}
\newtheorem{rmk}[thm]{Remark} 
\newtheorem{fact}[thm]{Fact}

\newtheorem{defi}[thm]{Definition}

\newtheorem{AFM}{{
\sc Claim}}

\def\prp#1){\hfill\break\hbox to\parindent{\hss{
\bf#1)}\enspace}
\ignorespaces}

\newcommand{\bess}{\begin{eqnarray*}}
\newcommand{\eess}{\end{eqnarray*}}
\input xy
\xyoption{all}

\usepackage{fancyhdr} 
\begin{document}

%




\author{Pascale Roesch}
\address{Centre de Math\' ematiques et Informatique (CMI),
Aix-Marseille Universit\'e, Technop\^ole Ch\^ateau-Gombert,
39, rue F. Joliot Curie, 13453 Marseille Cedex 13, FRANCE}
\email{pascale.roesch@cmi.univ-mrs.fr}

\author{Xiaoguang Wang}
\address{School of Mathematical Sciences, Zhejiang University, Hangzhou, 310027, China}
\email{wxg688@163.com}

 \author{Yongcheng Yin}
 \address{School of Mathematical Sciences, Zhejiang University, Hangzhou, 310027, China}
 \email{yin@zju.edu.cn}

    \title[]{Moduli space of cubic Newton maps}

       \begin{abstract} 
         In this article, we study the topology and bifurcations of  the moduli space $\mathcal{M}_3$ of  cubic Newton maps. It's a subspace of the moduli space of cubic rational maps, carrying the Riemann orbifold  structure   $(\mathbb{\widehat{C}}, (2,3,\infty))$.
We prove two results:  

 $\bullet$ The boundary of the unique unbounded  hyperbolic component is a Jordan arc  and the  boundaries  of all other   hyperbolic components are Jordan curves;

 $\bullet$ The Head's angle map is surjective and monotone.   
 The fibers of this map are characterized  completely.
  
 The first result  is a  moduli space analogue of the first author's dynamical regularity theorem \cite{Ro08}. The second result  confirms a conjecture  of Tan Lei.
   \end{abstract}



     \date{\today}


   \maketitle


\section{Introduction}\label{int}

Let $P$ be a polynomial of degree $d\geq2$. It can be written as
$$P(z)=a_d z^d+a_{d-1}z^{d-1}+\cdots+ a_1z+a_0,$$
where $a_0,\cdots, a_d$ are  complex numbers and $a_d\neq0$.
 The {\it Newton's method} $N_P$ of $P$ is defined by
$$N_P(z)=z-\frac{P(z)}{P'(z)}.$$
The method,  also known as the {\it Newton-Raphson method} named after Isaac Newton and Joseph Raphson,  was
first proposed to find successively better approximations to the roots (or zeros) of a real-valued function.
 In 1879,  Arthur Cayley \cite{C} first  noticed the difficulties in generalizing the Newton's method to complex roots of polynomials with degree greater than 2 and complex initial values.  This opened the way to  study  the theory of iterations of holomorphic functions, as initiated by Pierre Fatou and Gaston Julia around 1920.
In the literature, $N_P$ is also called the {\it Newton map} of $P$.  The study of Newton maps attracts a lot of people both in complex dynamical systems and in computational mathematics.

 \subsection{What is known} The Newton maps can be viewed as a dynamical system as well as a root-finding algorithm. Therefore, it provides a rich source to study  from various purposes.
 Here is an incomplete list of what's known for
  Newton maps from different views: 

\vspace{4pt}

{\it Topology  of Julia set:} The simple connectivity of the immediate attracting basins of cubic Newton maps was first proven by  Przytycki \cite{P}.  Shishikura \cite{Sh} proved that the  Julia sets of the  Newton maps  of   polynomials are always connected by means of quasiconformal surgery.
Applying  the Yoccoz puzzle theory,  Roesch \cite{Ro08} proved the local connectivity of the Julia sets for most cubic Newton maps. 

The combinatorial structure of the Julia sets of cubic Newton maps was first studied by Janet Head \cite{He}.  With the help of Thurston's theory on characterization of rational maps,  Tan Lei \cite{Tan97} showed that every post-critically finite cubic Newton map can be constructed by mating  two cubic polynomials; Building on the thesis \cite{Mi},  Lodge, Mikulich and  Schleicher \cite{LMS1,LMS2} gave a combinatorial classification of post-critically
finite Newton maps.

\vspace{4pt}

{\it Root-finding algorithm:} As a root-finding algorithm, Newton's method is effective for quadratic polynomials but may fail in the cubic case.
 McMullen \cite{Mc}  exhibited a generally convergent algorithm (apparently different from Newton's method) for cubics  and proved that there are no generally convergent
purely iterative algorithms for solving polynomials of degrees four or more. On the other hand, by generalizing  a previous result of Manning \cite{Ma},  Hubbard, Schleicher  and Sutherland \cite{HSS} proved that for every $d\geq2$,
  there is a finite universal set $S_d$ with cardinality at most ${O}(d\log^2 d)$ such that for any root of any   suitably normalized   polynomial of degree $d$, there is an  initial point in $S_d$ whose orbit converges to this root under iterations of its Newton map.
 For further extensions of these results, see  \cite{S} and the references therein.

\vspace{4pt}


{\it Beyond rational maps:} The dynamics of Newton's method for transcendental entire maps are intensively studied by many authors. Bergweiler
\cite{Be} proved a no-wandering-domain theorem for transcendental Newton maps
that satisfy some finiteness assumptions. Haruta \cite{Ha} showed  that when the Newton's method is applied to the exponential function of the
form $Pe^Q$ (where $P,Q$ are polynomials), the attracting basins of roots have finite area.  For the  Newton maps of entire functions,   Mayer and Schleicher \cite{MS}  showed
that the immediate basins are simply connected
and unbounded;  Buff, R\"uckert and Schleicher  further   investigated the dynamical  properties of  these maps, see \cite{BR, RS}.
For the higher dimensional cases,
Hubbard and  Papadopol \cite{HP}, Roeder \cite{Ro} studied the Newton's methods for two complex variables. 

 \subsection{Main results}

Most above known results share a common feature. They focus on the dynamical aspect of the Newton maps. In this paper, we study  the 
 topology and bifurcations  in the moduli space.

We first give some notations.
Let $f$ be a rational map, ${\rm Aut}(\mathbb{\widehat{C}})$ be the group of  M\"obius transformations. We use
$[f]=\{\phi f \phi^{-1};\phi\in {\rm Aut}(\mathbb{\widehat{C}})\}$ to denote the M\"obius conjugate class of $f$.

  It is worth observing that for any polynomial  $P$ of degree $d\geq2$, a  simple root of $P$ corresponds to  a super-attracting fixed point of its  Newton map $N_P$, and that $N_P$ and $P$ have the same degree   if and only if  $P$ has $d-1$ distinct roots.
  The moduli space of degree $d$ Newton maps, denoted by   $\mathcal{M}_d$,  is defined as the following set
  $$\{[N_P]; P \text{ is a degree $d$ polynomial with $d-1$ distinct roots}\}\footnote{Note that our definition excludes the degree-$d$ Newton maps arising from  polynomials of degree $>d$.}$$
  endowed with an orbifold structure.
   A point $\tau=[f]\in\mathcal{M}_d$ is said {\it hyperbolic } if the rational map $f$ is  {\it hyperbolic} (i.e. all critical orbits  of $f$ are attracted by the attracting cycles). 
   It's known that the hyperbolic set   $\mathcal{M}_d^{hyp}=\{\tau\in \mathcal{M}_d; \tau \text{ is hyperbolic}\}$ is an open subset of $\mathcal{M}_d$. 
   A connected component of 
  $\mathcal{M}_d^{hyp}$ is  called  a {\it hyperbolic component}.

The space $\mathcal{M}_2$  is trivial, it consists of  a singleton because every Newton map of a quadratic polynomial with distinct roots is M\"obius conjugate to the square map $z^2$, see \cite{B}. The moduli space $\mathcal{M}_3$ of Newton maps of cubic polynomials with distinct roots is a  non-trivial space  with the lowest degree.  It carries a Riemann orbifold  structure  (see $\S$ \ref{m-s}) 
and  it has  a unique unbounded hyperbolic component $\mathcal{H}$. The component  $\mathcal{H}$
 consists of the M\"obius conjugate classes of cubic Newton maps $f$ for which the free critical point is contained in the immediate basin of a polynomial root. By quasiconformal surgery (see\cite[Remark 2.2]{Ro08}), one sees that $\mathcal{H}$ consists of points $\tau=[f]$ for which the cubic Newton map $f$ is quasiconformally conjugate to the cubic polynomial $z^3+3z/2$   near its Julia set,  see Figure 3 (right) of the current paper.  Thus all maps in $\mathcal{H}$ have polynomial dynamical behaviors, and they are not {\it genuine} cubic rational maps.
 The picture of the  moduli space $\mathcal{M}_3$  (with a suitable parameterization)   first appeared in Curry, Garnett and  Sullivan's paper \cite[Fig. 3.1]{CGS}. In \cite{Tan97}, Tan Lei gave some descriptions of
this space as well as the hyperbolic components.

The current paper  is the   continuation of the first named author's work, where the following fundamental result \cite[Theorem 6]{Ro08} is proven:

\begin{thm}[Roesch \cite{Ro08}] \label{dyn-jordan}
\label{main0}  For any cubic Newton map in $\mathcal{M}_3\setminus \mathcal{H}$ and with no Siegel disk,  all the  Fatou components are bounded by Jordan curves.
\end{thm}

Our first  main result  is an analogue  of Theorem \ref{dyn-jordan} in the moduli space:

\begin{thm}\label{hyp-boundary}
\label{main} In the moduli space $\mathcal{M}_3$ of cubic Newton maps,

1.  the boundary of the  hyperbolic component  $\mathcal{H}$ is a Jordan arc, and

2. the boundaries of all other hyperbolic components  are  Jordan curves.
\end{thm}

Here, we say a set  $\gamma$ is a {\it Jordan curve} (resp. {\it Jordan arc}) if it is homeomorphic to the circle $\mathbb{S}$
 (resp. the open interval $(0,1)$). 
 The reason that  $\partial \mathcal{H}$ is a  Jordan arc rather than a Jordan curve is that $\partial \mathcal{H}$
stretches towards the infinity `$\infty$' (abstract point where the  cubic Newton maps degenerate in the moduli space). This  $\infty$ is  exactly  $[{(2z^2-z)}/{(3z-2)}]$, the M\"obius conjugate class of  the Newton
map of cubic polynomial with double roots, say $z^2(z-1)$.  Nevertheless, the one-point compactification $\overline{\mathcal{M}_3}=\mathcal{M}_3\cup \{\infty\}$ of $\mathcal{M}_3$ is a topological sphere and we will see in Section \ref{rigidity} that $\partial \mathcal{H}\cup \{\infty\}$ is a Jordan curve in $\overline{\mathcal{M}_3}$.

 Theorem \ref{hyp-boundary}  
  is  partially  proven by the first named author in
 her thesis \cite{R97}, using parapuzzle techniques. The result and a sketch of proof  were announced in \cite{Ro99}. 
 The current paper will present a complete proof, with methods different from parapuzzle techniques, allowing us to treat all hyperbolic components.
 The strategy of the proof  
  follows the  treatment of the McMullen maps in \cite{QRWY}.  The difference is,  in the  McMullen map case, both the dynamical  plane and the parameter space have  rich symmetries,  allowing us to handle the combinatorial structure of the maps easily, however, for the cubic Newton maps, both the dynamical  plane and the parameter space  lack the  symmetries, we need exploit    
  {\it the Head's angle}  (see below) to classify the maps with different  combinatorial structure. The  complexity of  the combinatorics  makes the proof more delicate.

For each cubic Newton map $f$ with $[f]\in \mathcal{M}_3\setminus \mathcal{H}$, one can associate $f$   canonically with a  combinatorial number, {\it the Head's angle} $\boldsymbol{h}(f)\in(0,1/2]$  (see Section \ref{h-t} for precise definition).  As one will see in Section \ref{head}, this number characterizes   how and where the two adjacent immediate  basins of roots of the polynomial defining  $f$ touch.
It's known \cite{He,R97,Tan97} that  $\boldsymbol{h}(f)$
 is  contained in the set $\Xi$, which is defined by
 $$\Xi=\{\theta\in (0,1/2] ; 2^k\theta\in [\theta,1] \   ({\rm mod}  \ \mathbb{Z})  \text{ for all } k\geq0\}.$$
The set $\Xi\cup\{0\}$ is known to be  closed, perfect, totally disconnected and with Lebesgue measure zero, see \cite[Prop 2.16]{Tan97}.
 Tan Lei \cite{Tan97} proved that  every  rational number $\theta\in \Xi$ can be realized as the Head's angle of some cubic Newton map, by means of  mating  cubic polynomials  and applying   Thurston's Theorem \cite{DH2}.
 She conjectured \cite[p.229, Remark]{Tan97} that each irrational  angle $\theta\in \Xi$ can also be realized as the Head's angle of some cubic Newton map. 

Our second main result confirms this conjecture and characterizes the uniqueness of the realization:

 \begin{thm}\label{main3} Every angle $\theta\in \Xi$ can be realized as the Head's angle of some cubic Newton map. This Newton map is unique  up to M\"obius conjugation if and only if $\theta$ is not periodic under the  doubling map $t\mapsto 2t ({\rm mod} \  \mathbb{Z})$.
\end{thm}

Since the Head's angle is  invariant under M\"obius conjugation, it induces a map, still denoted by $\boldsymbol{h}$, 
from $\mathcal{M}_3\setminus \mathcal{H}$ to $\Xi$. The map $\boldsymbol{h}$ is defined by sending the  
point $\tau=[f]\in\mathcal{M}_3\setminus \mathcal{H}$ to $\boldsymbol{h}(f)$. For any $\theta\in \Xi$,
 let $\boldsymbol{h}^{-1}({\theta})=\{\tau\in \mathcal{M}_3\setminus\mathcal{H}; \mathcal{\boldsymbol{h}}(\tau)=\theta\}$ be the fibre of $\boldsymbol{h}$ over $\theta$.
 Let $\Theta_{per}$ and $ \Theta_{dya}$ be the set of  periodic and dyadic angles $\theta\in (0,1/2]$  under the doubling map $t\mapsto 2t   ({\rm mod} \  \mathbb{Z})$, respectively.
They are written precisely as
\bess
&\Theta_{per}=\{t\in (0,1/2]; 2^kt=t \ ({\rm mod}  \ \mathbb{Z})  \text{ for some }  k\geq1 \}, & \\
&\Theta_{dya}=\{t\in (0,1/2]; 2^kt=1 \ ({\rm mod} \  \mathbb{Z})  \text{ for some }  k\geq1 \}. &
\eess
It's easy to check that $\Theta_{dya}\setminus\Xi\neq \emptyset $ and $\Theta_{per}\setminus\Xi\neq \emptyset$ (e.g. $5/16\in \Theta_{dya}\setminus\Xi, \   2/7\in \Theta_{per}\setminus\Xi$). With these notations, 
 Theorem \ref{main3} can be reformulated in terms of the mapping property of $\boldsymbol{h}$:

\begin{thm}\label{main2} The Head's angle map $\boldsymbol{h}:\mathcal{M}_3\setminus\mathcal{H}\rightarrow \Xi$ is surjective and  monotone\footnote{A map
 is said {\it monotone} if each of its fibre is connected.}. Precisely,

1.  if $\theta\in   \Xi\cap \Theta_{per}$, then $\boldsymbol{h}^{-1}({\theta})$ is  homeomorphic  to   $\overline{\mathbb{D}}\setminus\{1\}$;  

 2.  if $\theta\in \Xi\setminus \Theta_{per}$, then
 $\boldsymbol{h}^{-1}({\theta})$  is  a singleton.
 
\noindent  Moreover,  $\boldsymbol{h}$ is continuous at $\lambda$ if and only if $\boldsymbol{h}(\lambda)\in\Xi\setminus\Theta_{dya}$.
\end{thm}

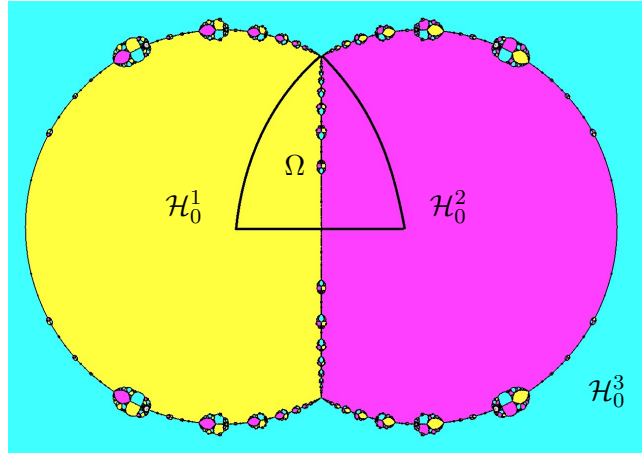
\begin{figure}[h]
\begin{center}
\includegraphics[height=6cm]{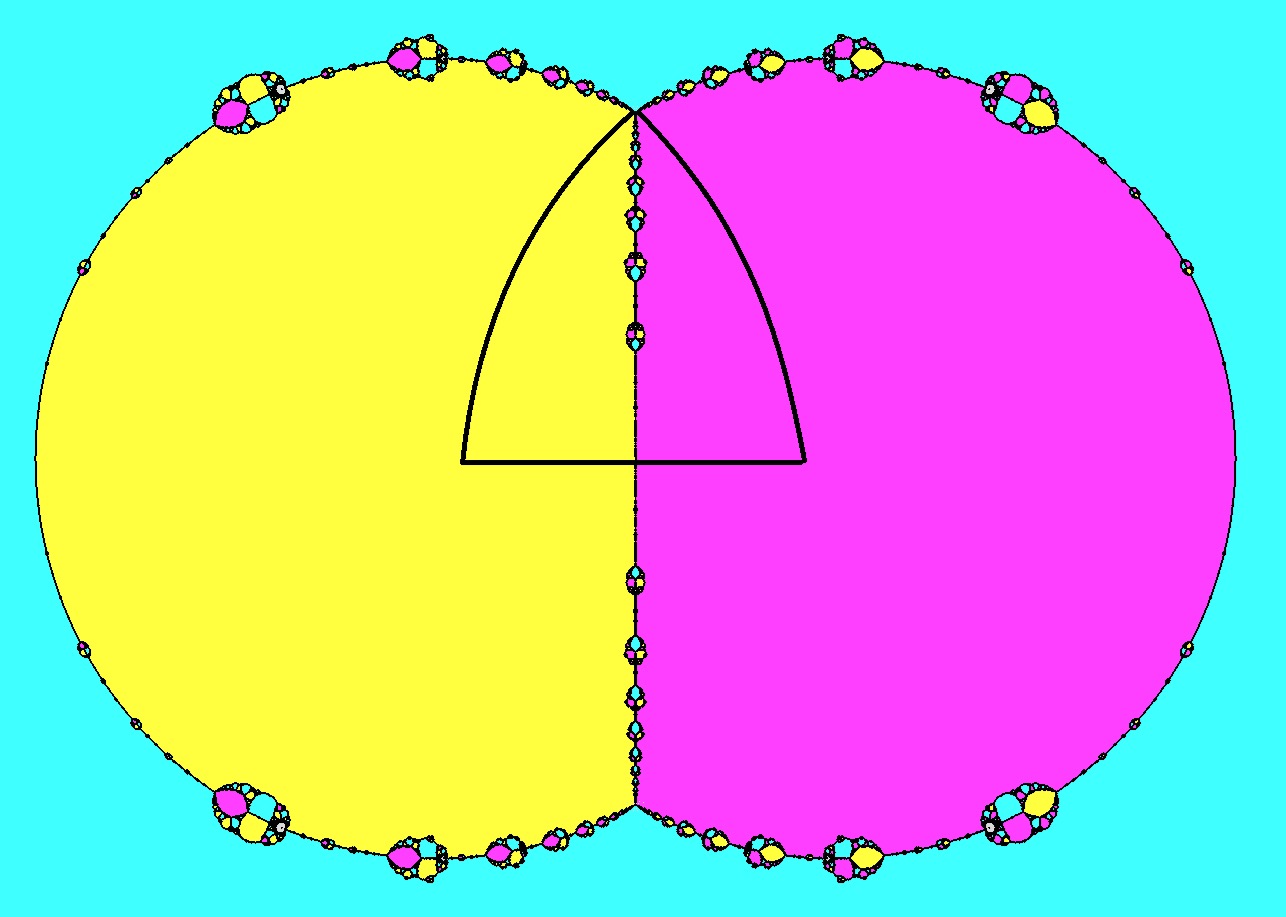}
\put(-135,105){$\Omega$} 
\put(-180,90){$\mathcal{H}_0^1$}   \put(-80,90){$\mathcal{H}_0^2$} \put(-20,20){$\mathcal{H}_0^3$}
 \caption{Head's angle and the parameter space $\mathcal{X}$.}
\end{center}\label{f5}
\end{figure}

Let's briefly describe the relation between the Head's angle and the space $\mathcal{M}_3$, with the help of Figure 1. The space $\mathcal{M}_3$ is six-fold covered by the parameter space $\mathcal{X}=\mathbb{C}\setminus\{\pm\frac{3}{2},0\}$ of  $N_\lambda(z)=\frac{2z^3-(\lambda^2-\frac{1}{4})}{3z^2-(\lambda^2+\frac{3}{4})}$ (see Section \ref{m-s}). There is a one-to-one correspondence between $\mathcal{M}_3$ and the region $\Omega$ (which is bounded by ray segments) plus some maps on the boundary $\partial \Omega$.
As is shown in Figure 1, there are two main hyperbolic components: $\mathcal{H}_0^1$ (in yellow) and $\mathcal{H}_0^2$ (in purple). Their boundaries in 
$\Omega$ can be parameterized by the angles in $(0,1/2)$ (or $(1/2,1)$). The complementary $\Omega-\mathcal{H}_0^1\cup\mathcal{H}_0^2$ in $\Omega$ is a `string of beans' (some kind of `Cantor necklace'), and each bean corresponds to a complementary interval of $\Xi$. The preimage of $\boldsymbol{h}$ either consists of a single point that is on the common boundary $\partial\mathcal{H}_0^1\cap \partial\mathcal{H}_0^2$, or is one of the beans (minus one boundary point).  This is the geometric picture of Theorem \ref{main2}.

Theorem \ref{main2}
  is slightly stronger than Theorem \ref{main3} for two reasons: first, it completely  characterizes the fibers of $\boldsymbol{h}$ over the whole set $\Xi$;
  second, it  characterizes the  points where the map $\boldsymbol{h}:\mathcal{M}_3\setminus\mathcal{H}\rightarrow \Xi$ is discontinuous.
  In fact, the characterization of the fiber of  $\boldsymbol{h}$ over $ \Xi\cap \Theta_{per}$ confirms another
 conjecture of Tan Lei \cite[p. 231, Conjecture]{Tan97}.  
  
 At last, we remark that for $d\geq4$, the moduli space $\mathcal{M}_d$ is a complex orbifold with dimension at least two. The boundary of the  hyperbolic components would  be much more complicated than that in dimension one.  We don't know how to deal with the higher dimensional case.

\subsection{Organization of the paper}   

In Section \ref{m-s}, we discuss the topological structure of the moduli space $\mathcal{M}_3$.  To this end, we introduce an one-parameter family of cubic Newton maps 
$\mathcal{F}$, parameterized by $\mathcal{X}$ which is a six-fold  covering of $\mathcal{M}_3$. Our main task is then reduced to study the 
topology and bifurcations in the underlying space $\mathcal{X}$.  The unbounded hyperbolic component $\mathcal{H}$ in $\mathcal{M}_3$ will split into three components
$\mathcal{H}_0^1, \mathcal{H}_0^2, \mathcal{H}_0^3$  in $\mathcal{X}$.

In Section \ref{description},   we give a dynamical parameterization of the hyperbolic components of  $\mathcal{F}$. It's the first step to study the topology of the hyperbolic components. Further 
steps will involve the dynamical properties of the cubic Newton maps,  as presented in the following three sections.

Precisely, Section \ref{fund-int} provides the basic knowledge of {\it the internal rays}, Section \ref{h-t} introduces   {\it the Head's angle} and its properties. 
The  Head's angle can  be used to classify the combinatorics of the cubic Newton maps    in a rough sense.
With these preparations, we recall the constructions  {\it the articulated rays} due to the first named author  in Section \ref{a-y} and highlight their {\it local stability property}.   The articulated rays are used to construct {\it the Yoccoz puzzle}  while their local stability property is used to study the boundary regularity of the hyperbolic components, as we shall see in the forthcoming sections.
 
 The aim of the next three sections is to show that   $\partial \mathcal{H}_0^\varepsilon$  is a Jordan curve. 
 To this end, we first characterize the maps on $\partial \mathcal{H}_0^\varepsilon$ and give a correspondence between the dynamical rays and parameter rays in Section \ref{char-maps}. Then we 
 revisit the dynamical  Yoccoz puzzle theory in Section \ref{yoccoz}. Using the theory, we establish the rigidity theorem in Section \ref{rigidity}. This enables us to prove further that  $\partial \mathcal{H}_0^\varepsilon$  is a Jordan curve. 
 
 In Section \ref{cap-jordan}, we will show that the boundaries of the hyperbolic components  of {\it capture type} are Jordan curves,  using technical arguments involving  the local stability property and the holomorphic motion  theory.  
 (Note that the boundaries of non-capture type hyperbolic components are already treated by Theorem \ref{1c}.)
 
 Finally, in Section \ref{final-s}, we will prove Theorems \ref{main2} and \ref{main3},  using  Theorem \ref{main}.

\subsection{Notations}       The most commonly used notations are as follows: 

$\mathbb{C}$: the set of all complex numbers or the complex plane.

$\mathbb{C}^*=\mathbb{C}\setminus\{0\}$: the complex plane minus the origin. 

$\mathbb{\widehat{C}}=\mathbb{C}\cup\{\infty\}$: the Riemann sphere.

$\mathbb{D}=\{z\in \mathbb{C}; |z|<1\}$: the unit  disk. 

$\mathbb{D}^*=\mathbb{D}\setminus\{0\}$: the unit  disk    minus the origin. 

$\mathbb{D}_r=\{z\in \mathbb{C}; |z|<r\}$: the  disk with radius $r$.

$\mathbb{S}=\mathbb{R}/\mathbb{Z}$: the unit circle.

$U\Subset V$ (or $V\Supset U$) means that the closure of $U$ is contained in $V$.

\subsection{Acknowledgement} We thank Tan Lei for  leading the first author to this problem and  offering generous ideas and constant
help.    Fei Yang provided the programs and  figures  in the paper.   X. Wang and Y. Yin are supported by National Science Foundation of China.

\section{Orbifold structure of $\mathcal{M}_3$} \label{m-s}

In this section, we discuss the orbifold structure of the moduli space $\mathcal{M}_3$.   For  a brief introduction to the Newton maps of quadratic and cubic polynomials, see Blanchard's paper \cite{B}.

Let $P$ be a polynomial of degree at least two. It can be factored as
$$P(z)=a(z-a_1)^{m_1}\cdots(z-a_d)^{m_d}$$
where $a$ is a nonzero complex number and $a_1,\cdots, a_d$ are  distinct roots of $P$, with multiplicities $m_1,\cdots,m_d\geq 1$, respectively.

Recall that the Newton map $N_P$ of $P$ is defined by 
$$N_P(z)=z-\frac{P(z)}{P'(z)}.$$
It satisfies that for every $1\leq k\leq d$, 
$$N_P(a_k)=a_k, \ N_P'(a_k)=\lim_{z\rightarrow a_k}\frac{P(z)P''(z)}{P'(z)^2}=\frac{m_k-1}{m_k}.$$

Therefore, each root $a_k$ of $P$ corresponds to  an   attracting  fixed point of $N_P$ with multiplier $\frac{m_k-1}{m_k}$.
It follows from the equation 
$$\frac{1}{N_P(z)-z}=-\sum_{k=1}^d\frac{m_k}{z-a_k}$$
that the degree of $N_P$ equals $d$, the number of  distinct roots of $P$.
One may also verifies that $\infty$ is a repelling fixed point of $N_P$ with multiplier 
$$\lambda_\infty=\frac{\sum_{k=1}^d m_k}{\sum_{k=1}^d m_k-1}.$$

The following result, essentially due to Janet Head, gives a characterization of the Newton maps of polynomials:

\begin{pro} \label{char} A rational map $f:\mathbb{\widehat{C}}\rightarrow \mathbb{\widehat{C}}$ of degree $d\geq2$
is the Newton map of a polynomial  if and only if
$f(\infty)=\infty$ and for all other fixed points $a_1, \cdots, a_d\in \mathbb{C}$, there exist  integers $m_k\geq 1$ such that
$f'(a_k)=\frac{m_k-1}{m_k}$ for all $1\leq k\leq d$.
\end{pro}

We remark that the rational map $f$ satisfying the latter half part of Proposition \ref{char} is exactly the Newton map of the polynomial
$$P(z)=a(z-a_1)^{m_1}\cdots(z-a_d)^{m_d}$$
with  $a\neq0$. See  \cite[Proposition 2.1.2]{He}  or \cite[Corollary 2.9]{RS} for a proof.


Now we turn to discuss the space of cubic Newton maps. We first introduce an one-parameter
family $\mathcal{F}$ of monic and centered cubic polynomials with  distinct roots. We will see that the space of Newton maps of  this family is
a  six-fold (branched) covering space of $\mathcal{M}_3$.

The family $\mathcal{F}=\{P_\lambda\}_{\lambda\in\mathcal{X}}$ that we are interested in consists of the  following cubic polynomials with three distinct roots:
 $$P_\lambda(z)=\Big(z+\frac{1}{2}+\lambda\Big)\Big(z+\frac{1}{2}-\lambda\Big)(z-1),\  \lambda\in  \mathcal{X}:=\mathbb{C}\setminus\Big\{\pm\frac{3}{2},0\Big\}.$$
 We remark that the Newton map of any cubic polynomial with distinct roots is M\"obius conjugate to the Newton map of some $P_\lambda$ (in fact, by Lemma \ref{2-1}, there are six choices of $\lambda$).
 
  The Newton map of $P_\lambda$ is
 $$N_\lambda(z)=z-\frac{P_\lambda(z)}{P'_\lambda(z)}=\frac{2z^3-(\lambda^2-\frac{1}{4})}{3z^2-(\lambda^2+\frac{3}{4})}.$$
For any $ \lambda\in \mathcal{X}$,   the map $N_\lambda$ has four critical points
 $$b_0(\lambda)=0, \ b_1(\lambda)=-\lambda-\frac{1}{2}, \ b_2(\lambda)=\lambda-\frac{1}{2}, \ b_3(\lambda)=1.$$
Note that when $\lambda\in \mathcal{X}\setminus\{\pm\frac{1}{2}\}=\mathbb{C}\setminus\big\{\pm\frac{3}{2}, \pm\frac{1}{2}, 0\big\}$, the last three critical points are  simple\footnote{A critical point $c$ is {\it simple}
if the local degree of $N_\lambda$ at $c$ is two.} and fixed, so the dynamical
 behavior of $N_\lambda$ is essentially determined by the orbit of the  free critical point $b_0(\lambda)=0$.
Let $\mathcal{G}\subset Aut(\mathbb{\widehat{C}})$ be the finite group of M\"obius maps permuting three points $\pm\frac{1}{2}, \infty$. In fact, this group is generated by $\gamma_1$ and $\gamma_2$, which are defined by 
$$\gamma_1(\lambda)=\frac{\lambda+\frac{1}{2}}{\lambda-\frac{1}{2}}-\frac{1}{2}, \ \gamma_2(\lambda)=\frac{\frac{1}{2}-\lambda}{\frac{1}{2}+\lambda}+\frac{1}{2}. $$
One may verify that $\mathcal{G}$ consists of six elements:
$$\mathcal{G}=\langle\gamma_1,\gamma_2\rangle=\{id,\gamma_1,\gamma_2, \gamma_1\circ\gamma_2, \gamma_2\circ\gamma_1, \gamma_2^{-1}\circ\gamma_1\circ\gamma_2 \}.$$

\begin{lem}\label{2-1} Let $\mathcal{Q}$ be the space of quasi-regular maps \footnote{A {\it quasi-regular map}
 is locally a composition of a holomorphic map and a quasi-conformal map.}  $f:\widehat{\mathbb{C}}\rightarrow\widehat{\mathbb{C}}$ of degree three, with four critical points, three of which are    simple and fixed.

1. (Characterization) Any rational map $f\in \mathcal{Q}$ is M\"obius conjugate to some cubic Newton map $N_\lambda$ with $\lambda\in \mathcal{X}$.

2.  (Conjugation) Two cubic Newton maps $N_{\lambda_1}$ and  $N_{\lambda_2}$ are  M\"obius conjugate  if and only if $\lambda_1=\gamma(\lambda_2)$ for some $\gamma\in \mathcal{G}$.

3. (Deformation) Let $\{L_t\}_{t\in\mathbb{D}}\subset \mathcal{Q}$ be a continuous family of quasi-regular maps and $\{\mu_t\}_{t\in\mathbb{D}}$ be a continuous family of  Beltrami differentials, such that

  (a). $L_0=N_{\lambda_0},\mu_0\equiv 0$;

  (b). For all $t\in\mathbb{D}$,   $L_t^* \mu_t=\mu_t$ and $\|\mu_t\|:={\rm ess.sup}|\mu_t(z)|<1$.

Let  $\psi_t$  solve  $\frac{\bar{\partial}\psi_t}{{\partial}\psi_t}=\mu_t$ with $0,1,\infty$ fixed.
 Then $\psi_t\circ L_t\circ\psi_t^{-1}=N_{\lambda(t)}$, where
 $$\lambda(t)=\frac{3(\psi_t(\lambda_0-\frac{1}{2})-\psi_t(-\lambda_0-\frac{1}{2}))}{2(2-\psi_t(\lambda_0-\frac{1}{2})-\psi_t(-\lambda_0-\frac{1}{2}))}, t\in\mathbb{D}.$$
\end{lem}

\begin{proof} 1. It's known from \cite[Lemma 12.1]{M} that any cubic   rational map $f\in \mathcal{Q}$  has four fixed points,  counted with multiplicity. Since $f$ already has three super-attracting fixed points, the fourth one must be repelling (see \cite{M}).
By M\"obius conjugation,  we may assume that  the repelling fixed point is at  $\infty$. Let $c_1, c_2, c_3$ be the  other three fixed points of $f$,  they are also critical points by assumption.
 By Proposition \ref{char}, we see  that $f$ is  the Newton map  of
$P(z)=(z-c_1)(z-c_2)(z-c_3)$.

Define the cross-ratio $\chi(z_1,z_2,z_3,z_4)$ of the quadruple $(z_1,z_2,z_3,z_4)$ by
$$\chi(z_1,z_2,z_3,z_4)=\frac{z_4-z_1}{z_4-z_3}\cdot\frac{z_2-z_3}{z_2-z_1}.$$

The map $\lambda\mapsto \chi(b_1(\lambda),b_2(\lambda), b_3(\lambda), \infty)$ is a M\"obius map.
So there is a unique
$\lambda_0\in\mathcal{X}$ satisfying
$$\chi(b_1(\lambda_0),b_2(\lambda_0), b_3(\lambda_0), \infty)=\chi(c_1,c_2, c_3, \infty).$$
This implies that there exists a  M\"obius map $\gamma\in Aut(\mathbb{\widehat{C}})$ sending
$c_1,c_2$, $c_3, \infty$ to $b_1(\lambda_0)$, $b_2(\lambda_0)$, $b_3(\lambda_0), \infty$, respectively.
By Proposition \ref{char}, the map $\gamma\circ f\circ\gamma^{-1}$ is the Newton map of $P_{\lambda_0}$.

2.  Note that if $h$ is a M\"obius conjugacy between $N_{\lambda_1}$ and $N_{\lambda_2}$, then $h$ maps the quadruple of fixed points
$(b_1(\lambda_1),b_2(\lambda_1), b_3(\lambda_1), \infty)$ to  $(b_{\varepsilon_1}(\lambda_2),b_{\varepsilon_2}(\lambda_2)$, $b_{\varepsilon_3}(\lambda_2),\infty)$,
 and vice visa, where $(\varepsilon_1,\varepsilon_2,\varepsilon_3)$ is a permutation of $(1,2,3)$.

 Since the  M\"obius maps  preserve the cross-ratio, we see  that $N_{\lambda_1}$ is M\"obius conjugate to $N_{\lambda_2}$ if and only if
$$\chi(b_1(\lambda_1),b_2(\lambda_1), b_3(\lambda_1), \infty)=\chi(b_{\varepsilon_1}(\lambda_2),b_{\varepsilon_2}(\lambda_2),b_{\varepsilon_3}(\lambda_2),\infty).$$

 Equivalently,  $\lambda_1=\gamma(\lambda_2)$ for some $\gamma\in \mathcal{G}$.

3.  By the first statement,    the map $\psi_t\circ L_t\circ\psi_t^{-1}$ is a cubic Newton map $N_{\lambda(t)}$, where
$\lambda(t)$ is  determined by
$$\chi(\psi_t(b_1(\lambda_0)), \psi_t(b_2(\lambda_0)),\psi_t(b_3(\lambda_0)),\infty)
=\chi(b_1(\lambda(t)), b_2(\lambda(t)), b_3(\lambda(t)),\infty).$$
Equivalently,
$$\frac{\psi_t(b_2(\lambda_0))-\psi_t(b_3(\lambda_0))}{\psi_t(b_2(\lambda_0))-\psi_t(b_1(\lambda_0))}
=\frac{b_2(\lambda(t))-b_3(\lambda(t))}{b_2(\lambda(t))-b_1(\lambda(t))}=\frac{\lambda(t)-3/2}{2\lambda(t)}.$$
Then we get $\lambda(t)$ as required.
 \end{proof}

The following result concerns the  orbifold structure of the moduli space $\mathcal{M}_3$. Here we will not give the precise definition of
orbifold, which can be found in  \cite[Appendix A]{Mc2}. We only use the following fact: A Riemann surface $S$ modulo  a finite subgroup of the automorphisms  group $Aut(S)$ is an orbifold, 
called {\it Riemann orbifold}.
\begin{thm}
The moduli space $\mathcal{M}_3$ is isomorphic to the Riemann orbifold $\mathcal{X}/\mathcal{G}\cong (\mathbb{\widehat{C}}, (2,3,\infty))$.
\end{thm}
\begin{proof} By Lemma \ref{2-1},  the Newton map of any cubic polynomial with distinct roots  is M\"obius conjugate to some $N_\lambda$, and any M\"obius conjugacy
 descends to an element in $\mathcal{G}$.
Therefore, 
$\mathcal{M}_3 \cong \mathcal{X}/\mathcal{G}\cong (\mathbb{\widehat{C}}, (2,3,\infty)).$ \end{proof}

\begin{rmk}
A geometric picture of $\mathcal{M}_3$ is that it is the Riemann sphere with 3 special points, one is a puncture, 
the other two are locally quotients of the unit disk $\mathbb{D}$ by period 2 and period 3 rotations, respectively.
\end{rmk}

\section{Description of the hyperbolic components}\label{description}

This section gives the dynamical parameterizations of the hyperbolic components of the cubic Newton maps. The ideas resemble Douady-Hubbard's proof of the connectivity of the Mandelbrot set and the parameterization of the bounded hyperbolic components using the multiplier map. We include the details, for the readers' convenience.

It's known from the previous section that for 
$\lambda\in \mathcal{X}=\mathbb{C}\setminus\big\{\pm\frac{3}{2},0\big\}$, the Newton map $N_\lambda$ has three super-attracting fixed points:
$$b_1(\lambda)=-\lambda-\frac{1}{2}, \ b_2(\lambda)=\lambda-\frac{1}{2}, \ b_3(\lambda)=1.$$
Let $B_\lambda^\varepsilon$ be the immediate attracting basin of $b_\varepsilon(\lambda), \varepsilon=1,2,3$.

According to Tan Lei \cite{Tan97},  there are three types of hyperbolic components classified by  the behavior of the free critical orbit
 $\{N_\lambda^n(0);n\geq0\}$:

 \vspace{5pt}
 
\textbf{Type A} ({\it adjacent critical points}): the free critical point $0$ is contained in some immediate basin $B_\lambda^\varepsilon$.

\textbf{Type C} ({\it capture}):  the free critical orbit  converges to some attracting fixed point $b_\varepsilon(\lambda)$ but $0\notin B_\lambda^\varepsilon$.

\textbf{Type D} ({\it disjoint attracting orbits}):  the free critical orbit converges to some attracting cycle other than $b_1(\lambda), b_2(\lambda), b_3(\lambda)$.

 \vspace{5pt}
 
The connectivity of the Julia set $J(N_\lambda)$ is proven by Shishikura  \cite{Sh}. 
This implies,  in particular,  that each Fatou component of $N_\lambda$ is simply connected.  
 
For any $k\geq0$  and any $\varepsilon\in\{1,2,3\}$,  we define a parameter set $\mathcal{H}^\varepsilon_k$ by:
$$\mathcal{H}^\varepsilon_k=\{\lambda\in
\mathbb{C}^*; k \text{ is the first integer such that } N_\lambda^{k}(0)\in B^\varepsilon_\lambda\}.$$


See Figure 1.
Here are  some remarks about these parameter sets:

Type A components consist of $\mathcal{H}^1_0, \mathcal{H}^2_0, \mathcal{H}^3_0$.
Any map  $N_\lambda$ with  $\lambda\in \mathcal{H}^1_0\cup \mathcal{H}^2_0\cup\mathcal{H}^3_0$ is quasi-conformally conjugate to the cubic polynomial $z^3+\frac{3}{2}z$  
near  its Julia set (see \cite[Remark 2.2]{Ro08}).
For $\varepsilon=1,2$, the {\it center} $c_\varepsilon$ of $\mathcal{H}^\varepsilon_0$ is the unique parameter $\lambda\in\mathcal{H}^\varepsilon_0$ such that the free critical point $0$
coincides with $b_\varepsilon(\lambda)$.
 One has $c_1=-\frac{1}{2}, c_2=\frac{1}{2}$. The center of $\mathcal{H}^3_0$ is $c_3=\infty$.
 
 Type C components consist of all components of $\mathcal{H}_k^\varepsilon$ with $k\geq1$.
One may verify  that  $\mathcal{H}^\varepsilon_1=\emptyset$ (in fact, if $\mathcal{H}^\varepsilon_1\neq \emptyset$, then for any $\lambda\in\mathcal{H}^\varepsilon_1$, the set $N_\lambda^{-1}(B^\varepsilon_\lambda)$ has two connected components, each contains critical point and maps to $B^\varepsilon_\lambda$ of degree two, implying that $N_\lambda$ has degree at least four. Contradiction!).
A component of $\mathcal{H}_k^\varepsilon$ with $k\geq2$ is called a {\it capture domain} of level $k$. 

Type D components are  the hyperbolic components of  {\it renormalizable type}. This is because each  map $N_\lambda$  of Type $D$ is {\it renormalizable} in the  following sense: there exist two topological disks $U, V$ with $U\Subset V$, an integer $p\geq1$, such that  $N_\lambda^p: U\rightarrow V$ is a polynomial-like map of degree two  with connected filled Julia set, see \cite[Section 6]{Ro08}.

The following observation is due to Tan Lei \cite[Lemma 1.2]{Tan97}:

\begin{lem}[Tan Lei]\label{intersect}  $0\in \partial\mathcal{H}_0^1\cap \partial\mathcal{H}_0^2$ and $\pm \sqrt{3}i/2
\in \partial\mathcal{H}_0^1\cap \partial\mathcal{H}_0^2\cap \partial\mathcal{H}_0^3$.
\end{lem}

For $\varepsilon\in \{1,2,3\}$ and $\lambda\in \mathcal{X}$,   the Green function $G_\lambda^\varepsilon: B_\lambda^\varepsilon\rightarrow[-\infty, 0)$ of $N_{\lambda}$ is defined by
$$G_\lambda^\varepsilon(z)=\lim_{k\rightarrow\infty}{2^{-k}}\log |N_{\lambda}^k(z)-b_\varepsilon(\lambda)|.$$
Note that $(G_\lambda^\varepsilon)^{-1}(-\infty)$ consists of the iterated pre-images of $b_\varepsilon(\lambda)$ in  $B_\lambda^\varepsilon$.
The B\"ottcher map $\phi_{\lambda}^\varepsilon$ of $N_\lambda$ is defined in  a neighborhood $U_\lambda^\varepsilon$ of $b_\varepsilon(\lambda)$ by
$$\phi_{\lambda}^\varepsilon(z)=\lim_{k\rightarrow\infty}(N_{\lambda}^k(z)-b_\varepsilon(\lambda))^{2^{-k}},$$
where $ U_\lambda^\varepsilon= B_\lambda^\varepsilon$  if $0\notin  B_\lambda^\varepsilon$, and   $U_\lambda^\varepsilon$ 
is the connected component of 
$\{z\in B_\lambda^\varepsilon; G_\lambda^\varepsilon(z)<G_\lambda^\varepsilon(0)\}$ containing $b_\varepsilon(\lambda)$
  if $0\in  B_\lambda^\varepsilon$.
By definition, the B\"ottcher map $\phi_{\lambda}^\varepsilon$ satisfies   $\phi_{\lambda}^\varepsilon(b_\varepsilon(\lambda))=0$ and $\phi_{\lambda}^\varepsilon(N_{\lambda}(z))=N_{\lambda}(z)^2$  in $U_\lambda^\varepsilon$. It is  unique because  the local degree of $N_\lambda$ at $b_\varepsilon(\lambda)$ is two.

\begin{thm}\label{1a} For $\varepsilon\in\{1,2\}$, the map $\Phi^\varepsilon_0:\mathcal{H}^\varepsilon_0\rightarrow\mathbb{D}$ defined by
\begin{equation*}
\Phi^\varepsilon_0(\lambda)=
\begin{cases}
\phi^\varepsilon_\lambda(0),\ \  &\text{ when } \lambda\neq c_\varepsilon,\\
0,\ \ &\text{ when } \lambda=c_\varepsilon,
\end{cases}
\end{equation*}
is a double cover ramified exactly at $c_\varepsilon$. For $\varepsilon=3$, the map
\begin{equation*}
\Phi^3_0:
\begin{cases}
\mathcal{H}^3_0\rightarrow\mathbb{D}^* \\
\lambda\mapsto\phi^3_\lambda(0)
\end{cases}
\end{equation*}
is a double cover.
\end{thm}

\begin{proof}   
We only consider the cases when $\varepsilon=1,2$. The proof for $\Phi_0^3$ is essentially same.
We first show that for $\varepsilon\in\{1,2\}$, the map $\Phi^\varepsilon_0$ is proper. To this end, we will show that
\begin{equation*}
\Phi^\varepsilon_0:
 \mathcal{H}^\varepsilon_0\setminus\{c_\varepsilon\}\rightarrow
\mathbb{D}^*
\end{equation*}
 is a covering map. It will then be equivalent to prove  that $\Phi^\varepsilon_0$ is a local isometry because the domain and the  range are hyperbolic surfaces. 
 
Fix different parameters $\lambda_1, \lambda_2\in \mathcal{H}^\varepsilon_0\setminus\{c_\varepsilon\}$ with $\arg \phi_{\lambda_1}^\varepsilon(0)=\arg \phi_{\lambda_2}^\varepsilon(0)$. We claim that there is a quasi-conformal  conjugacy $h$ between $N_{\lambda_1}$ and $N_{\lambda_2}$, satisfying that
$$d_{\mathbb{D}}(0, \|\mu_h\|)=d_{\mathbb{D}^*}(\phi_{\lambda_1}^\varepsilon(0), \phi_{\lambda_2}^\varepsilon(0)),$$
where $d_{S}(\cdot,\cdot)$ is the hyperbolic distance of  $S=\mathbb{D}, \mathbb{D}^*$. 

For $k=1,2$,  recall that  
  $U_{\lambda_k}^\varepsilon$ 
is the component of 
$\{z\in B_{\lambda_k}^\varepsilon; G_{\lambda_k}^\varepsilon(z)<G_{\lambda_k}^\varepsilon(0)\}$ containing $b_\varepsilon(\lambda_k)$.
We define a quasiconformal  map
$\delta:U_{\lambda_1}^\varepsilon\rightarrow U_{\lambda_2}^\varepsilon$
by $\delta=(\phi^\varepsilon_{\lambda_2})^{-1}\circ (z\mapsto z^{\frac{\alpha+1}{2}}\bar{z}^{\frac{\alpha-1}{2}})\circ \phi^\varepsilon_{\lambda_1}$,
 where $\alpha$ satisfies $|\phi^\varepsilon_{\lambda_1}(0)|^\alpha=|\phi^\varepsilon_{\lambda_2}(0)|$.

Because  ${\lambda_1}$ and ${\lambda_2}$ are in the same hyperbolic component,  we can find a  
quasi-conformal  conjugacy $\zeta_0$  between $N_{\lambda_1}$ and $N_{\lambda_2}$.  
In order to improve the quality of  $\zeta_0$, we   may modify  $\zeta_0$ so that
\begin{equation*}
\zeta_0=\begin{cases}
(\phi^\epsilon_{\lambda_2})^{-1}\circ  \phi^\epsilon_{\lambda_1}, & \text {  near }  b_\epsilon(\lambda_1),  \epsilon\in\{1,2,3\}\setminus\{\varepsilon\}, \\
\delta,   & \text{ in  }   U_{\lambda_1}^\varepsilon . 
\end{cases}
\end{equation*}

Then we can get  a sequence of quasi-conformal  maps
$\zeta_n:\mathbb{\widehat{C}}\rightarrow \mathbb{\widehat{C}}$
such that $N_{\lambda_2}\circ \zeta_{n+1}= \zeta_n\circ N_{\lambda_1}$ for
$n\geq0$, and $\zeta_{n+1}$ is isotopic to $\zeta_n$ rel $
N_{\lambda_1}^{-n}(P(N_{\lambda_1}))$, where  $P(N_{\lambda_1})$ is the post-critical set of $N_{\lambda_1}$.
Note that 
the dilatations $K(\zeta_n)_{n\geq1}$  are uniformly 
bounded above by $K(\zeta_0)$,  this implies that $\{\zeta_n\}$ is a normal family. 
Therefore the maps $\{\zeta_n\}$ converge to a   quasi-conformal map  $h$,  conjugating  $N_{\lambda_1}$ to $N_{\lambda_2}$.
One may observe  that the Beltrami coefficient $\mu_h=h_{\bar{z}}/h_{z}$  of $h$ satisfies
$$\|\mu_h\|=\|\mu_{\delta}\|=\bigg|\frac{\alpha-1}{\alpha+1}\bigg|=\bigg|\frac{\log
|\phi^\varepsilon_{\lambda_1}(0)|-\log |\phi^\varepsilon_{\lambda_2}(0)|}{\log |\phi^\varepsilon_{\lambda_1}(0)|+\log
|\phi^\varepsilon_{\lambda_2}(0)|}\bigg|.$$
Equivalently, $d_{\mathbb{D}}(0, \|\mu_h\|)=d_{\mathbb{D}^*}(\phi^\varepsilon_{\lambda_1}(0), \phi^\varepsilon_{\lambda_2}(0)).$
This completes the proof of the claim.

  Let $\alpha_t$ solve the Beltrami equation
$$\frac{\overline{\partial} \alpha_t}{\partial \alpha_t}=t\frac{\mu_h}{\|\mu_h\|},  \ t\in \mathbb{D},$$
normalized so that $\alpha_t$ fixes $0,1,\infty$.
Since $\mu_h$ is $N_{\lambda_1}$-invariant (i.e. $N_{\lambda_1}^*(\mu_h)=\mu_h$), by Lemma \ref{2-1},  there is a holomorphic map $\lambda:\mathbb{D}\rightarrow \mathcal{H}^\varepsilon_0\setminus\{c_\varepsilon\}$ such that $\lambda(0)=\lambda_1$ and  $\alpha_t\circ N_{\lambda_1}\circ \alpha_t^{-1}=N_{\lambda(t)}$.
Moreover $\lambda(\|\mu_h\|)=\lambda_2$.
Note that $\Phi_0^\varepsilon\circ\lambda: \mathbb{D}\rightarrow\mathbb{D}^*$ is holomorphic, by  Schwarz lemma,
$$d_{\mathbb{D}^*}(\phi^\varepsilon_{\lambda_1}(0), \phi^\varepsilon_{\lambda_2}(0))=d_{\mathbb{D}^*}(\Phi_0^\varepsilon\circ\lambda(0),\Phi_0^\varepsilon\circ\lambda(\|\mu_h\|))\leq d_{\mathbb{D}}(0, \|\mu_h\|).$$
As we have shown that, the first and last terms of the  above inequality are equal,  
we deduce that $\Phi_0^\varepsilon\circ\lambda: \mathbb{D}\rightarrow\mathbb{D}^*$ is a covering map.
It turns out that   $\Phi_0^\varepsilon: \mathcal{H}^\varepsilon_0\setminus\{c_\varepsilon\}\rightarrow\mathbb{D}^*$ is necessarily a covering map.
 From the fact that  $\mathcal{H}^\varepsilon_0\setminus\{c_\varepsilon\}$ has  at least two boundary components, we see that $\mathcal{H}^\varepsilon_0\setminus\{c_\varepsilon\}$ is isomorphic to $\mathbb{D}^*$ and $\Phi_0^\varepsilon$ is a proper map  from $\mathcal{H}^\varepsilon_0\setminus\{c_\varepsilon\}$ to $\mathbb{D}^*$.

It remains to show that $\Phi_0^\varepsilon$ has degree two. For this, we will show 
$$\Phi_0^\varepsilon(\lambda)=3b_\varepsilon(\lambda)^2+{O}(b_\varepsilon(\lambda)^3)  \text{ near }  c_\varepsilon,$$
and the degree of $\Phi_0^\varepsilon$ is indicated in the power of  $b_\varepsilon(\lambda)$.

To get the above expression, we may conjugate  $N_\lambda$ to the new map
$$ f_\lambda(z)=\beta_\lambda\circ N_\lambda\circ \beta_\lambda^{-1}(z)=z^2\frac{1+\frac{2}{3}\frac{a(\lambda)}{b_\varepsilon(\lambda)^2}z}{1+2z+
\frac{a(\lambda)}{b_\varepsilon(\lambda)^2}z^2},$$ 
where $$\beta_\lambda(w)=\frac{b_\varepsilon(\lambda)}{a(\lambda)}(w-b_\varepsilon(\lambda)), a(\lambda)=b_\varepsilon(\lambda)^2-\frac{1}{3}\Big(\frac{3}{4}+\lambda^2\Big).$$
Then
the map $\beta_\lambda$ sends the critical point $0$ of $N_\lambda$ to the critical point $c(\lambda)=-b_\varepsilon(\lambda)^2/a(\lambda)$ of $f_\lambda$.
The B\"ottcher map $\psi_\lambda$ of $f_\lambda$ is well-defined at $c(\lambda)$ and
$$\phi^\varepsilon_\lambda(0)=\psi_\lambda(c(\lambda))=3b_\varepsilon(\lambda)^2+{O}(b_\varepsilon(\lambda)^3).$$
This completes the proof.
\end{proof}

\begin{rmk} \label{tran} The following  diagram of conformal maps is commutative
$$\xymatrix{
\mathcal{H}_0^1 \ar[r]^\gamma \ar[rd]_{\Phi_0^{1}} &
\mathcal{H}_0^2 \ar[d]^{\Phi_0^{2}} &
\mathcal{H}_0^3 \ar[l]_\beta \ar[ld]^{\Phi_0^{3}}\\
&\mathbb{D}}$$
where $\gamma(\lambda)=-\lambda, \beta(\lambda)=\frac{1}{2}+\frac{1}{\lambda-\frac{1}{2}}$.
\end{rmk}

\begin{thm} \label{1b}
 Let $\mathcal{H}$ be a component of $\mathcal{H}^\varepsilon_k$ with  $k\geq2$.
Then the map   $\Phi_\mathcal{H}:\mathcal{H}\rightarrow{\mathbb{D}} $ defined by
$\Phi_\mathcal{H}(\lambda)=\phi_\lambda^\varepsilon(N_\lambda^{k}(0))$
is a conformal isomorphism.
\end{thm}

\begin{proof}  It's clear that $\Phi_\mathcal{H}$ is holomorphic.
To show that $\Phi_\mathcal{H}$ is a conformal map, it suffices to construct a holomorphic map $\lambda: \mathbb{D}\rightarrow \mathcal{H}$ such that
$\Phi_\mathcal{H}\circ\lambda =id$.

Fix $\lambda_0\in \mathcal{H}$ and set $\zeta_0=\Phi_\mathcal{H}(\lambda_0)$. Let $B$ be the Fatou component of $N_{\lambda_0}$ containing $N_{\lambda_0}^{k-1}(0)$.
For $\kappa>0$, let
$D(\zeta_0,\kappa)=\{\zeta\in\mathbb{D};
d_{\mathbb{D}}(\zeta,\zeta_0)<\kappa\}$ be the hyperbolic disk
centered at $\zeta_0$ with radius $\kappa$,  and $B_{\kappa}=
(N_{\lambda_0}|_{B})^{-1}\circ(\phi_{\lambda_0}^\varepsilon)^{-1}(D(\zeta_0,\kappa))$.

Let $E=\partial D(\zeta_0,\kappa)\cup\{\zeta_0\}$. We
 define a map $h: D(\zeta_0,\kappa)\times E\rightarrow\mathbb{D}$ by:
\begin{equation*}
h(\zeta,z)=
\begin{cases}
 z,\ \  &z\in \partial D(\zeta_0,\kappa),\\
\zeta,\ \ &z=\zeta_0.
\end{cases}
\end{equation*}

It is easy to check that

$\bullet \ h(\zeta_0,z)=z,\ z\in E $,

$\bullet$  for every fixed $\zeta\in D(\zeta_0,\kappa), \ z\mapsto
h(\zeta,z)$ is injective on $E$,

$\bullet$  for every fixed $z\in E, \ \zeta\mapsto h(\zeta,z)$ is
holomorphic in $D(\zeta_0,\kappa)$.

Thus $h: D(\zeta_0,\kappa)\times E\rightarrow\mathbb{D}$ is a
holomorphic motion parameterized by $D(\zeta_0,\kappa)$ with base
point $\zeta_0$. By the Holomorphic Motion  Theorem (see \cite{GJW} or
\cite{Sl}), $h$ admits an extension   $H:
D(\zeta_0,\kappa)\times \mathbb{\widehat{C}}\rightarrow\mathbb{\widehat{C}}$, which is also a holomorphic motion.
Let
$$\delta_{\zeta}(z)=\phi_{\lambda_0}^{-1}\circ
H(\zeta,\phi_{\lambda_0}^\varepsilon\circ N_{\lambda_0}(z)),  \ (\zeta,z)\in D(\zeta_0,\kappa)\times B_\kappa. $$
 It is easy to check that $\delta_{\zeta}$ satisfies:

$\bullet \ \delta_{\zeta_0}(z)=N_{\lambda_0}(z) \text{ for all }\
z\in B_{\kappa}$;

$\bullet \
\delta_{\zeta}(N_{\lambda_0}^{k-1}(0))=(\phi^\varepsilon_{\lambda_0})^{-1}(\zeta)$;

$\bullet \
\delta_\zeta:B_{\kappa}\rightarrow(\phi^\varepsilon_{\lambda_0})^{-1}(D(\zeta_0,\kappa))$
is a quasi-conformal map for any fixed $\zeta$;

$\bullet \ \zeta\mapsto\delta_{\zeta}(z)$ is holomorphic for any fixed $z\in B_{\kappa}$.

\

 Now we define a quasi-regular map
\begin{equation*}
L_\zeta(z)=\begin{cases}
 \delta_{\zeta}(z),\ \  &z\in B_{\kappa},\\
N_{\lambda_0}(z),\ \ &z\in \widehat{\mathbb{C}}\setminus
B_{\kappa}.
 \end{cases}
 \end{equation*}
Let $\sigma$ be the standard complex structure, we  construct an $L_\zeta$-invariant complex structure $\sigma_\zeta$
as follows:
\begin{equation*} \sigma_\zeta=
\begin{cases} (N_{\lambda_0}^m)^*(\delta_{\zeta}^*\sigma),\ \
 &\mathrm{in} \ N_{\lambda_0}^{-m}(B_{\kappa}), \  \ m\geq 0,\\
 \sigma,\ \ &\mathrm{in}\ \widehat{\mathbb{C}}\setminus \bigcup_{m\geq0}N_{\lambda_0}^{-m}(B_{\kappa}).
 \end{cases}
 \end{equation*}

The   Beltrami coefficient $\mu_\zeta$ of  $\sigma_\zeta$ satisfies
$\|\mu_\zeta\|<1$ for
 all $\zeta\in D(\zeta_0,\kappa)$ since $L_\zeta$ is holomorphic outside $B_\kappa$. By
 the Measurable Riemann Mapping Theorem \cite{A}  and Lemma \ref{2-1}, there exist 
 
 1.  a  family of quasi-conformal maps
$\psi_\zeta:\mathbb{\widehat{C}}\rightarrow \mathbb{\widehat{C}}$, parameterized by $ D(\zeta_0,\kappa)$, 
each 
 solves $\frac{\overline{\partial} \psi_\zeta}{{\partial} \psi_\zeta}=\mu_\zeta$ and  fixes $0,1,\infty$, and

2.  a holomorphic map $\lambda:D(\zeta_0,\kappa)\rightarrow\mathcal{H}$ such that
$\lambda(0)=\lambda_0$ and $\psi_\zeta\circ L_\zeta\circ \psi_\zeta^{-1}=N_{\lambda(\zeta)}$.

\

   For $\zeta\in D(\zeta_0,\kappa)$, we have
 \bess
\Phi_{\mathcal{H}}(\lambda(\zeta))&=&
\phi^\varepsilon_{\lambda(\zeta)}(N_{\lambda(\zeta)}^k(0))
=\phi^\varepsilon_{\lambda(\zeta)}\circ \psi_\zeta\circ L_\zeta^k\circ
\psi_\zeta^{-1}(0)\\
&=&\phi^\varepsilon_{\lambda(\zeta)}\circ \psi_\zeta\circ \delta_\zeta\circ
N_{\lambda_0}^{k-1}(0)=\phi^\varepsilon_{\lambda_0}\circ \delta_\zeta\circ
N_{\lambda_0}^{k-1}(0)=\zeta.
 \eess

Note that $\kappa$ can be arbitrarily large,  the map $\Phi_{\mathcal{H}}$
actually admits a global inverse map.
\end{proof}

Let $\mathcal{B}$ be a hyperbolic component of renormalizable type.
 Then for any $\lambda\in \mathcal{B}$, the map $N_\lambda$ has an  attracting cycle other than $b_1(\lambda), b_2(\lambda), b_3(\lambda)$, with period say $p$. 
  This  cycle contains   a point, say $z_\lambda$,   in the  Fatou  component  containing the critical point $0$. 
 Clearly, the map $\lambda\mapsto z_\lambda$ is holomorphic throughout $\mathcal{B}$.

\begin{thm}\label{1c}
  The multiplier map $\rho:\mathcal{B}\rightarrow \mathbb{D}$ defined by $\rho(\lambda)=(N_\lambda^p)'(z_\lambda)$ is a conformal map.
 It can be extended to a homeomorphism from $\overline{\mathcal{B}}$ to $\overline{\mathbb{D}}$.
\end{thm}

\begin{proof}
  By the  implicit function theorem, if the sequence   $\lambda_n$ in $\mathcal{B}$  approaches  the boundary  $\partial\mathcal{B}$, then
$|\rho(\lambda_n)|\rightarrow1$. This implies that   the  map  $\rho:\mathcal{B}\rightarrow \mathbb{D}$  is proper.
To show $\rho$ is conformal, it suffices to show $\rho$ admits a local inverse.

Fix $\lambda_0\in \mathcal{B}$ and set $\rho_0=\rho(\lambda_0)$. Let 
$A_0$ be the Fatou component containing  $0$ and  $z_{\lambda_0}$. There is a conformal map
$\phi:A_0\rightarrow \mathbb{D}$ such that $\phi(z_{\lambda_0})=0$ and
$\phi N_{\lambda_0}^p \phi^{-1}(z)=B_{\rho_0}(z)$, where
$$B_{\zeta}(z)=z\frac{z+\zeta}{1+\bar{\zeta}z}.$$

 Then there is a neighborhood $\mathcal{U}$ of $\rho_0$ and
 a continuous family of quasi-regular maps $\widetilde{B}:\mathcal{U}\times \mathbb{D}\rightarrow\mathbb{D}$ such that

 (a).  $\widetilde{B}(\rho_0,z)=B_{\rho_0}(z)$ for any $z\in \mathbb{D}$, ;

 (b). Fix any $\zeta\in \mathcal{U}$,  define a quasi-regular map by

 \begin{equation*}
\widetilde{B}(\zeta,z)=\begin{cases}
B_{\zeta}(z), \ &|z|<\delta,\\
 B_{\rho_0}(z), \  &\frac{1}{2}<|z|<1,\\
\text{interpolation},\  &\delta\leq |z|\leq \frac{1}{2}.
\end{cases}
\end{equation*}
 where $\delta$ is a small positive  number.


In this way,  we  get a continuous family $\{L_\zeta\}_{\zeta\in\mathcal{U}}$ of quasi-regular maps:
\begin{equation*}
L_\zeta(z)=\begin{cases}
 (N_{\lambda_0}^{p-1}|_{N_{\lambda_0}(A_0)})^{-1} ( \phi^{-1} \widetilde{B}(\zeta, \phi(z))),\ &z\in  A_0,\\
N_{\lambda_0}(z),\ &z\in\mathbb{\widehat{C}}\setminus A_0.
\end{cases}
\end{equation*}

We construct an $L_\zeta$-invariant complex structure $\sigma_\zeta$  such  that

$\bullet$ $\sigma_{\rho_0}$ is  the standard complex structure $\sigma$ on $\mathbb{\widehat{C}}$.

$\bullet$ $\sigma_\zeta$ is continuous with respect to $\zeta\in\mathcal{U}$.

$\bullet$ $\sigma_{\zeta}$ is the standard complex structure near the attracting cycle and outside $\cup_{k\geq0}N_{\lambda_0}^{-k}(A_0)$.

The Beltrami coefficient $\mu_\zeta$ of $\sigma_\zeta$ satisfies $\|\mu_\zeta\|<1$. By the Measurable Riemann Mapping Theorem \cite{A} and Lemma \ref{2-1},
there exists a continuous family of
quasi-conformal maps $\psi_\zeta$ solving $\frac{\bar{\partial} \psi_\zeta}{\partial \psi_\zeta}=\mu_\zeta$ and fixing $0,1, \infty$,  and a continuous map $\lambda:\mathcal{U}\rightarrow\mathcal{B}$ such that $\psi_\zeta\circ L_\zeta\circ\psi_\zeta^{-1}=N_{\lambda(\zeta)}.$
The multiplier of $N_{\lambda(\zeta)}$ at the  attracting cycle other than $b_1(\lambda(\zeta)), b_2(\lambda(\zeta)), b_3(\lambda(\zeta))$ is exactly $\zeta$. Therefore $\rho(\lambda(\zeta))=\zeta$ for all $\zeta\in\mathcal{U}$.
This implies that $\rho$ is a covering map. Since $\mathbb{D}$ is simply connected,
    $\rho$ is actually a conformal map.

The map $\rho$ has a continuous extension to the boundary $\partial \mathcal{B}$. By the implicit function theorem, the boundary $\partial \mathcal{B}$
is an analytic curve except at $\rho^{-1}(1)$. So $\partial \mathcal{B}$ is locally connected. Since for any $\lambda\in \partial \mathcal{B}$,
the multiplier $e^{2\pi i t}$ of the neutral  cycle of $N_\lambda$ is uniquely determined by the angle $t\in \mathbb{S}$, we conclude that 
  $\partial \mathcal{B}$
is a Jordan curve.  This is equivalent to say that  $\rho$ can be extended to  a homeomorphism from $\overline{\mathcal{B}}$ to $\overline{\mathbb{D}}$.
\end{proof}


%




\section{Fundamental domain and rays}\label{fund-int}

In this section, we first introduce the fundamental domain in the parameter plane, then study the basic properties of the dynamical internal rays and parameter rays.

\subsection{The  fundamental domain $\mathcal{X}_{FD}$} \label{parameter-ray}
 We first define
 $$\Omega=\Big\{\lambda\in\mathcal{X}=\mathbb{C}\setminus\big\{\pm\frac{3}{2},0\big\}; |\lambda-\frac{1}{2}|<1,  |\lambda+\frac{1}{2}|<1, {\rm Im} (\lambda)>0\Big\}.$$
 By Theorem \ref{1a}, the maps
\bess         
 \Phi_0^1:\mathcal{H}_0^1\cap\overline{\Omega}\rightarrow\mathbb{D}\cap \{ {\rm Im} (z)\geq0\}, \ 
 \Phi_0^2:\mathcal{H}_0^2\cap\overline{\Omega}\rightarrow\mathbb{D}\cap \{ {\rm Im} (z)\leq0\}
 \eess
 are homeomorphisms.

The {\it parameter ray} $\mathcal{R}^1_0(t)$ of angle $t\in [0, \frac{1}{2}]$ in $\mathcal{H}_0^1$ and the {\it parameter ray} $\mathcal{R}_0^2(\theta)$ of angle $\theta\in [\frac{1}{2},1]$ in $\mathcal{H}_0^2$ are defined respectively by
$$\mathcal{R}^1_0(t):=(\Phi_0^1)^{-1}((0,1)e^{2\pi i t}), \  \mathcal{R}_0^2(\theta):=(\Phi_0^2)^{-1}((0,1)e^{2\pi i \theta}).$$
Here are four examples of parameter rays:
$$\mathcal{R}_0^1(0)=(-1/2,0),\  \mathcal{R}_0^1({1}/{2})=\{{1}/{2}+e^{ia}; a\in ({2\pi}/{3}, \pi)\},$$
$$\mathcal{R}_0^2(1)=(0,{1}/{2}),\  \mathcal{R}_0^2({1}/{2})=\{-{1}/{2}+e^{ia}; a\in (0,{\pi}/{3})\}.$$

\

\begin{figure}[h]
\begin{center}
\includegraphics[height=7cm]{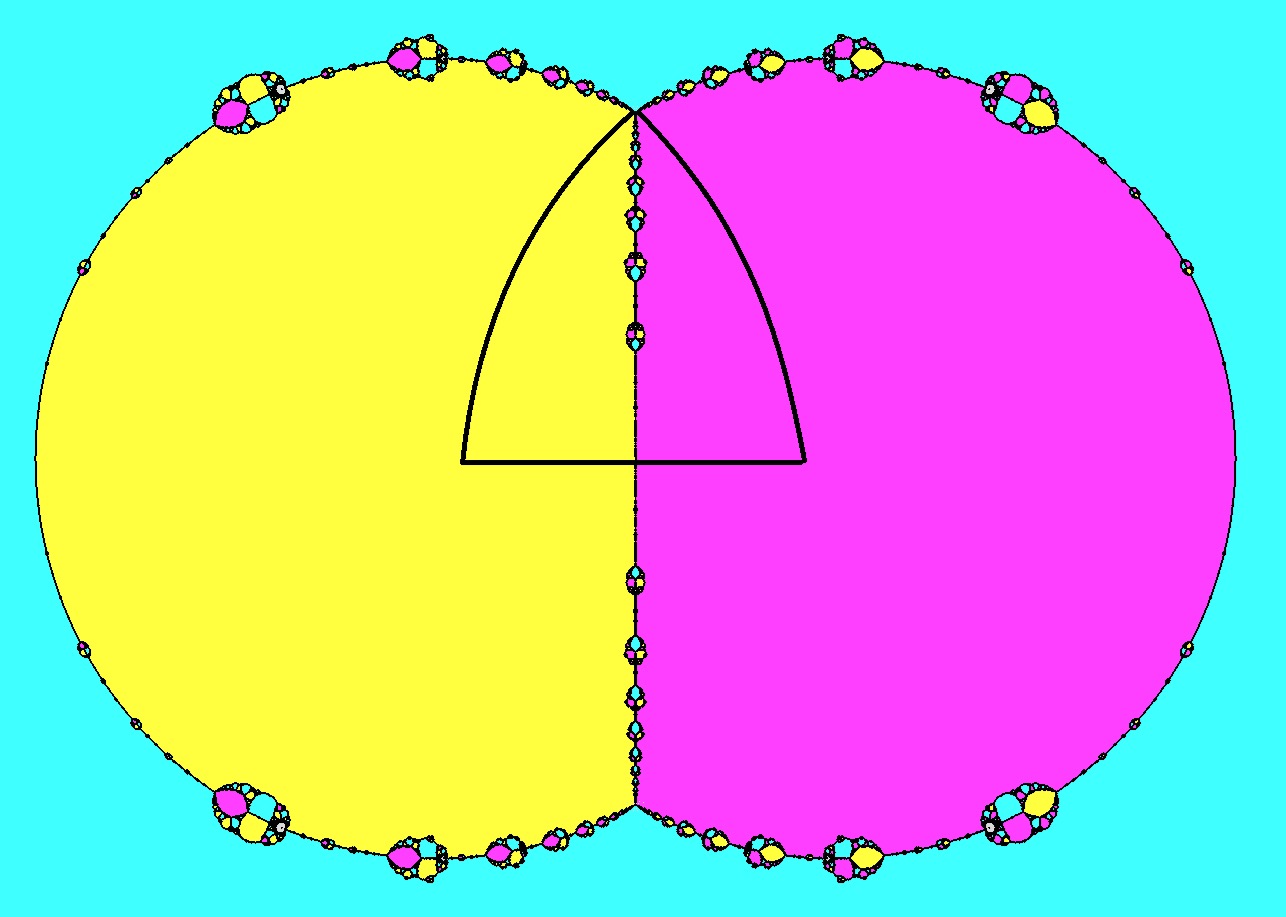}
\put(-200,90){$-\frac{1}{2}$}  \put(-182,97){$\bullet$} \put(-210,140){$\mathcal{R}_0^1(1/2)$}  \put(-175,85){$\mathcal{R}_0^1(0)$} \put(-135,85){$\mathcal{R}_0^2(1)$}  \put(-100,90){$\frac{1}{2}$}   \put(-108,97){$\bullet$}  
\put(-155,120){$\Omega$} \put(-113,140){$\mathcal{R}_0^2(1/2)$} \put(-155,185){$\sqrt{3}i/2$} \put(-143,172){$\bullet$}
 \caption{The region $\Omega$ is bounded by  four parameter rays.}
\end{center}\label{f5}
\end{figure}


By  Lemma \ref{2-1},   for each map $N_{\lambda}$, one can find  $\lambda_0\in \mathcal{X}_{FD}$, where
$$\mathcal{X}_{FD}=\Omega\cup\mathcal{R}_0^1(0)\cup\mathcal{R}_0^1({1}/{2})\cup\{\sqrt{3}i/2, -1/2\},$$
so that $N_{\lambda}$ is  conjugate to $N_{\lambda_0}$ by a M\"obius transformation.  Besides, no two maps in  $\mathcal{X}_{FD}$ are M\"obius conjugate.
For this reason, we call $\mathcal{X}_{FD}$  {\it the fundamental domain} of the parameter space. Clearly, there is a bijection between  $\mathcal{X}_{FD}$ and $\mathcal{M}_3$.



Note that  $\lambda=\sqrt{3}i/2$ is the only parameter in  $\mathcal{X}_{FD}$ for which the free critical point $0$ is mapped by $N_\lambda$ to the repelling fixed
 point $\infty$. In this case, the map $N_\lambda$ is post-critically finite and therefore has a locally connected Julia set. In fact, each Fatou component is bounded by a  Jordan 
 curve.
 
 \begin{figure}[!htpb]
  \setlength{\unitlength}{1mm}
  {\centering
  \includegraphics[width=62mm]{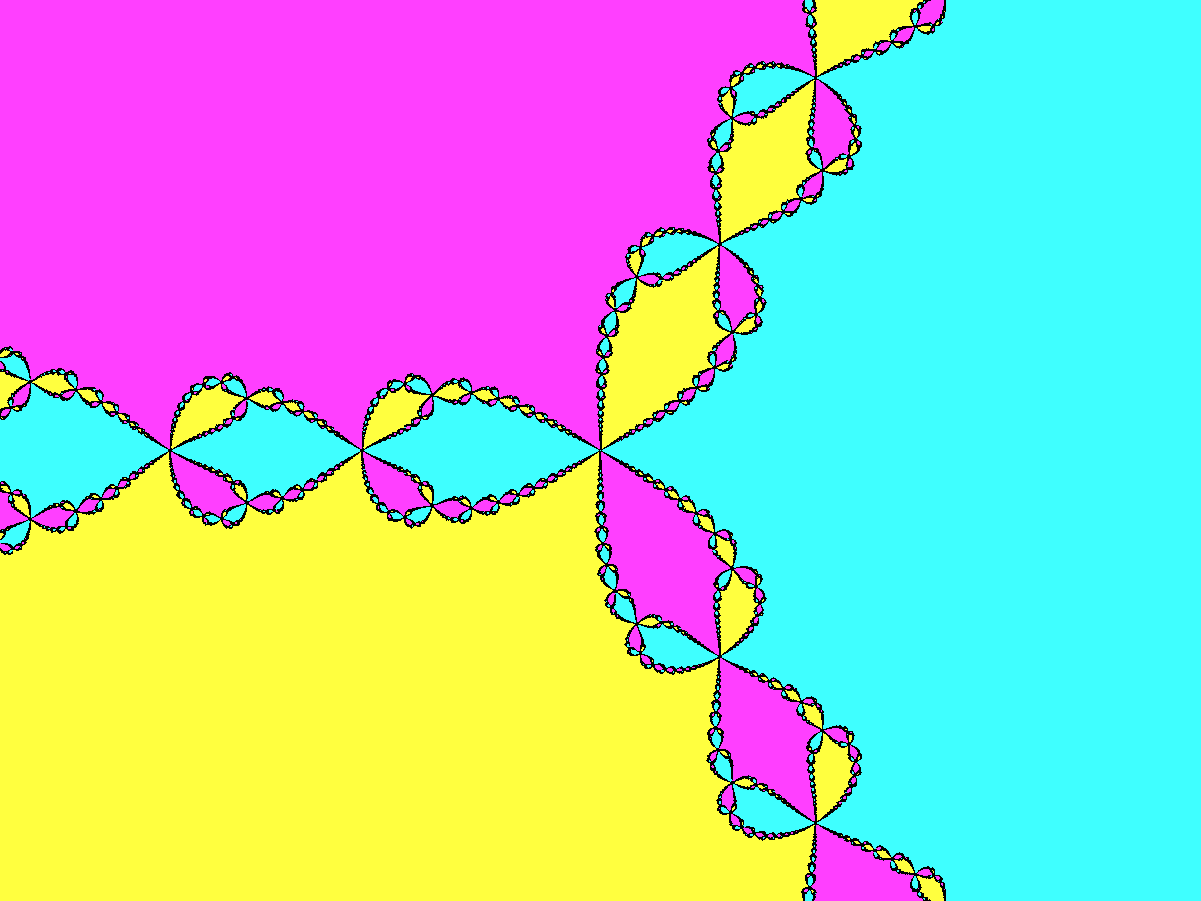}
  \includegraphics[width=62mm]{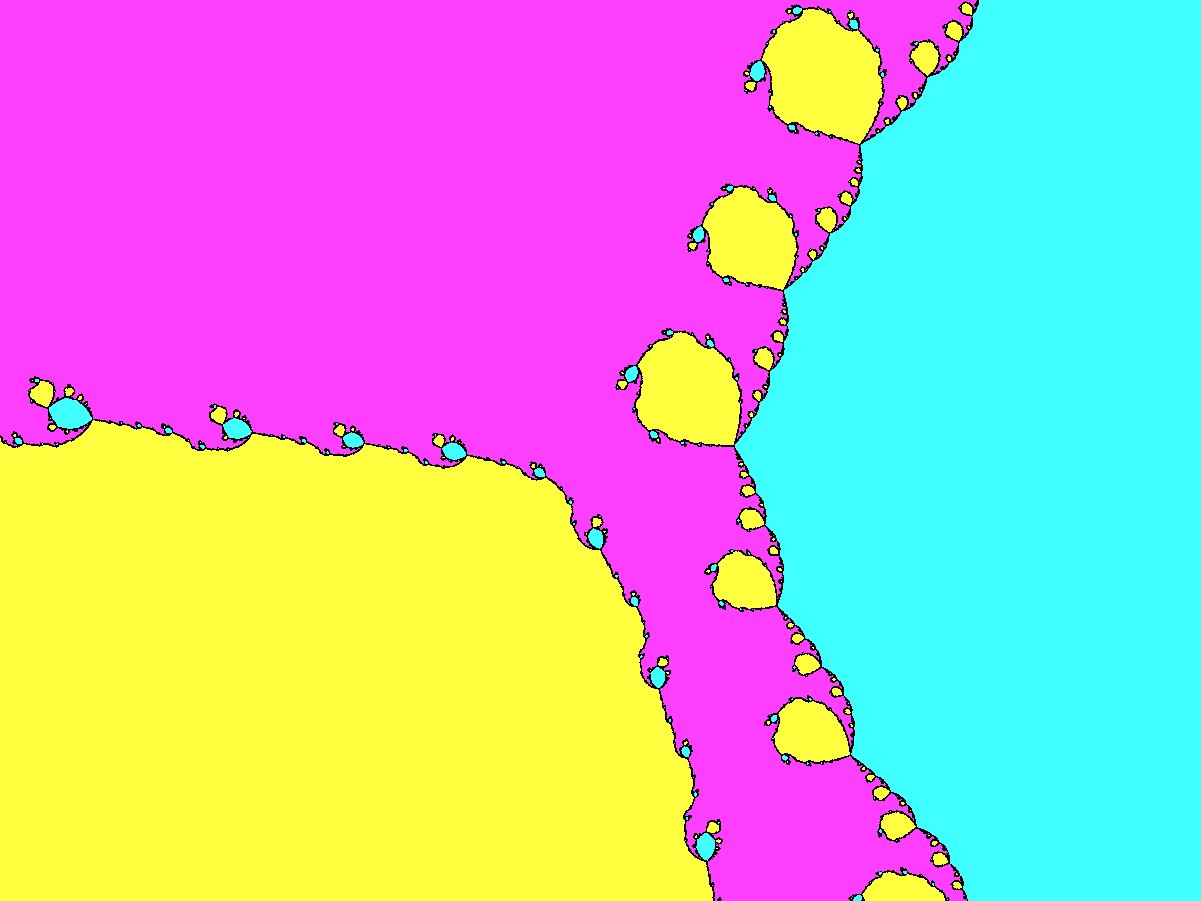}   }
  \caption{Julia sets of $N_\lambda$ with $\lambda=\sqrt{3}i/2$ (left) and $0.05+0.4i  \in \Omega\cap  \mathcal{H}_0^2$ (right). 
  The basins of $b_1(\lambda), b_2(\lambda), b_3(\lambda)$ are colored yellow, purple, cyan, respectively.} 
\end{figure}

%

Define $\Omega_0=\Omega-\mathcal{H}_0^1\cup\mathcal{H}_0^2.$ Let's recall the following dynamical result

\begin{thm}[Roesch \cite{Ro08}]\label{roesch-jordan}  \label{Jordan} For any  $\lambda\in \Omega_0$,  the boundaries of  $B^1_\lambda, B^2_\lambda, B^3_\lambda$ are all Jordan curves. So are their iterated pre-images. 
\end{thm}

\begin{rmk}  Theorem \ref{roesch-jordan} is not true for $\lambda\in \Omega\cap(\mathcal{H}_0^1\cup\mathcal{H}_0^2)$. For example, when $\lambda=0.05+0.4i\in \Omega\cap\mathcal{H}_0^2$, the boundary $\partial B^2_\lambda$ is homeomorphic to the boundary of basin of infinity the cubic polynomial $z^3+\frac{3}{2}z$, which is not a Jordan curve. See Figure 3.
\end{rmk}

\subsection{The (dynamical) internal ray}
As is known in Section \ref{description}, for $\varepsilon\in\{1,2,3\}$ and $\lambda\in\mathcal{X}\setminus\{\pm\frac{1}{2}\}$,  the B\"ottcher map $\phi_\lambda^\varepsilon$ is a conformal map from the neighborhood $U_\lambda^\varepsilon$ of $b_\varepsilon(\lambda)$ to   $\mathbb{D}_{\delta}$, where
\begin{equation*}
\delta=\begin{cases}
 1,\ &\text{ if } 0\notin B_\lambda^\varepsilon,\\
e^{G_\lambda^\varepsilon(0)},\ &\text{ if }  0\in B_\lambda^\varepsilon.
\end{cases}
\end{equation*}

Fix an angle $\theta\in \mathbb{S}=\mathbb{R}/\mathbb{Z}$. We will define the internal ray with angle $\theta$ in various situations, as follows: 

\textit{ Case 1 : $0\notin B_\lambda^\varepsilon$}. In this case,   the internal ray of angle $\theta$, denoted by $R_\lambda^\varepsilon(\theta)$,  is defined as
 $(\phi_\lambda^\varepsilon)^{-1}(\{re^{2\pi i \theta}; 0<r<1\}).$
 
\textit{ Case 2 : $0\in B_\lambda^\varepsilon$ }.   In this case,  the B\"ottcher map
$\phi_\lambda^\varepsilon$ can be extended to  a homeomorphism from $\overline{U_\lambda^\varepsilon}$ to   $\overline{\mathbb{D}_{\delta}}$. Therefore the angle  $\theta_0=\arg\phi_\lambda^\varepsilon(0)$ is well-defined.  Let $\Theta_0=\{\beta\in\mathbb{S}; 2^n\beta=\theta_0  \text{ for some } n\geq0\}$.  For any $\alpha\in \mathbb{S}$, let 
$\ell_0^\lambda(\alpha)=(\phi_\lambda^\varepsilon)^{-1}(\{re^{2\pi i \alpha}; 0<r<\delta\})$ and for $n\geq1$, denote inductively  $\ell_n^\lambda(\theta)$ to be the component of $N_\lambda^{-n}(\ell_0^\lambda(2^n\theta))$ containing $\ell_{n-1}^\lambda(\theta)$. 
Let's now look at  the following set
$$R_\lambda^\varepsilon(\theta)=\bigcup_{k\geq 1}\ell_{k}^\lambda(\theta).$$

There are two possibilities:

 \textit{ Case 2.1:  $\theta\notin\Theta_0$. }  In this case,  all sets $\ell_n^\lambda(\theta)$ avoid the iterated preimages of the free  critical point $0$. Therefore all $\ell_n^\lambda(\theta)$ are analytic curves.  The set $R_\lambda^\varepsilon(\theta)$ is an analytic curve called {\it the internal ray of angle $\theta$}. 
 
  \textit{ Case 2.2 :  $\theta\in\Theta_0$. }  Clearly  $2^n\theta=\theta_0$ for some $n\geq0$ and $\ell_{n+1}^\lambda(\theta)$ is not an analytic  curve. In this case, 
  we say that  the set  $R_\lambda^\varepsilon(\theta)$   {\it bifurcates}.
  

In Case 1 and Case 2.1,  the map $r_\theta:R_\lambda^\varepsilon(\theta)\rightarrow (0,1)$ defined by
$r_\theta(z)=e^{G_\lambda^\varepsilon(z)}$ gives  a natural parameterization  of $R_\lambda^\varepsilon(\theta)$.

The internal rays can also be defined for any Fatou component $U_\lambda$ eventually mapped to  $B_\lambda^\varepsilon$. By pulling back the set $R_\lambda^\varepsilon(\theta)$ in  $B_\lambda^\varepsilon$, one gets the set $R_{U_\lambda}(\theta)$ in  $U_\lambda$. We call  $R_{U_\lambda}(\theta)$  an {\it internal ray of angle $\theta$ in $U_\lambda$}, if its orbit 
$$R_{U_\lambda}(\theta)\mapsto N_\lambda(R_{U_\lambda}(\theta))\mapsto N^2_\lambda(R_{U_\lambda}(\theta))\mapsto\cdots$$
does not meet the free critical point $0$.  When $U_\lambda=B_\lambda^\varepsilon$, we simply write $R_{B_\lambda^\varepsilon}(\theta)$ as $R_\lambda^\varepsilon(\theta)$.


\

In the following, we  concentrate our attention on the maps in $\Omega$. We will give some properties of the internal rays of these maps.


\begin{fact}\label{con} Fix  $\lambda\in \Omega$. Suppose that the internal ray  $R_\lambda^\varepsilon(\theta)$ lands at a repelling periodic point, say $p_{\lambda}$.
Then there is a neighborhood $\mathcal{U}\subset \Omega$ of $\lambda$ satisfying:

1. For  every $u\in \mathcal{U}$, the set $R_u^\varepsilon(\theta)$  does not bifurcate (therefore it is an internal ray),  and it lands at a  repelling periodic point $p_{u}$.

2. The closed ray  $\overline{R_u^\varepsilon}(\theta)$ moves continuously in Hausdorff topology with respect to $u\in \mathcal{U}$.

\end{fact}
\begin{proof}   The idea is to decompose the internal ray into two parts: one near the attracting fixed point and the other near the repelling periodic point.  Each part  moves continuously.
This implies that,  after gluing them  together, the internal ray itself moves continuously.  Here is the detail:

There exist a neighborhood $\mathcal{U}$ of $\lambda$ and a number $\delta\in(0,1)$ such that for all $u\in \mathcal{U}$,  the B\"ottcher map $\phi_u^\varepsilon$  is defined in  a neighborhood $U_u^\varepsilon$ of 
$b_\varepsilon(u)$, and maps $U_u^\varepsilon$  onto $\mathbb{D}_\delta$. We use the notations $\ell_{k}^u(\theta)$ as above. We may shrink $\mathcal{U}$ if necessary so that 

1.  after perturbation in  $\mathcal{U}$,   the $N_\lambda$-repelling periodic  point $p_\lambda$  becomes an  $N_u$-repelling periodic point $p_u$, and 

 2. the  section
 $S_{u}=\overline{\ell_{m+1}^u(\theta)\setminus \ell_{m}^u(\theta)}$ for some large $m$ (independent of $u\in \mathcal{U}$) is contained in a linearized neighborhood $Y_u$ of $p_u$.

We may further assume that $ \mathcal{U}$ is small enough so that  for all   $u\in  \mathcal{U}$,  the set $\bigcup_k\ell_{m+1}^u(2^k\theta)$ avoids the free critical point $0$.  This guarantees that the set $R_u^{\varepsilon}(\theta)$ does not bifurcate. Therefore, it defines an internal ray.
It's clear that $\ell_{m+1}^u(\theta)$ moves continuously with respect to $u\in  \mathcal{U}$.

Note that $\theta$ is periodic under the doubling map $t\mapsto 2t  \ ({\rm mod} \  \mathbb{Z})$. Let $l$ be its period.
In the neighborhood $Y_u$ of $p_u$, the inverse $(N_u^l|_{Y_u})^{-1}$ is contracting.   
This implies that the closure of  the arc
$$T_{u}=\bigcup_{k\geq0}  (N_u^{lk}|_{Y_u})^{-1}(S_{u})$$
moves continuously with respect to $u\in \mathcal{U}$. 

Finally, the continuity of  $u\mapsto\overline{R_u^{\varepsilon}}(\theta)$ follows   from the fact $R_u^{\varepsilon}(\theta)={\ell_{m+1}^u(\theta)\cup T_{u}}$.\end{proof}

An immediate corollary of  Fact \ref{con} is the following

\begin{fact}\label{int-ray}  For all $\lambda\in \Omega$, the set  $R_\lambda^\varepsilon(t)$ with $\varepsilon\in \{1,2,3\}$ and $t\in \{0,1/2\}$  does not bifurcate, and its closure 
$\overline {R_\lambda^\varepsilon}(t)$ moves continuously in Hausdorff topology with respect to the parameter $\lambda\in \Omega$.
\end{fact}
\begin{proof}  Since  $\Omega$ is simply connected,   the argument can be local. 
 Note that  if  we can  show that  the set  $R_\lambda^\varepsilon(0)$ with $\varepsilon\in \{1,2,3\}$  does not bifurcate, then it necessarily lands at  the  repelling fixed point $\infty$ (because the other three fixed points are super-attracting).  
It  will then follow from  Fact \ref{con} that  the closure  $\overline {R_\lambda^\varepsilon}(0)$ (and therefore its preimage  $\overline {R_\lambda^\varepsilon}(1/2)$)  moves continuously in a neighborhood of $\lambda$.
 
 In the following, we show that $R_\lambda^\varepsilon(0)$  does not bifurcate. This is clearly true for $\varepsilon=3$ because $\lambda\in \Omega$, so it suffices to consider the case 
 $\varepsilon=1,2$.
 There are two possibilities:  
 
If $\lambda\in \Omega_0=\Omega\setminus(\mathcal{H}_0^1\cup \mathcal{H}_0^2)$, then the free critical point $0\notin B_\lambda^1\cup B_\lambda^2$. So   $R_\lambda^\varepsilon(t)$ does not bifurcate.

 If $\lambda\in  \Omega\cap\mathcal{H}_0^\varepsilon$ for $\varepsilon\in\{1,2\}$, then $\arg\phi_\lambda^\varepsilon(0)\in (0,1/2)$ or $(1/2,1)$. In this case, neither $0$ nor $1/2$ is in the set $\{\theta\in\mathbb{S}; 2^n\theta=\arg\phi_\lambda^\varepsilon(0) \text{ for some } n\geq0\}$. 
 It follows from the discussion of Case 2.1 in the beginning of this section that   $R_\lambda^\varepsilon(t)$ does not bifurcate.
\end{proof}

By the proof of Fact \ref{int-ray},
for any $\lambda\in\Omega$, the  repelling fixed point $\infty$ of $N_\lambda$ is the common landing point of the internal rays $R_\lambda^1(0),
R_\lambda^2(0), R_\lambda^3(0)$.
Besides itself, the point $\infty$ has two other preimages,  counted with multiplicity.
Therefore there are exactly two rays of  $R_\lambda^1(1/2),
R_\lambda^2(1/2), R_\lambda^3(1/2)$ landing at the same preimage of $\infty$. Here is  a precise statement: 

\begin{fact}\label{order0} For any $\lambda\in \Omega$, 

1.  the rays
$R_\lambda^1(1/2),
R_\lambda^2(1/2)$ land at the same preimage of $\infty$, and

2.   the rays $R_\lambda^1(0), 
R_\lambda^2(0), R_\lambda^3(0)$ land at   $\infty$ in positive cyclic order.
\end{fact}
\begin{proof} 

First note that fix any $\varepsilon\in\{1,2\}$, by Fact \ref{int-ray}, the maps $\lambda\mapsto\overline{R_\lambda^{\varepsilon}}(0)$ and
  $\lambda\mapsto\overline{R_\lambda^{\varepsilon}}(1/2)$ 
 are continuous throughout $\Omega$. 
So it suffices to prove Fact \ref{order0} for some particular $\lambda\in \Omega$.   

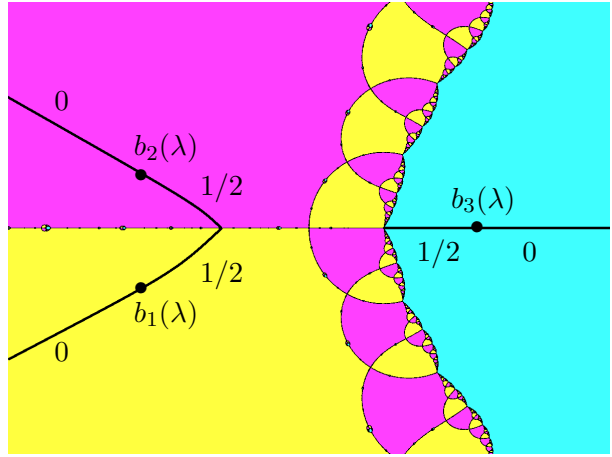
\begin{figure}[h]
\begin{center}
\includegraphics[height=6cm]{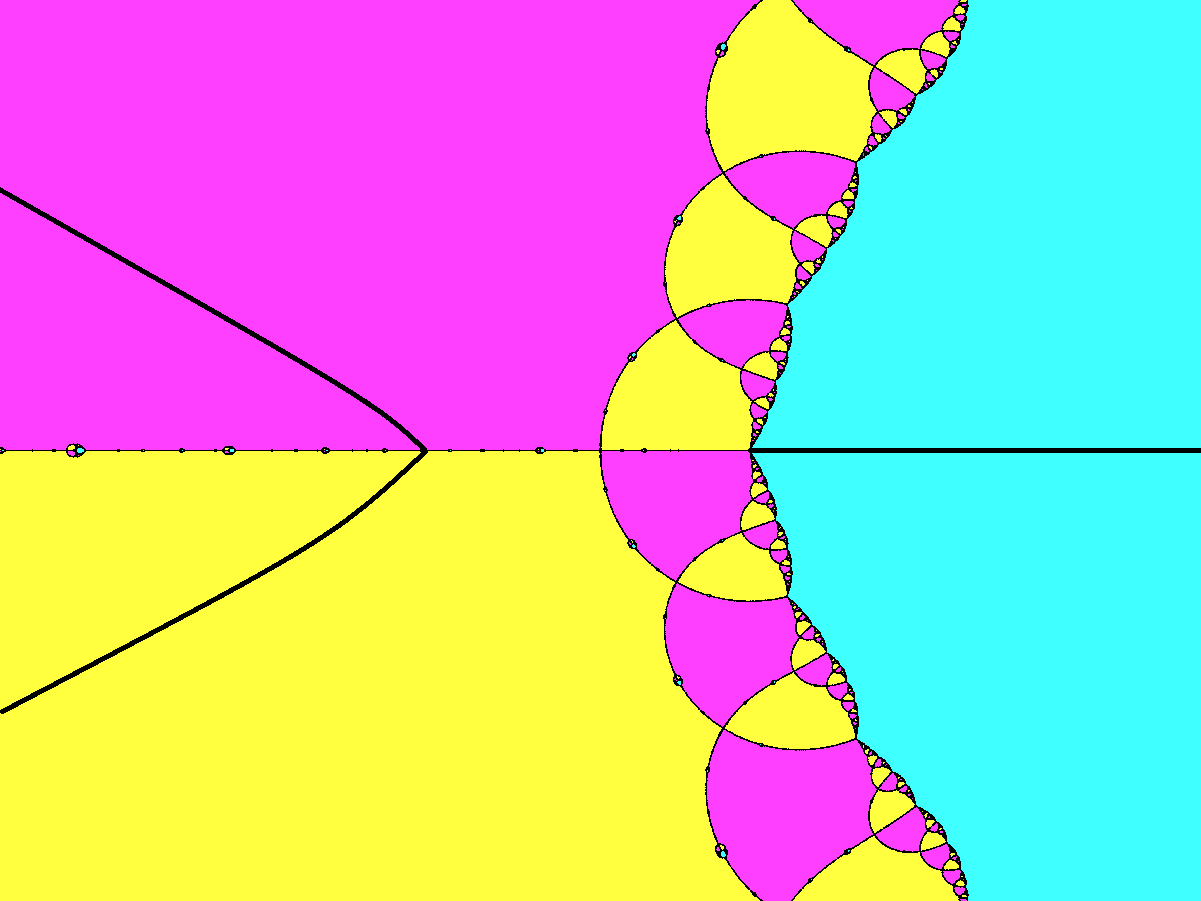}
\put(-210,35){$0$}\put(-210,130){$0$}     \put(-155,98){$1/2$}   \put(-155,65){$1/2$}   \put(-180,103){$\bullet$}   \put(-180,60){$\bullet$}  \put(-180,50){$b_1({\lambda})$}
\put(-180,113){$b_2(\lambda)$}    \put(-60,93){$b_3(\lambda)$}    \put(-53,83){$\bullet$}   \put(-73,73){$1/2$}   \put(-33,73){$0$}   
 \caption{Julia set of   $N_{\lambda}$ with $\lambda=0.1i$. The  internal angles $0, 1/2$ are labeled beside the corresponding internal rays.}
\end{center}\label{f5}
\end{figure}

Let's look at the case  
 $\lambda=i \epsilon$ for some small number $\epsilon>0$. In this case,  the map $N_\lambda$ is a real rational function and  satisfies  $N_{\lambda}(z)=N_{\bar{\lambda}}(z)=
\overline{N_\lambda(\overline{z})}$. So the B\"ottcher maps satisfy 
\bess\overline{\phi_{\lambda}^1({\overline{z}})}&=&\lim_{k\rightarrow\infty}(N_{\overline{\lambda}}^k(\overline{\overline{z}})-\overline{b_1(\lambda)})^{2^{-k}}\\
&=&\lim_{k\rightarrow\infty}(N_{{\lambda}}^k({z})-b_2(\lambda))^{2^{-k}}\\
&=&\phi_{\lambda}^2({z})
\eess
in a neighborhood of the super-attracting point $b_2(\lambda)$, this implies that $R_\lambda^1(1/2)$ and $R_\lambda^2(1/2)$ are symmetric about the real axis.
On the other hand, it follows from the fact $\overline{\phi_{\lambda}^3({\overline{z}})}=\phi_{\lambda}^3(z)$ that each  of $R_\lambda^3(0), R_\lambda^3(1/2)$ is 
symmetric about the real axis. Therefore,
$$R_\lambda^3(0)=(1,+\infty), \   R_\lambda^3(1/2)\subset (-\infty,1).$$
Note that $N^{-1}_\lambda(\infty)=\{\infty, \pm \sqrt{\frac{3+4\lambda^2}{12}}\}$. So $R_\lambda^3(1/2)=(\sqrt{\frac{3+4\lambda^2}{12}},1)$, and $R_\lambda^1(1/2),  R_\lambda^2(1/2)$ land at the same point $-\sqrt{\frac{3+4\lambda^2}{12}}$.   See Figure 4.

Note that $R_\lambda^1(0), R_\lambda^2(0)$ are also symmetric about the real axis, stemming from $b_1(\lambda)=-i\epsilon-1/2,  b_2(\lambda)=i\epsilon-1/2$, respectively.
Therefore,   $R_\lambda^1(0)$ lies in the lower half plane and $R_\lambda^2(0)$ is  in the upper half plane, implying that the rays   $R_\lambda^1(0),
R_\lambda^2(0), R_\lambda^3(0)$ land at   $\infty$ in positive cyclic order (also called counter clockwise order).
\end{proof}

\subsection{The parameter ray}

We shall prove the following basic properties of the parameter rays:

\begin{lem} [Rational rays land] \label{r-ray-land}  Let $t\in(0,1/2)\cap \mathbb{Q}$.

1. The parameter ray $\mathcal{R}^1_0(t)${\rm(}or $\mathcal{R}^2_0(1-t)${\rm)} converges to a parameter $\lambda_0$.

2. For this $\lambda_0$, in the dynamical plane of $N_{\lambda_0}$, the internal ray $R_{\lambda_0}^1(t)${\rm(}or $R_{\lambda_0}^2(1-t)${\rm)} converges to a pre-periodic point $x_0$, either pre-repelling or pre-parabolic.  In the former case, $x_0=0$.  
 \end{lem}
 
 \begin{proof}  Let $\lambda_0$ be an accumulation point of $\mathcal{R}^1_0(t)$. That is, there is a sequence of parameters  $\{\lambda_j\}_{j\geq 1}$ on $\mathcal{R}^1_0(t)$ with $\lambda_j\rightarrow \lambda_0$. Since $t$ is rational, there are two integers $m\geq0, k>0$ such that $N_{\lambda_0}^{k}\circ N_{\lambda_0}^{m}(R_{\lambda_0}^1(t))=N_{\lambda_0}^{m}(R_{\lambda_0}^1(t))$. By the Snail Lemma \cite[Lemma 16.2]{M}, the landing point $x_0$ of  internal ray $R_{\lambda_0}^1(t)$ 
 is  either  (pre-)parabolic satisfying that $(N_{\lambda_0}^{k})'(N_{\lambda_0}^{m}(x_0))=1$,   or (pre-)repelling.
 In the latter case, by Fact \ref{con},  we have that 
 
(1). $t$ is strictly preperiodic and $m>0$.  (If not, then $x_0$ is a repelling periodic point. By Fact \ref{con}, there is a neighborhood $\mathcal{U}$ of $\lambda_0$ such that all $R_{\lambda}^1(t)$ with  $\lambda\in \mathcal{U}$ are internal rays. But this is impossible for those $\lambda_j\in \mathcal{U}$).

(2). when $j$ is large, the closures of the internal rays $\overline{R_{\lambda_j}^1}(2^mt)$ converge to $\overline{R_{\lambda_0}^1}(2^mt)$ in Hausdorff topology as $j\rightarrow\infty$.

 Since $N_{\lambda_j}^m(0)\in R_{\lambda_j}^1(2^mt)$ for all $j$, we necessarily have $N_{\lambda_0}^m(0)\in\overline{R_{\lambda_0}^1}(2^mt)$. It follows that $R_{\lambda_0}^1(2^mt)$ lands at $N_{\lambda_0}^m(0)$. By Theorem \ref{roesch-jordan}, the internal ray $R_{\lambda_0}^1(t)$ lands at the critical point $0$. This proves the second statement.

 The proof of the first statement goes as follows:  note that the system 
 \bess &N_{\lambda_0}^{k}\circ N_{\lambda_0}^{m}(x_0)=N_{\lambda_0}^{m}(x_0), \ 
  (N_{\lambda_0}^{k})'(N_{\lambda_0}^{m}(x_0))=1 \    \text{(parabolic case)}&\\
  &N_{\lambda_0}^{k}\circ N_{\lambda_0}^{m}(0)=N_{\lambda_0}^{m}(0) \    \text{(repelling case)}&
  \eess
  has  finitely many solutions of $\lambda_0$'s. Therefore the accumulation set of $\mathcal{R}^1_0(t)$, known  as 
  a connected subset of $\partial \mathcal{H}_0^1$, necessarily consists of finitely many points. 
  This implies that the accumulation set of $\mathcal{R}^1_0(t)$ is a singleton, showing the convergence of  $\mathcal{R}^1_0(t)$.
   \end{proof}

\begin{lem} \label{pre-rigidity} Let $\lambda_1, \lambda_2\in\Omega_0$ and $t=p/2^k\in (0,1/2)$. Assume that

(1) the  internal ray   $R_{\lambda_1}^1(t)$ lands at the critical point $0$;

(2) the parameter ray $\mathcal{R}_0^1(t)$ lands at $\lambda_2$.

 Then we have $\lambda_1=\lambda_2$.
 \end{lem}

Lemma \ref{pre-rigidity} is useful in the proof of Theorem \ref{H-tongue1}. Its stronger version, Theorem \ref{rigidity1}, will be proven   after we introduce the Yoccoz puzzle theory.

 \begin{proof} By (the proof of) Lemma \ref{r-ray-land}, we know that the  internal ray   $R_{\lambda_2}^1(t)$ also lands at the critical point $0$.   So the maps $N_{\lambda_1}$ and $N_{\lambda_2}$ both are post-critically finite. 
   Define a homeomorphism $\psi: \mathbb{\widehat{C}}\rightarrow \mathbb{\widehat{C}}$
 satisfying the following two properties:
 
$\bullet$  $\psi|_{\overline{B_{\lambda_1}^1}}=(\phi_{\lambda_2}^{1})^{-1}\circ\phi_{\lambda_1}^{1}$ (note that the right side can be extended to $\overline{B^1_{\lambda_1}}$ because  
 $\partial {B}_{\lambda_1}^1$  and $\partial {B}_{\lambda_2}^1$ are Jordan curves, see Theorem \ref{roesch-jordan}).
 
$\bullet$   $\psi|_{\overline{R^\varepsilon_{\lambda_1}}(0)}=(\phi_{\lambda_2}^{\varepsilon})^{-1}\circ\phi_{\lambda_1}^{\varepsilon}|_{\overline{R^\varepsilon_{\lambda_1}}(0)}, \  \varepsilon=2,3$. 
  
  Then there is a homeomorphism $\phi: \mathbb{\widehat{C}}\rightarrow \mathbb{\widehat{C}}$ satisfying that 
  $$\psi \circ N_{\lambda_1}=N_{\lambda_2}\circ \phi, \  \psi|_{\overline{B_{\lambda_1}^1}\cup \overline{R^2_{\lambda_1}}(0)\cup \overline{R^3_{\lambda_1}}(0)}= \phi|_{\overline{B_{\lambda_1}^1}\cup \overline{R^2_{\lambda_1}}(0)\cup \overline{R^3_{\lambda_1}}(0)}.$$
  Note that  $\mathbb{\widehat{C}}\setminus (\overline{B_{\lambda_1}^1}\cup \overline{R^2_{\lambda_1}}(0)\cup \overline{R^3_{\lambda_1}}(0))$ is simply connected,  
  $\psi$ can be deformed continuously to $\phi$ in  $\mathbb{\widehat{C}}\setminus (\overline{B_{\lambda_1}^1}\cup \overline{R^2_{\lambda_1}}(0)\cup \overline{R^3_{\lambda_1}}(0))$. This, in particular,  means that $\psi$ and $\phi$ are homotopic rel the post-critical set $P(N_{\lambda_1}) \subset \overline{B_{\lambda_1}^1}\cup \overline{R^2_{\lambda_1}}(0)\cup \overline{R^3_{\lambda_1}}(0)$.  In other words, $N_{\lambda_1}$ and $N_{\lambda_2}$ are combinatorially equivalent in the sense of Thurston. By Thurston's Theorem \cite{DH2},   $N_{\lambda_1}$ and $N_{\lambda_2}$ are M\"obius conjugate.  We necessarily have $\lambda_1=\lambda_2$ because they are both in the fundamental domain $\mathcal{X}_{FD}$.
 \end{proof}

\section{The Head's angle}\label{h-t}

In this section, we will introduce a very important combinatorial number for the cubic Newton maps: {\it the Head's angle}.  This number   characterizes completely how and where the two adjacent immediate  basins  of the super-attracting fixed points  touch (when $\lambda \in \Omega_0$, these two adjacent immediate  basins are exactly $B_\lambda^1$ and $B_\lambda^2$).   

The main part of this section,  Section \ref{head-tongue}, is to  characterize all the   maps sharing the same   Head's angle $\theta$, in the case that $\theta$ is periodic or dyadic. 
Properties of the Head's angles are also given.

\subsection{The Head's angle}\label{head}
 Fix any  $\lambda\in \Omega_0\cup\{\sqrt{3}i/2\}$ and any $t\in \mathbb{S}$, 
  the sets $R_\lambda^1(t)$ and $R_\lambda^2(t)$   are  internal rays.  
 We define the following set of
internal angles:
 \bess &\Theta_\lambda=\{t\in (0,1]; \text{ the rays } R_\lambda^1(t)  \text{ and } R_\lambda^2(1-t)  \text{  land at
the same point}\}.&\eess

The set $\Theta_\lambda$ plays an important role in exploring the dynamical behavior  of $N_\lambda$. 
So it is worth listing some  properties of $\Theta_\lambda$ :

\vspace{5 pt}

 (1). {\it For all $\lambda\in \Omega_0\cup\{\sqrt{3}i/2\}$, the set $\Theta_\lambda$
is closed, forward invariant under the doubling map $\tau: t\mapsto 2t   ({\rm mod } \ \mathbb{Z})$,
 without interior and with $1$ as an accumulation point.}
 
  For a proof of this fact, 
  see  \cite[Lemma 2.12]{Tan97};
 
\vspace{5 pt}

(2).  {\it  For all $\lambda\in \Omega_0$, the set $\Theta_\lambda$ contains $\{1, 1-{1}/{2^n}; n\geq 1\}$.}

 See \cite[Corollary 3.7 and  Corollary 7.11]{Ro08}, we remark that this fact also holds when $\lambda=\sqrt{3}i/2$. These angles are the ones where the internal ray lands at $\infty$ or its iterated preimages.

\vspace{5 pt}

(3).  {\it  For any $\lambda\in \Omega_0$, there is a large integer $\ell=\ell(\lambda)>0$ such that when $p\geq \ell$, one has $1-\frac{1}{2^p-1}\in \Theta_\lambda$.}

 See \cite[Corollary 3.16]{Ro08}. We remark that it suffices to choose $\ell$ so that 
 $$\min\Big\{2^k(1-\frac{1}{2^\ell-1}) \  {\rm mod } \ \mathbb{Z}; k\geq 0\Big\}=\frac{1}{2}-\frac{1}{2(2^\ell-1)}>\boldsymbol{h}(\lambda),$$ where $\boldsymbol{h}(\lambda)$ is the Head's angle of $N_\lambda$ defined below.

\vspace{5 pt}

Even if we can't  write down explicitly  all  the elements  of  $\Theta_\lambda$,   we may associate the set  $\Theta_\lambda$ with a real number, which will uniquely determine  $\Theta_\lambda$.
This number is called  {\it the Head's angle}. 

\begin{figure}[h]
\begin{center}
\includegraphics[height=6cm]{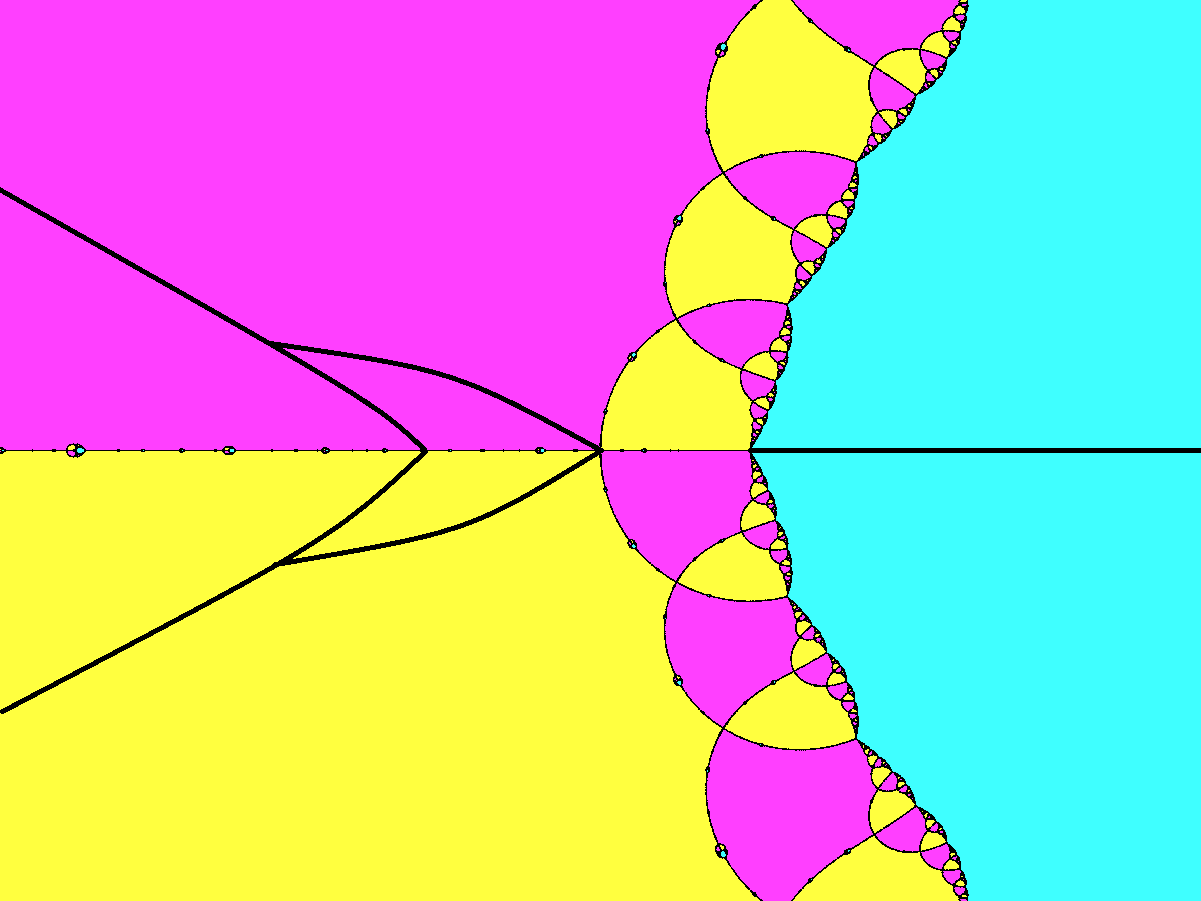}
\put(-210,35){$0$}   \put(-210,130){$0(=1)$}      \put(-172,90){$\frac{1}{2}$}   \put(-172,72){$\frac{1}{2}$}  \put(-177,102){$\bullet$}   \put(-177,62){$\bullet$}  \put(-180,53){$b_1({\lambda})$}
 \put(-140,60){$\alpha$}  \put(-140,103){$1-\alpha$}
\put(-180,110){$b_2(\lambda)$}    \put(-60,93){$b_3(\lambda)$}    \put(-53,83){$\bullet$}   
 \caption{Julia set of   $N_{\lambda}$ with $\lambda=0.1i$. The  Head's angle is $\alpha$.}
\end{center}\label{f5}
\end{figure}
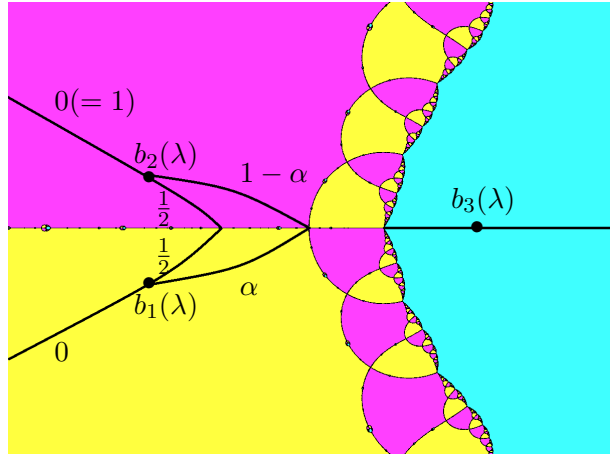




\begin{defi} For $\lambda\in\Omega_0\cup\{\sqrt{3}i/2\}$,  the Head's angle $\boldsymbol{h}(\lambda)$ of $N_\lambda$ is defined  as the infimum of the angles in  $\Theta_\lambda$ (as a subset of $(0,1]$).
\end{defi}


\begin{fact}\label{head-special} When  $\lambda=\sqrt{3}i/2$, the Head's angle $\boldsymbol{h}(\lambda)=1/2$.
\end{fact}

 \begin{proof}
 In this case, $N_\lambda(0)=\infty$. In the dynamical plane, it is known from Fact \ref{order0} that the internal rays $R_\lambda^1(1/2)$ and $R_\lambda^2(1/2)$ land at the  critical point $0$.  This implies that $\boldsymbol{h}(\sqrt{3}i/2)\leq1/2$.  One should also observe by the local behavior of $N_\lambda$ near the critical point $0$  that $R_\lambda^1(1/2),   T_\lambda^2, T_\lambda^1, R_\lambda^2(1/2)$ attach  $0$ in positive cyclic order, where $T_\lambda^\varepsilon$ is the component of $N_\lambda^{-1}(B_\lambda^\varepsilon)$ other that $B_\lambda^\varepsilon$. 
 To prove that $\boldsymbol{h}(\lambda)=1/2$, it suffices to show

{\it Claim:  For any $\theta\in[1/2,1)\cap \Theta_\lambda$,  we have $\theta/2\notin\Theta_\lambda$.}
 
 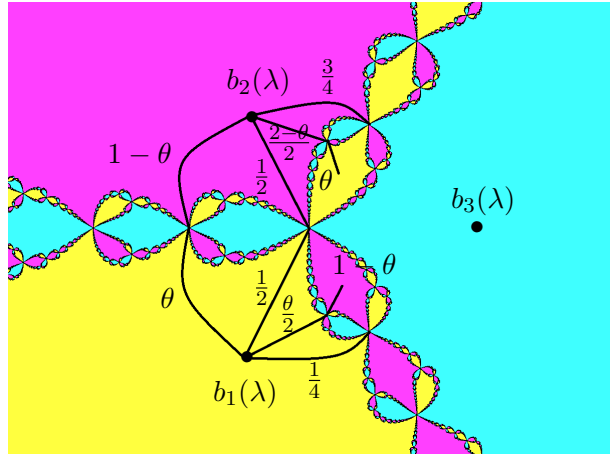
\begin{figure}[h]
\begin{center}
\includegraphics[height=6cm]{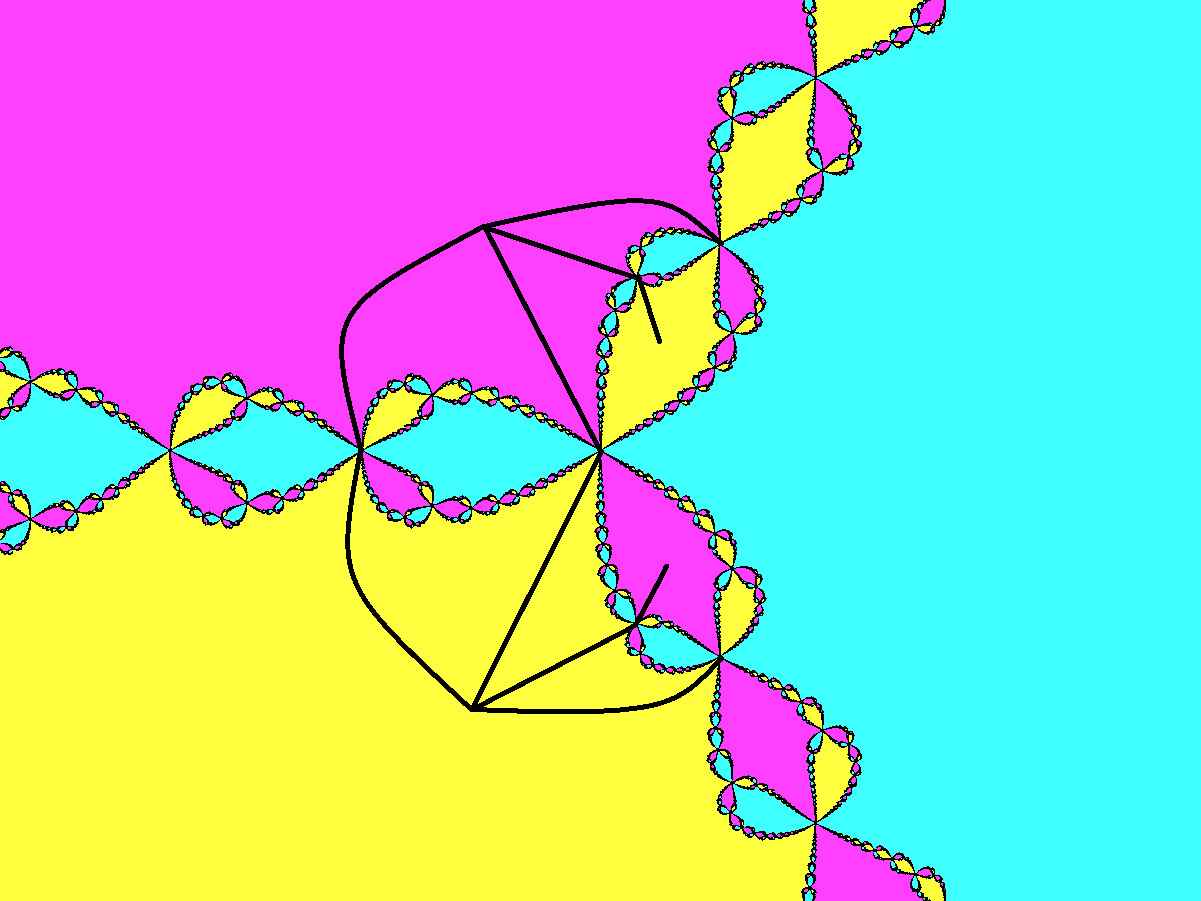}
   \put(-190,110){$1-\theta$}    \put(-170,55){$\theta$}   \put(-135,102){$\frac{1}{2}$}     \put(-130,115){$\frac{2-\theta}{2}$}   \put(-110,100){$\theta$}  \put(-105,68){$1-\theta$} 
     \put(-110,140){$\frac{3}{4}$}  \put(-135,60){$\frac{1}{2}$}   
    \put(-125,50){$\frac{\theta}{2}$}     \put(-115,25){$\frac{1}{4}$} 
   \put(-138,125){$\bullet$}   \put(-140,34){$\bullet$}  \put(-150,20){$b_1({\lambda})$}
\put(-145,138){$b_2(\lambda)$}    \put(-60,93){$b_3(\lambda)$}    \put(-53,83){$\bullet$}   
 \caption{Julia set of   $N_{\lambda}$ with $\lambda=\sqrt{3}i/2$. The  internal angles  are labeled beside the corresponding internal rays.}
\end{center}\label{f5}
\end{figure}

 The claim implies that $(0,1/2)\cap \Theta_\lambda=\emptyset$. In what follows, we prove this claim.   By looking at the $N_\lambda$-preimage of the 
  components of  $$\mathbb{\widehat{C}}\setminus \Big( \overline{{R}_\lambda^1}(0)\cup  \overline{{R}_\lambda^1}(\theta)\cup \overline{{R}_\lambda^2}(0)\cup  \overline{{R}_\lambda^2}(1-\theta)\Big),$$ 
  we see that $R_\lambda^1(\theta/2)$  attaches the internal ray in 
 $ T_\lambda^2$ of angle $1-\theta$, and $R_\lambda^2(1-\theta/2)$  attaches the internal ray in 
 $ T_\lambda^1$ of angle $\theta$.  Therefore the internal rays $R_\lambda^1(\theta/2)$ and $R_\lambda^2(1-\theta/2)$ can not land at the same point. This means $\theta/2\notin\Theta_\lambda$, completing the proof.
  \end{proof}

 
  

\begin{pro} \label{head-property}  For any $\lambda\in\Omega_0$,  the Head's angle $\boldsymbol{h}(\lambda)$ satisfies:

1. $0<\boldsymbol{h}(\lambda)<\frac{1}{2}$ and  $\boldsymbol{h}(\lambda)\in \Theta_\lambda$.

2. $\theta\in \Theta_\lambda$ if and only if $2^n\theta\in[\boldsymbol{h}(\lambda),1]$ for all $n\geq0$. This implies that  $\Theta_\lambda=\bigcap_{k\geq0}\tau^{-k}([\boldsymbol{h}(\lambda),1])$ is totally disconnected .

\end{pro}
\begin{proof} For the  proof, see \cite[Lemma 3.15 and  Corollary 7.11]{Ro08}. \end{proof}






\subsection{The set $\Xi$}  \label{cand-set}


It's known from Proposition \ref{head-property} that for any $\lambda\in \Omega_0$, the Head's  angle $\boldsymbol{h}(\lambda)$ of $N_\lambda$ is contained in the set
$\Xi$,  here we recall the definition
 $$\Xi=\{\theta\in (0,1/2] ; 2^k\theta\in [\theta,1] \   ({\rm mod}  \ \mathbb{Z})  \text{ for all } k\geq0\}.$$
 It's clear that $1/2^k \in \Xi$, for all $k\geq 1$.
 
 

\begin{lem}[Tan Lei \cite{Tan97}] \label{char-S}  The set  $\Xi$ satisfies the following properties:

1. The set $\Xi\cup\{0\}$  is closed, totally disconnected, perfect (without isolated points)
and of $0$-Lebesgue measure. 

2. Let  $(\beta,\alpha)$  be a connected component of 
$(0,1/2)\setminus \Xi$, then   $\alpha$ and  $\beta$  are of the form 
$$ \alpha = {p\over2^n-1} \quad\text{ and }\quad
   \beta = {p\over2^n},  \quad\text{ with }\quad n\geq 2.$$
   
3. If $\beta=p/2^m$ {\rm (}resp. $\alpha=p/(2^n-1)${\rm)} is in~$\Xi$, it is necessarily the infimum 
 {\rm (}resp. supremum{\rm )} of a connected component $(\beta,\overline\alpha)$ {\rm(}resp. $(\overline\beta,\alpha)${\rm)} of
$(0,1/2)\setminus \Xi$, where $\overline\alpha=p/(2^m-1)$ {\rm(}resp. $\overline\beta=p/2^n${\rm)}.   

\end{lem}
%
%
%
%
\proof
Points 1) and~3) correspond to   Proposition~2.16  in \cite{Tan97}. Point 2)
is exactly  Lemma~2.24 in~\cite{Tan97}. 
\qed

%


\vspace{5pt}

Recall that the sets $\Theta_{per}, \Theta_{dya}$ in Section \ref{int} are defined as follows
\bess
&\Theta_{per}=\{t\in (0,1/2]; 2^kt=t \ ({\rm mod}  \ \mathbb{Z})  \text{ for some }  k\geq1 \}, & \\
&\Theta_{dya}=\{t\in (0,1/2]; 2^kt=1 \ ({\rm mod} \  \mathbb{Z})  \text{ for some }  k\geq1 \}. &
\eess
The following fact is easy to verify 


\begin{fact} \label{angle-Xi}  Let $\partial(\Xi)$ be the collection of  the endpoints of all  interval  components of $(0,1/2)\setminus\Xi$, together with $1/2$. Then
$$ \partial(\Xi)=(\Xi\cap\Theta_{per})\cup(\Xi\cap\Theta_{dya}).$$
\end{fact}


By definition and Proposition \ref{head-property}, we know that the Head's angle $\boldsymbol{h}(\lambda)$ is the minimum of the angles in  the set $\Theta_\lambda$. The following result will tell us in which  case   $\boldsymbol{h}(\lambda)$ is an accumulation point of  $\Theta_\lambda$.

\begin{pro} \label{density} 
Let $\alpha\in \Xi$, we define a set $C_\alpha$ by
$$C_\alpha=\bigcap_{k\geq0}\tau^{-k}([\alpha,1]).$$

  1. If  $\alpha$ is dyadic, then  $\alpha$ is an isolated point of  $C_\alpha$;

2.   If  $\alpha$ is non-dyadic, then  $\alpha$ is an accumulation point of  the non-dyadic (or  dyadic) rational  numbers in $C_\alpha$. 
\end{pro}
\begin{proof}   If   $\alpha$ is dyadic, then there is a smallest integer $k\geq 1$ such that $2^k \alpha=1$. Take $\delta=\alpha/2^{k+1}$,  we claim that $(\alpha, \alpha+\delta)\cap C_\alpha=\emptyset$. This is because $2^{k}(\alpha, \alpha+\delta)= (0, \alpha/2) ({\rm mod } \ \mathbb{Z})$ which has no intersection with $[\alpha,1]$. Therefore $\alpha$ is an isolated point of  $C_\alpha$.


To prove the second statement, note that $[\alpha, 1)\cap \Xi\subset C_\alpha$. Because $\alpha$ is non-dyadic and by Lemma  \ref{char-S}, we can find a decreasing sequence of rationals $\{r_n\}\subset (\Xi\cap\Theta_{per})\cap C_\alpha$ (or $(\Xi\cap\Theta_{dya}) \cap C_\alpha$) with limit $\alpha$, where each $r_n$ takes the form $\frac{p_n}{2^{q_n}-1}$ (or $\frac{p_n}{2^{q_n}}$). This completes the proof.
\end{proof}

\subsection{Maps with the same Head's angle}  \label{head-tongue}
There are two natural questions: 

{\it Q1: Which angles $\theta\in\Xi$ arise  as Head's angles for cubic Newton maps?}

{\it Q2: Which cubic Newton maps have the same Head's angle?}

These questions both have complete answers, see Theorems \ref{main3} and \ref{main2}, whose proofs are given in Section \ref{final-s}.
But for the moment, with what we have known in hand, we only consider (special cases of) the second question Q2. 
For some special cases of  $\theta\in\Xi$,   Tan Lei posed the following conjecture \cite[p.231 Conjecture]{Tan97}:
 
\begin{conj} Given any component $(\beta,\alpha)$ of  $(0,1/2)\setminus \Xi$,  there is a unique connected component of
 $\Omega-\overline{\mathcal{H}_0^1\cup \mathcal{H}_0^2}$ such that
the maps in the closure of this component are precisely the maps with Head's
angle $\beta$ or $\alpha$. 
\end{conj}

To confirm this conjecture (a tiny mistake of the conjecture is corrected in Corollary \ref{H-tongue2}), we establish the following crucial result. 

\begin{thm}\label{H-tongue1} Let $(\beta,\alpha)$ be a  connected component of $(0,1/2)\setminus \Xi$, and let $\theta=\beta$ or $\alpha$.

1. The parameter rays $\mathcal{R}_0^1(\theta)$ and $\mathcal{R}_0^2(1-\theta)$ land at the same point, say $\lambda_\theta$.

2. Let $\mathcal{U}_\theta$ be the bounded component of
 $$\mathbb{C}\setminus \Big(\Big[-\frac{1}{2},\frac{1}{2}\Big]\cup \overline{\mathcal{R}_0^1}(\theta)\cup \overline{\mathcal{R}_0^2}(1-\theta)\Big),$$ then $\mathcal{U}_\theta\cup \{\lambda_\theta\}$
 consists of the parameters $\lambda\in \Omega$ for which, in the dynamical plane, the internal rays ${R}_\lambda^1(\theta)$ and ${R}_\lambda^2(1-\theta)$
 land at the same point $x_\lambda$. Moreover, this  point $x_\lambda$ is 
 \begin{itemize}
      \item 0, if $\lambda=\lambda_\beta$;
     \item a parabolic periodic point, if $\lambda=\lambda_\alpha$;
     \item pre-repelling point,  whose orbit does not contain  the critical point $0$, if $\lambda\in \mathcal{U}_\theta$. 
     \end{itemize}
 In particular, $\lambda_\alpha\ne\lambda_\beta$.
\end{thm}

See Figure 7 and Figure 8.

 \begin{figure}[!htpb]
  \setlength{\unitlength}{1mm}
  {\centering
  \includegraphics[width=62mm]{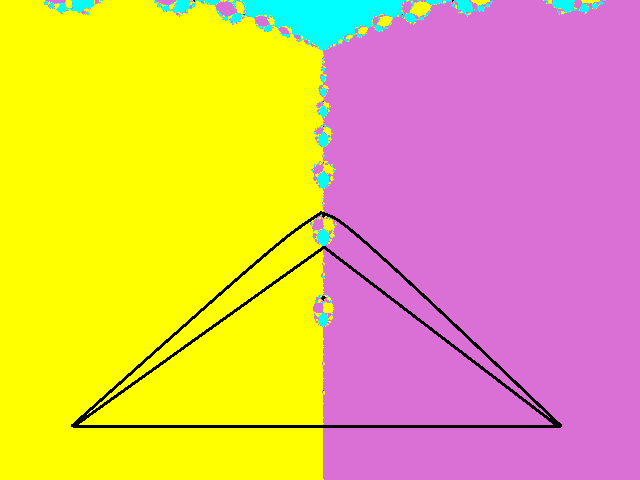}
  \includegraphics[width=62mm]{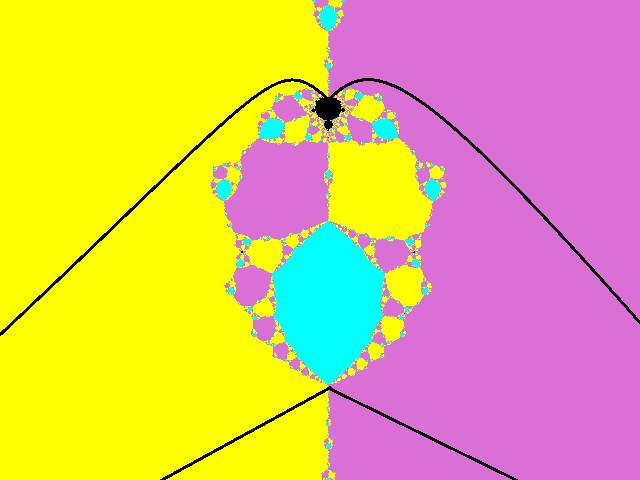}
  \put(-47,33){$\alpha$} \put(-47,2){$\beta$}     \put(-32,40){$\lambda_\alpha$} \put(-32,12){$\lambda_\beta$} 
  \put(-105,11){$\beta$}    \put(-107,21){$\alpha$}    \put(-92,10){$1-\beta$}    \put(-85,20){$1-\alpha$}
  \put(-95,27){$\lambda_\alpha$} \put(-95,18){$\lambda_\beta$} 
}
  \caption{Parameter rays landing at the same point (left) and the local picture (right).
  } 
\end{figure}

 \begin{figure}[!htpb]
  \setlength{\unitlength}{1mm}
  {\centering
  \includegraphics[width=62mm]{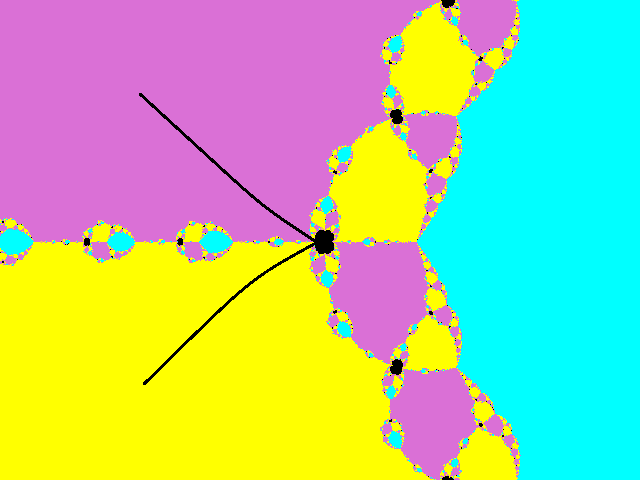}
  \includegraphics[width=62mm]{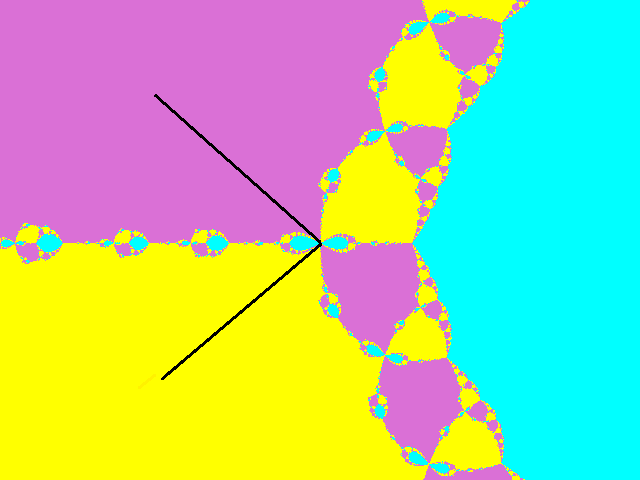}
  \put(-32,22){$\bullet \ 0$}   
  \put(-40,12){$\beta$}
      \put(-102,15){$\alpha$}   
 }
  \caption{Left:  the dynamical plane of $N_{\lambda_\alpha}$, the  Head's angle   is  $\alpha$ and the ray $R_{\lambda_\alpha}^1(\alpha)$ lands at a parabolic point. Right:  the dynamical plane of $N_{\lambda_\beta}$, the  Head's angle   is  $\beta$ and the ray $R_{\lambda_\beta}^1(\beta)$ lands at the critical point $0$.
   } 
\end{figure}

We remark  that the theorem is also true for $\theta=1/2$, which is not included in the statement. Before the proof, we need  a lemma:

\begin{lem}\label{exist-beta}
Let $n\ge1$ and $\theta=p/2^n\in \Xi$. Then there  exists  a 
parameter~$\lambda$ in~$\Omega_0$ with $\boldsymbol{h}(\lambda)=\theta$. 
\end{lem}

\proof
For $n=1$, it suffices  to take the map $N_\lambda$ with 
$\lambda=i\sqrt3/2$ (see Fact \ref{head-special}). For $n>1$, let $\lambda$  be the landing  point of the  ray
$\mathcal{R}_0^1(p/2^{n})$. Since the ray $R^1_\lambda(p/2^{n})$ lands at  a
 pre-repelling  point, it lands in fact at the 
critical point $0$ (see the proof of Lemma~\ref{r-ray-land}). 

Let  $w_2=w_2(\lambda)$ denote the first  pre-image of $b_2(\lambda)$, $U_\lambda$ be the Fatou component containing $w_2$. 
We denote by
$\gamma$ the largest $t\in(0,1/2)$ such that the rays
$R^1_\lambda(t)$ and $(N_\lambda|_{U_\lambda})^{-1}(R^2_\lambda(1-2t))$ land at the same
point. It is clear that $2\gamma \in \Theta_\lambda$. 
By the location of the critical point $0$ \cite[Lemma 3.14]{Ro08}, we have   $\gamma \le \theta \le \boldsymbol{h}(\lambda)$.
Let's denote by $\theta'$  the last iteration of   $\theta$
in $[\gamma,\boldsymbol{h}(\lambda)]$ (it can happen that $\gamma=\boldsymbol{h}(\lambda)$). Since $\tau^j(\theta')\ge\boldsymbol{h}(\lambda)$ for all $j\ge1$, \ 
$\tau(\theta')\in C_{\boldsymbol{h}(\lambda)}=\Theta_\lambda$. 
There are three possibilities:


%
%

Case 1: $\theta'\in(\gamma,\boldsymbol{h}(\lambda))$. Therefore, $\tau(\theta')\in \Theta_\lambda\cap(\tau(\gamma),\tau(\boldsymbol{h}(\lambda)))$. But this set is empty by the definition of $\gamma$. 

Case 2: $\theta'=\boldsymbol{h}(\lambda)$. By looking at the successive $N_\lambda$-preimages of $\overline{R_\lambda^1}(\theta')\cup \overline{R_\lambda^2}(1-\theta')$, we see that $\theta=\theta'$ implying  $\boldsymbol{h}(\lambda)=\theta$.

Case 3: $\theta'=\gamma$.
%
This implies $2\theta'\in \Theta_\lambda$, namely $R_\lambda^1(2\theta')$ and  $R_\lambda^2(1-2\theta')$ land at the same point. 
By looking at the  $N_\lambda$-preimage of $\overline{R_\lambda^1}(2\theta')\cup \overline{R_\lambda^2}(1-2\theta')$, we see that 
$R_\lambda^1(\theta')$ and $R_\lambda^2(1-\theta')$ both land at the  critical point $0$, meaning that 
$\boldsymbol{h}(\lambda)=\gamma=\theta=\theta'$.
\endproof

\proof:
Let  $\mathcal Y_\theta$ be the set of   parameters $\lambda\in \Omega$ for which the following 
conditions are satisfied\,:
\begin{itemize}
\item
the sets $R_\lambda^1(\theta)$ and $R_\lambda^2(1-\theta)$ are internal rays;
\item
these two rays land at the same  point~$x=x(\lambda)$\,;
\item
the point $x$ is an  eventually repelling  periodic point.
\end{itemize}
We then  define  $\mathcal X_\theta$ as the set of   parameters of  
$\mathcal Y_\theta$ for which the orbit of the landing point~$x$ does not  contain
the critical point~$0$.



\medskip

\setcounter{AFM}{0}
\begin{AFM}  
If  $\theta,\theta'$ are in  $\Xi\cap\mathbb{Q}$ and if  $\theta<\theta'$, then
$$ \mathcal X_\theta \subset \mathcal Y_\theta \subset 
   \mathcal X_{\theta'} \subset \mathcal Y_{\theta'} \,.$$
\end{AFM}

\proof(\textsc{of Claim} a) 
Indeed, if $\lambda$ is in   $\mathcal Y_\theta$, since $R_\lambda^1(\theta)$ and
$R_\lambda^2(1-\theta)$ land at the same   point, the critical point~$0$ are in
the closure of  $V_{1,2}(0,\theta)$ (see  \cite[Lemma 3.14]{Ro08}), where  $V_{1,2}(0,\theta)$ is  the component of 
$\mathbb{\widehat{C}}\setminus \big( \overline{{R}_\lambda^1}(0)\cup \overline{{R}_\lambda^2}(0)\cup \overline{{R}_\lambda^1}(\theta) \cup \overline{{R}_\lambda^2}(1-\theta)\big)$ 
containing  the set ${R}_\lambda^1(\theta/2)$.
  Since $\theta'\in\Xi$, for every $i\ge0$, \ $\tau^i(\theta')$
belong to  $[\theta',1]$,   implying that the orbits of the two sets
$R_\lambda^1\bigl(\theta'\bigr)$ and $R_\lambda^2\bigl(1-\theta'\bigr)$ are always in   $\mathbb{\widehat{C}}\setminus \overline{V_{1,2}(0,\theta)}$.  Therefore $R_\lambda^1\bigl(\theta'\bigr), R_\lambda^2\bigl(1-\theta'\bigr)$ are internal rays.

Let $U_k$ be the component of $N_\lambda^{-k}(\mathbb{\widehat{C}}\setminus \overline{V_{1,2}(0,\theta)})$ containing $R_\lambda^1\bigl(\theta'\bigr)\cup R_\lambda^2\bigl(1-\theta'\bigr)$. Then $V_k=U_0\cap U_1\cap \cdots \cap U_k,  \ k\in \mathbb{N}$ is a shrinking sequence of disks,   each bounded by four internal rays. 
It's easy to verify that 
$$\overline{R_\lambda^1}\bigl(\theta'\bigr)\cup \overline{R_\lambda^2}\bigl(1-\theta'\bigr)=\bigcap_k \overline{V_k}.$$
This implies that the internal rays  $R_\lambda^1\bigl(\theta'\bigr), R_\lambda^2\bigl(1-\theta'\bigr)$ land at the same point
 in  $\mathbb{\widehat{C}}\setminus \overline{V_{1,2}(0,\theta)}$. Moreover, this point is necessarily pre-repelling because of the location of the critical point (see  \cite[Lemma 3.14]{Ro08}),
 and also in  $\mathcal X_{\theta'}$ since $R_\lambda^1(\tau^i(\theta'))$
stays in  $\mathbb{\widehat{C}}\setminus \overline{V_{1,2}(0,\theta)}$  for $i\ge0$.
\endproof

\medskip

From Lemma~\ref{char-S},  if $(\beta,\alpha)$ is a connected component of 
$(0,1/2)\setminus \Xi$, its bounds can be written as  $\alpha=p/(2^k-1)$ and $\beta=p/2^k$.

We  now  prove  by induction on 
$k$,  the following property  
\begin{align*}
\hbox to0pt{\hss$\displaystyle\frak P(k)$\,:\enspace}
   \text{for all \ } & \theta \in
   \Bigl\{\alpha={p\over2^k-1}, \beta={p\over2^k}\Bigr\} \cap
   \Xi, \\
   \left\{ \vphantom{\begin{aligned}
   &\mathcal X_\theta = \mathcal U_{\theta} \text{ \ and} \\
   &\mathcal R_{\mathcal H_-}(2\theta),\ \ \mathcal R_{\mathcal H_+}(-2\theta)
   \text{ \ land at the same    point.} \end{aligned}}\right.
 & \begin{aligned}
   &\mathcal X_\theta = \mathcal U_{\theta} \text{ \ and} \\
   &\mathcal{R}_0^1(\theta), \ \mathcal{R}_0^2(1-\theta)
   \text{ \ land at the same    point.} \end{aligned}
\end{align*}

\smallskip
$\bullet$ {\sl Proof of $\frak P(1)$}\,: \ 
In this case, we need show that $\mathcal X_{1/2}=\Omega$. By Fact \ref{order0},
for all $\lambda\in\Omega$, the sets  $R_\lambda^1(1/2)$ and 
$R_\lambda^2(1/2)$ are internal rays and land at a   pre-image of infinity,  which is  a repelling fixed point. 
Thus, $\mathcal Y_{1/2}=\Omega$. On the other side, 
if $\lambda$ is in $\Omega$, it is not on   $\overline{\mathcal {R}_0^1}(1/2)$.
Thus, $R_\lambda^1(1/2)$ does not land at the critical point~$0$. So, $\lambda$
is in  $\mathcal X_{1/2}$, hence $\mathcal X_{1/2}=\Omega$. Finally, $\frak P(1)$ is  verified since  the rays $\mathcal R_0^1(1/2)$ and $\mathcal R_0^2(1/2)$ both land at   $i\sqrt3/2$ and it follows that  $\Omega=\mathcal U_{1/2}$.

\medskip
$\bullet$ {\sl Idea of the proof of  $\frak P(k)$}\,: \ 
Thanks to the  Claim, we can proof at first that  $\mathcal X_\theta$ is 
included in an open set of the   form $\mathcal U_{\beta'}$, where  $\beta'$ verifies
one of the  properties $\frak P(i)$, ${1\le i \le k-1}$. Then we prove that 
$\mathcal X_\theta$  is an open set whose boundary is contained  in  $\overline{\mathcal{R}_0^1}(\theta)\cup\overline{\mathcal{R}_0^2}(1-\theta)$   and  a finite number of points. 
It follows then by   a topology  argument (see below)  that 
$\mathcal X_\theta$ and $\mathcal U_{\theta}$   differ by a finite set, which is shown to be empty in the final step.

\medskip
$\bullet$ {\sl Proof of  $\frak P(k)$}\,: \ 
Define 
$$ \mathcal E_k = \left\{ {q\over2^r}\in \Xi \mid
   {q\over2^r} > \alpha, \text{ and } r<k\right\} \,.$$
Every element of $\mathcal E_k$ is the  infimum of  a connected component 
  $(q/2^r,q/(2^r-1))$ of $(0,1/2)\setminus \Xi$. Denote by  $\beta'$  the  minimum
of~$\mathcal E_k$.

\begin{AFM}
The sets  $\mathcal X_\alpha$ and  $\mathcal X_\beta$ are open and non empty. 
\end{AFM}

\proof 
If  $\beta_0$ is a dyadic element of~$\Xi$,  Lemma~\ref{exist-beta}
implies that  $\mathcal Y_{\beta_0}$ is non empty.  By Claim a,  for any  $\beta_0<\beta$,
 we have $\mathcal Y_{\beta_0}\subset\mathcal X_\beta\subset\mathcal X_\alpha$. This  allows us to conclude that   $\mathcal X_\beta$ and $\mathcal X_\alpha$
are  non empty.

By definition, for $\lambda_0\in\mathcal X_\theta$, $\theta\in\{\alpha,
\beta\}$, the landing point~$z(\lambda_0)$ of the rays
$R^1_{\lambda_0}(\theta), R^2_{\lambda_0}(1-\theta)$ is  pre-repelling,
and does not contain the critical point in its orbit. Thus, $z(\lambda_0)$ posses a neighborhood in which 
$N_\lambda$, for  $\lambda$ in a neighborhood of~$\lambda_0$, has a  unique pre-repelling point
 $z(\lambda)$  (Fact \ref{con}). Moreover,  the 
orbit of $z(\lambda)$  does not contain the critical   point. This shows that    $\lambda\in\mathcal X_\theta$.
\endproof

\begin{AFM}
For $\theta\in\{\alpha,\beta\}$, the intersection $\partial\mathcal X_\theta \cap
\Omega$ is contained in 
$$ \bigcup_{i\in I}\mathcal R_{0}^1 \bigl(\tau^i(\theta)\bigr) \cup
   \mathcal R_0^2\bigl(1-\tau^i(\theta)\bigr) \cup \mathcal F_\theta $$
where  $\mathcal F_\theta$ is a finite set of points and 
$I=\{i\in\{0,\dots,k-1\}\mid \tau^i(\theta)\in[0,1/2]\}$. Moreover, if  $\lambda
\in\mathcal F_\theta$ and if  $\theta=\beta$ {\rm(}resp. if  $\theta=\alpha${\rm)}, one of the  rays 
$R^1_\lambda(\theta), R^2_\lambda(1-\theta)$ converges to the  critical point~$0$, whose orbit contains $\infty$
{\rm(}resp.  to a parabolic periodic  point{\rm)}.
\end{AFM}

\proof
Indeed, take $\lambda \in \partial\mathcal X_\theta\cap \Omega$. Since $\lambda$ 
is not in $\mathcal X_\theta$,  one of the following cases arises:
\begin{itemize}
\item
One of the sets  $R^1_\lambda(\theta)$, $R^2_\lambda(1-\theta)$ bifurcates. It means that
 one of the sets $R^1_\lambda(\tau^i(\theta))$,
$R_\lambda^2(1-\tau^i(\theta))$ with $0\le i\leq k-1$, contains the critical point~$0$,
in particular $\tau^i(\theta)$ is in   $[0,1/2]$. The parameter~$\lambda$ then belongs
to  $\mathcal R_{0}^1 \bigl(\tau^i(\theta)\bigr)$ or to $\mathcal R_0^2\bigl(1-\tau^i(\theta)\bigr)$.
\item
The sets   $R^1_\lambda(\theta)$, $R^2_\lambda(1-\theta)$ are internal rays and one of them converges to a  pre-periodic point 
whose orbit contains the critical point $0$.
Necessarily, $\theta=\beta$ and the  parameter~$\lambda$ is a solution of one of the equations 
$N_\lambda^{i+1}(0)=N^i_\lambda(0)$, \ $i\leq k$. This gives a finite number of values of $\lambda$. 
\item
The sets   $R^1_\lambda(\theta)$, $R^2_\lambda(1-\theta)$ are internal rays and one of them  converges to a pre-periodic point $x$, not pre-repelling.   In this case   $\theta=\alpha$, and  point $x$ is 
 necessarily
parabolic.   Then, the  system
$$ \left\{ \begin{aligned}
N^k_\lambda(x) &= x \\ (N^k_\lambda)'(x) &= 1 \end{aligned} \right. $$
has a  solution. So  $\lambda$ belongs to a finite set. 
\endproof
\end{itemize}

\begin{AFM}
For $\theta\in\{\alpha,\beta\}$, the set  $\mathcal X_\theta$  is contained in  
$\mathcal U_{\beta'}$ and the intersection $\partial\mathcal X_\theta\cap\mathcal U_{\beta'}$
is  included in 
$$\mathcal R_{0}^1 \bigl(\theta\bigr) \cup
   \mathcal R_0^2\bigl(1-\theta\bigr) \cup
   \mathcal F_\theta \,.$$
\end{AFM}

\proof
The inclusion $\mathcal X_\theta\subset\mathcal X_{\theta'}=\mathcal U_{\beta'}$ follows from the  induction hypothesis on~$\beta'$ and \textsc{Claim} a.

By \textsc{Claim} c, it is enough to prove that 
$$\{\tau^i(\theta) \in [0,1/2]; 0\le i \le k-1\} \cap (0, \beta')=\{\theta\}.$$

%

For $\theta=\beta$, let $i$ be the largest integer such that  $\tau^i(\beta)
\in (0,\beta')$. Necessarily, $\tau^i(\beta)\in \Xi$, because $\tau^{i+k}(\beta)\geq \beta'>\tau^i(\beta)$ for all 
$k\geq 0$.  Moreover, if $i>0$, then $\tau^i(\beta)\in
\mathcal E_k$ because  $\tau^i(\beta)>\alpha$.  But this contradicts the   minimality of  
$\beta'$, so $i=0$.

For $\theta=\alpha$ and $1\le i\le k-1$, the fact that $(\tau^i(\beta),\tau^i(\alpha))$ is an interval  implying that
 $\beta'\le\tau^i(\beta)<\tau^i(\alpha)$.
\endproof

\begin{AFM}
The open sets $\mathcal U_{\beta'}\setminus\overline{\mathcal X}_\alpha$ and
$\mathcal U_{\beta'}\setminus\overline{\mathcal X}_\beta$ are non empty. 
\end{AFM}

\proof
Indeed, it is known from Lemma \ref{char-S} that  $\Xi\cup\{0\}$ has no isolated points and since   $\alpha$  is the supremum 
of the connected  component $(\beta, \alpha)$ of  $(0,1/2)\setminus \Xi$, the interval
$(\alpha,\beta')$ contains at least a connected  component of the open set  $(0,1/2)\setminus \Xi$. 
We denote by $\beta_1$ its infimum. 
By Lemma~\ref{exist-beta}, there is a parameter~$\lambda_1$ 
with $\boldsymbol{h}(\lambda_1)=\beta_1$, and the ray $R_{\lambda_1}^1(\beta_1)$ lands at  the critical point $0$. 
Such  $\lambda_1$  is 
in $\mathcal Y_{\beta_1}$, thus in  $\mathcal X_{\beta'}=\mathcal U_{\beta'}$ (by induction). 
It's clear that $\lambda_1\notin \mathcal X_\alpha$.
If  $\lambda_1\in\partial \mathcal X_\alpha \cap \mathcal U_{\beta'}$,
then by Claim \textbf{d}, either  it is on  $\mathcal R_0^1(\alpha )\cup\mathcal R_0^2(1-\alpha)$, or  $N_{\lambda_1}$ has a  parabolic  point 
(by the proof of Claim c). But, by construction of $N_{\lambda_1}$, this is impossible. 
In conclusion, $\lambda_1$ is in  $\mathcal U_{\beta'}\setminus \overline{\mathcal
X}_\alpha \subset \mathcal U_{\beta'} \setminus \overline{\mathcal X}_\beta$.
\endproof

We finish the proof  of  $\frak P(k)$. \ By the hypothesis of induction, the open set  $\mathcal X_{\beta'}=\mathcal U_{\beta'}$ is a disk of  $\widehat{\mathbb C}$.  The set 
$\mathcal X_\theta$ is included in  $\mathcal U_{\beta'}$  from  Claim a.

Define two closed arcs $\gamma_1:[0,1]\rightarrow \overline{\mathcal R_0^1}(\theta)$ and $\gamma_2:[0,1]\rightarrow \overline{\mathcal R_0^2}(1-\theta)$ by
$$\gamma_1(t)=(\Phi_0^1)^{-1}(te^{2\pi i \theta}), \  \gamma_2(t)=(\Phi_0^2)^{-1}(te^{2\pi i (1-\theta)}).$$

The two open sets $\mathcal X_\theta, \mathcal U_{\beta'}$  and the arcs $\gamma_1, \gamma_2$ satisfy that 
\begin{itemize}

\item
semi-open  arcs  $\gamma_\varepsilon((0,1])$, $\varepsilon=1,2$, are contained in the disk ~$\mathcal U_{\beta'}$;

\item
$\gamma_\varepsilon(0)\in\partial  \mathcal U_{\beta'}$, $\varepsilon=1,2$\,\rm;

\item
$\partial \mathcal X_\theta \cap \mathcal U_{\beta'}$ is included in the  union of a finite set of points and  the semi-open  arcs
  $\gamma_\varepsilon((0,1])$, $\varepsilon=1,2$\,\rm;
  
\item
$ \mathcal U_{\beta'}\setminus\overline{\mathcal X_\theta } \neq \emptyset$.
\end{itemize}

In the following, we claim that $\gamma_1(1)=\gamma_2(1)$, which means that $\mathcal R_0^1(\theta)$ 
and $\mathcal R_0^2(1-\theta)$ land at the same point.  To this end, let $F$ be the finite set so that 
$\partial \mathcal X_\theta \cap \mathcal U_{\beta'} \subset \gamma_1((0,1])\cup \gamma_2((0,1])\cup F$. If 
$\gamma_1(1)\neq\gamma_2(1)$, let's consider the open set $U=\mathcal U_{\beta'}
 \setminus (\gamma_1((0,1])\cup \gamma_2((0,1])\cup F)$. It's clear that $ \mathcal X_\theta\subset U$ and 
 $\partial \mathcal X_\theta \cap U=\emptyset$. This implies that  $U=\mathcal X_\theta$ and $\overline{\mathcal X_\theta }
 =\overline{U}\supset  \mathcal U_{\beta'}$, contradicting the fact $ \mathcal U_{\beta'}\setminus\overline{\mathcal X_\theta } \neq \emptyset$.
 This proves the claim.

The open arc $\gamma_1((0,1])\cup \gamma_2((0,1])$  then separates the  $\mathcal U_{\beta'}$ into two components: $\mathcal U_{\theta}$ and  $ \mathcal U_{\beta'}\setminus\overline{\mathcal U_\theta }$.

 By Claim e,   we know that the  parameter~$\lambda_1$ of Head angle~$\beta_1(> \alpha)$ is contained in $ \mathcal U_{\beta'}\setminus\overline{\mathcal X_\theta}$. Note that $\beta_1\notin \mathcal U_{\beta'}\setminus\overline{\mathcal U_\theta }$.
 We have that 
 $$\mathcal U_\theta= \mathcal X_\theta \cup  \mathcal F_\theta,$$
 where $\mathcal F_\theta$ is a finite set in $\mathcal U_\theta$.  To obtain $\frak P(k)$, we will show in the following that 
 $\mathcal F_\theta=\emptyset$.
 In fact, if $\lambda_0\in\mathcal F_\theta\neq\emptyset$, then there are two possibilities:

If $\theta=\beta$, then one of the   rays $R^1_{\lambda_0}(\tau^i(\theta)),
R^2_{\lambda_0}(1-\tau^i(\theta))$, $0\le i\le k-1$, converges to the critical point $0$
 (\textsc{Claim} c). It follows that the  corresponding angle $\tau^i(\theta)$ 
belongs to  $[0,1/2]$ and $\lambda_0$~is the landing point of the  ray
$\mathcal R_{0}^1(\tau^i(\theta))$ or $\mathcal R_{0}^2(1-\tau^i(\theta))$ \
(Lemma \ref{pre-rigidity}).  But, since $\tau^i(\theta)\ge\theta$ (because $\theta\in \Xi$), none of these  rays
is in $\mathcal U_{\theta}$.

If  $\theta=\alpha$, then one of the  rays $R^1_{\lambda_0}(\theta),
R^2_{\lambda_0}(1-\theta)$ converges to a parabolic   point of period $k$
(\textsc{Claim} c). Let us take a  disk $\Delta \subset \mathcal U_{\alpha}$ such that 
$\Delta\cap\mathcal F_\theta=\{\lambda_0\}$. For $\lambda\in\Delta\setminus
\{\lambda_0\}$, the rays $R_\lambda^1(\theta),R_\lambda^2(1-\theta)$ converge
to  repelling  $k$-periodic points which are stable and that we  denote by 
$z_1(\lambda), z_2(\lambda)$. The functions
$$ \rho_\varepsilon : \Delta\setminus\{\lambda_0\} \longrightarrow {\mathbb C}\setminus \mathbb D,
   \quad \lambda\longmapsto \rho_\varepsilon(\lambda)
 = \bigl(N_\lambda^k\bigr)'\bigl(z_\varepsilon(\lambda)\bigr), $$
are holomorphic  functions and  have an unessential   singularity at~$\lambda_0$.
However, $|\rho_\varepsilon|$ reaches its maximum at~$\lambda_0$, which is  impossible.

This finishes  the proof of the identity $\mathcal U_{\theta}=\mathcal X_\theta$ and
so of  $\frak P(k)$.

 To  complete the proof  of Theorem \ref{H-tongue1}, it 
is enough to verify the  following  two claims.

\begin{AFM} Let $\mathcal Z_\beta$  be the set of  parameters $\lambda\in\Omega$  for which $R^1_\lambda(\beta)$, $R^2_\lambda(1-\beta)$  are internal rays converging to the critical  point~$0$. Then
$$ \mathcal Y_\beta \setminus \mathcal X_\beta = \mathcal Z_\beta
 = \bigl\{\lambda_\beta\bigr\} \,.$$
\end{AFM}

\proof
The identity $\{\lambda_\beta\}=\mathcal Z_\beta$ comes from Lemma \ref{pre-rigidity}. On the other side, if $\lambda$ belongs to  $\mathcal Y_\beta\setminus
\mathcal X_\beta$,
the rays $R^1_\lambda(\beta), R^2_\lambda(1-\beta)$ land at the same 
point whose orbit contains the critical  point~$0$. Since  $0$ and the set ${R}_\lambda^1(\beta/2)$
are contained in the same connected component of 
$\mathbb{\widehat{C}}\setminus \big( \overline{{R}_\lambda^1}(0)\cup \overline{{R}_\lambda^2}(0)\cup \overline{{R}_\lambda^1}(\beta) \cup \overline{{R}_\lambda^2}(1-\beta)\big)$ 
   and that $\tau^k(\beta)\in (\beta,1]$ for every
$k>0$, we conclude that  $R^1_\lambda(\beta)$  converges to~$0$ and $\lambda$
thus belongs to   $\mathcal Z_\beta$. The converse is clear. 
\endproof

\begin{AFM}
Let $\mathcal Z_\alpha$ be the set of  parameters $\lambda\in\Omega$ for which $R^1_\lambda(\alpha)$, $R^2_\lambda(1-\alpha)$ are internal rays converging to the same parabolic 
point. Then
$$ \mathcal Z_\alpha = \bigl\{\lambda_\alpha\} \,.$$
\end{AFM}

\proof
Let  $\beta_n$ be a sequence of dyadic  elements of~$\Xi$ converging to~$\alpha$ and satisfying $\beta_n>\alpha$. We show now that 
$$ \mathcal Z_\alpha = \mathcal Q_\alpha \quad\text{ where }\quad
   \mathcal Q_\alpha = \bigcap_{n\ge0} \Bigl( \overline{\mathcal U}_{\beta_n} \setminus
   \bigl( \mathcal U_{\alpha} \cup \mathcal H_0^1 \cup \mathcal H_0^2 \bigr) \Bigr) \,.$$
The conclusion  then follows since it is clear that   $\mathcal Z_\alpha$ is  finite,  
$\mathcal Q_\alpha$ is  connected and  that $\lambda_\alpha$ belongs to~$\mathcal
Q_\alpha$.

If $\lambda\in\mathcal Q_\alpha$, for every $n\ge0$, the rays
$R^1_\lambda(\beta_n), R^2_\lambda(1-\beta_n)$ land at the same point.
By continuity, the rays $R^1_\lambda(\alpha), R^2_\lambda(1-\alpha)$
land also at a common  point, say $x$. Since 
$x\in\partial B^1_\lambda$ which is a Jordan curve, it is not an  irrational indifferent point. On the other side, $x$ is periodic and in the Julia set,  so its 
 orbit cannot contain the critical  point. It follows that, $x$~is not repelling since   
  $\lambda$ does not belong to~$\mathcal U_{\alpha}$. Therefore, $x$~is
parabolic and  $\lambda\in\mathcal Z_\alpha$.

If $\lambda\in\mathcal Z_\alpha$, then   $\boldsymbol{h}(\lambda)\leq\alpha$. 
Since  $\tau^k(\beta_n)\ge\beta_n>\alpha$ for every  $k\ge0$, the angle $\beta_n\in \Theta_\lambda$ 
and the internal rays 
$R^1_\lambda(\beta_n),R^2_\lambda(-\beta_n)$ thus land at the same pre-repelling point. 
It follows that $\lambda\in\mathcal
U_{\beta_n}$ for every~$n$. On the other side,  $\lambda$ is not in
$\mathcal U_{\alpha}\cup\mathcal H_0^1\cup\mathcal H_0^2$ since $\mathcal U_{\alpha}=\mathcal X_\alpha$
and  $N_\lambda$ has a parabolic point. Therefore,  $\lambda$ belongs to
$\mathcal Q_\alpha$.
\endproof

The proof  of Theorem \ref{H-tongue1} is completed.
\endproof

Recall that the fiber of $\boldsymbol{h}:\Omega_0\cup\{\sqrt{3}i/2\}\rightarrow \Xi$ over $\theta\in\Xi$ is defined by
$$\boldsymbol{h}^{-1}(\theta)=\{\lambda\in \Omega_0\cup\{\sqrt{3}i/2\};\boldsymbol{h}(\lambda)=\theta \}.$$


\begin{cor} \label{H-tongue2}  Let $(\beta,\alpha)$ be a  connected component of $(0,1/2)\setminus \Xi$, then
$$\boldsymbol{h}^{-1}(\alpha)=\Omega_0\cap\big(\big(\mathcal{U}_\alpha\cup\{\lambda_\alpha\}\big)\setminus\big(\mathcal{U}_\beta\cup\{\lambda_\beta\}\big)\big), \ \boldsymbol{h}^{-1}(\beta)=\{\lambda_\beta\}.$$
\end{cor}

\subsection{Further properties of the Head's angle}\label{further-prop}

Theorem \ref{H-tongue1} allows us to give some further properties of the Head's angles.

Given two numbers $\gamma_1, \gamma_2\in  \partial(\Xi)$ 
 with $0<\gamma_1<\gamma_2<1/2$,
 define 
$V(\gamma_1, \gamma_2)$ to be the bounded component of 
 $$\mathbb{C}\setminus \Big( \overline{\mathcal{R}_0^1}(\gamma_1)\cup \overline{\mathcal{R}_0^2}(1-\gamma_1)\cup \overline{\mathcal{R}_0^1}(\gamma_2) \cup \overline{\mathcal{R}_0^2}(1-\gamma_2)\Big).$$    
 
  \begin{lem} \label{Head-bounds} For any $u\in V(\gamma_1, \gamma_2)\cap \Omega_0$, the Head's angle of $N_u$ satisfies 
 $$\gamma_1\leq\boldsymbol{h}(u)\leq \gamma_2.$$
  \end{lem}
  \begin{proof}
 By Theorem \ref{H-tongue1},  we see that the two internal rays $R_u^1(\gamma_2)$ and $R_u^2(1-\gamma_2)$ land at the same point.
 So by the definition of the Head's angle, we have $\boldsymbol{h}(u)\leq \gamma_2$. On the other hand, for any $\theta\in(0, \gamma_1)$, again by Theorem  \ref{H-tongue1}, the  internal rays $R_u^1(\theta)$ and $R_u^2(1-\theta)$ would never land at the same point. So we have $\boldsymbol{h}(u)\geq \gamma_1$. 
 \end{proof}
 
   \begin{lem} \label{semi-cont} For any  $\epsilon>0$ and any $\lambda\in \Omega_0$, there is a neighborhood $\mathcal{U}$ of  $\lambda$, such that for all  $u\in\mathcal{U}\cap\Omega_0 $,  the Head's angle $\boldsymbol{h}(u)$ of $N_u$ satisfies 
 $$(1- \epsilon)\boldsymbol{h}(\lambda) \leq\boldsymbol{h}(u)\leq \frac{4}{3}\boldsymbol{h}(\lambda).$$
  \end{lem}

  We remark that the right part of the above inequality cannot be improved to be $ \boldsymbol{h}(u)\leq (1+ \epsilon)\boldsymbol{h}(\lambda)$.  
    This is because that the Head's angle map  
  $\boldsymbol{h}: \Omega_0\rightarrow \Xi$ is not continuous, following from Theorem \ref{H-tongue1}.
  
\begin{proof}  
Recall that by Lemma \ref{char-S}, each number $s\in \Xi$ can be approximated by a sequence of  numbers  in $\partial(\Xi)$.   Let  $(\beta,\alpha)$  be a component of $(0,1/2)\setminus \Xi$.
Fix a small number $\epsilon>0$.  There are three possibilities:


If $\boldsymbol{h}(\lambda)=\alpha\in \partial(\Xi)$,  then $\boldsymbol{h}(\lambda)$ can be approximated by a decreasing   sequence of  numbers  in $\partial(\Xi)$. We may take $\gamma_1=\beta$ and  $\gamma_2\in (\alpha, (1+\epsilon) \alpha)\cap \partial(\Xi)$. By Theorem \ref{H-tongue1}, the set  $V(\gamma_1, \gamma_2)$ contains a neighborhood $\mathcal{U}$ of $\lambda$.  By Lemma \ref{Head-bounds}, for any $u\in \mathcal{U}\cap\Omega_0 $, we have $\boldsymbol{h}(\lambda)=\alpha \leq\boldsymbol{h}(u)\leq (1+\epsilon)  \boldsymbol{h}(\lambda)$.

If $\boldsymbol{h}(\lambda)=\beta\in \partial(\Xi)$, then $\boldsymbol{h}(\lambda)$ can be approximated by an increasing   sequence of  numbers  in $\partial(\Xi)$. Take $\gamma_1\in ((1-\epsilon)\beta, \beta)\cap \partial(\Xi)$ and $\gamma_2=\alpha$. By  Lemma \ref{char-S},  the number 
$\frac{\alpha}{\beta}$ takes the form $\frac{2^p}{2^{p}-1}$ for some $p\geq 2$. By Theorem \ref{H-tongue1}, the set  $V(\gamma_1, \gamma_2)$ contains a neighborhood $\mathcal{U}$ of $\lambda$.  By Lemma \ref{Head-bounds}, for any $u\in \mathcal{U}\cap\Omega_0 $, we have 
$$\boldsymbol{h}(\lambda)(1-\epsilon) \leq\boldsymbol{h}(u)\leq  \alpha= \frac{\alpha}{\beta} \boldsymbol{h}(\lambda)=\frac{2^p}{2^{p}-1}\boldsymbol{h}(\lambda) \leq  \frac{4}{3}\boldsymbol{h}(\lambda).$$

If $\boldsymbol{h}(\lambda)\in \Xi\setminus  \partial(\Xi)$, then $\boldsymbol{h}(\lambda)$ can be approximated by two sequences of  numbers  in $\partial(\Xi)$ from both sides. We may choose two numbers $\gamma_1, \gamma_2$ with
$$\gamma_1\in ((1-\epsilon)\boldsymbol{h}(\lambda) , \boldsymbol{h}(\lambda))\cap \partial(\Xi), \  \gamma_2\in (\boldsymbol{h}(\lambda), (1+\epsilon)\boldsymbol{h}(\lambda))\cap \partial(\Xi).$$  By Theorem \ref{H-tongue1}, the set  $V(\gamma_1, \gamma_2)$ contains a neighborhood $\mathcal{U}$ of $\lambda$.  By Lemma \ref{Head-bounds}, for any $u\in \mathcal{U}\cap\Omega_0 $, we have 
$(1-\epsilon)\boldsymbol{h}(\lambda) \leq\boldsymbol{h}(u)\leq (1+\epsilon)  \boldsymbol{h}(\lambda)$.
\end{proof}


The following result is  a byproduct of the proof of Lemma \ref{semi-cont}:
 \begin{cor} \label{head-semi-con}  The Head's angle map $\boldsymbol{h}: \Omega_0\rightarrow \Xi$ is lower semi-continuous, that is,   $\liminf_{u\rightarrow \lambda}\boldsymbol{h}(u)\geq \boldsymbol{h} (\lambda)$ for all $\lambda\in \Omega_0$. Moreover, $\lambda\in \Omega_0$ is a point of discontinuity if and only if $\boldsymbol{h}(\lambda)\in \Xi\cap \Theta_{dya}$.    \end{cor}

\section{Articulated rays}\label{a-y}

 In this section, we first recall the  construction of the articulated rays due to the first author \cite{Ro08}, which play a crucial role in the Yoccoz puzzle theory, then we show that 
 the articulated rays satisfy the {\it  local stability property}. This property is important in our later discussion.

 \subsection{The articulated ray}

The articulated ray, analogous to the external ray in the polynomial case, is  the starting point of the  Yoccoz puzzle theory for the  cubic Newton maps.
It's a kind of  Jordan arc that intersects with the Julia set in  countably many  points. The construction of  the articulated rays is due to Roesch \cite[Proposition 4.3]{Ro08}:

\begin{thm}[Roesch]\label{a-ray}
Let $\lambda\in \Omega_0$ and $\zeta\in \Theta_\lambda\cap(\boldsymbol{h}(\lambda), 2\boldsymbol{h}(\lambda))$ be a non-dyadic rational angle.
Then   there exists a unique 
closed arc $L_\lambda(\zeta)$ stemming from $B_\lambda^2$ with angle $-\zeta/4$, converging to $y_\lambda$ which is the  landing point of the internal ray $R_\lambda^3(1/7)$  and satisfying the
3-periodicity condition:
$$N_\lambda^3(L_\lambda(\zeta))=L_\lambda(\zeta)\cup
\overline{R_\lambda^1}(\zeta)\cup \overline{R_\lambda^2}(-\zeta)\cup \overline{R_\lambda^1}(2\zeta)\cup \overline{R_\lambda^2}(-2\zeta).$$
\end{thm}
 
  \begin{figure}[h]
\begin{center}
\includegraphics[height=5.5cm]{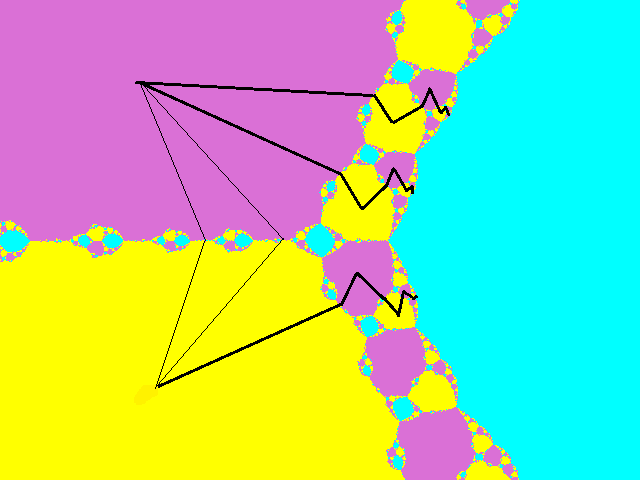}
   \put(-120,133){$-\frac{\zeta}{4}$}     \put(-100,133){$L^0$}       \put(-100,102){$L^1$} \put(-110,62){$L^2$}        \put(-170,139){$b_2(\lambda)$}     \put(-166,126){$\bullet$}  
    \put(-120,113){$\frac{-\zeta}{2}$}   \put(-139,102){$-\zeta$}    \put(-175,100){$-2\zeta$}  
  \put(-160,28){$\bullet$}  \put(-168,17){$b_1({\lambda})$}
   \put(-55,88){$b_3(\lambda)$}    \put(-45,78){$\bullet$}     \put(-65,115){$ \ \ y_\lambda$}  
 \caption{The articulate ray $L_\lambda(\zeta)=L^0$, satisfying $N_\lambda(L^0)=L^1$, $N_\lambda^2(L^0)=L^2\cup
\overline{R_\lambda^1}(\zeta)\cup \overline{R_\lambda^2}(-\zeta)$, and 
$N_\lambda^3(L^0)=L^0\cup
\overline{R_\lambda^1}(\zeta)\cup \overline{R_\lambda^2}(-\zeta)\cup \overline{R_\lambda^1}(2\zeta)\cup \overline{R_\lambda^2}(1-2\zeta)$. Note that four internal rays are  also included in the picture.}
\end{center}\label{f5}
\end{figure}
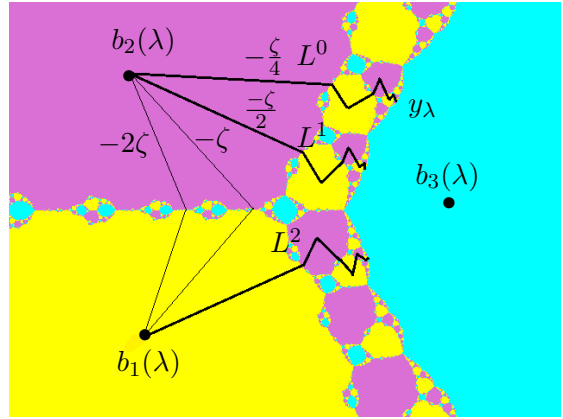

 The closed arc constructed in Theorem \ref{a-ray} is called  {\it the articulated ray}.
Theorem \ref{a-ray} is of fundamental importance  when we study the dynamics of a single map.
 However,  it is   insufficient for our purpose when we  are exploring the parameter plane because
 it does not tell us how the articulated ray moves as the parameter changes.
We need to establish {\it the local stability  property} (though it looks somewhat standard), saying that the articulated rays are also defined for the nearby maps, and what's more, they move  continuously  in Hausdorff topology as  the parameter varies. 
 To this end, we will  sketch the construction of the articulated rays following Roesch \cite{Ro08}, and we will see that the local stability  property will  follow immediately.
 

In our discussion,  we fix  some  $\lambda\in\Omega_0=\Omega\setminus(\mathcal{H}_0^1\cup\mathcal{H}_0^2)$.  By Lemma \ref{semi-cont}, there is a neighborhood $\mathcal{U}\subset\Omega$ of $\lambda$ such that for all $u\in \mathcal{U}\cap\Omega_0$,
$$\Big(1-\frac{1}{10}\Big)\boldsymbol{h}(\lambda)\leq \boldsymbol{h}(u)\leq\Big(1+\frac{1}{3}\Big)\boldsymbol{h}(\lambda).$$
This implies  that
$$H_{\mathcal{U}}:=\sup_{u\in \mathcal{U}\cap\Omega_0 } \boldsymbol{h}(u)<  \frac{3}{2}\inf_{u\in \mathcal{U}\cap\Omega_0} \boldsymbol{h}(u).$$
By shrinking $\mathcal{U}$ if necessary, we may assume that $H_{\mathcal{U}}<1/2$. 
By Lemma \ref{char-S}, we know that 
 $\Xi\cup\{0\}$ is  closed  and perfect, therefore  $H_{\mathcal{U}}\in\Xi$
 and there is   a non-dyadic rational number $\sigma\in\Xi\cap [H_{\mathcal{U}},1/2)$,  equal to or  sufficiently closed
to $H_{\mathcal{U}}$, satisfying that
 $$C_{\sigma} \bigcap\Big[\sigma, \frac{4}{3}\sigma\Big]\subset\bigcap_{u\in\mathcal{U}\cap\Omega_0}\Big(\Theta_u\cap(\boldsymbol{h}(u), 2\boldsymbol{h}(u))\Big),  \ \ \  (*)$$
here, recall that the  sets $C_\sigma, \Theta_u$ are  defined by
$$C_\sigma=\bigcap_{k\geq0}\tau^{-k}([\sigma,1]), \  \Theta_u=C_{\boldsymbol{h}(u)}=\bigcap_{k\geq0}\tau^{-k}([\boldsymbol{h}(u),1]),$$
where $\tau$ is the doubling map on the unit circle, see  Section \ref{head}. 

By Proposition \ref{density}, there is a dyadic angle $\theta\in C_{\sigma} \bigcap(\sigma, \frac{4}{3}\sigma)$.
By $(*)$ and  Facts \ref{con} and \ref{int-ray}, for all $u\in \mathcal{U}$ 
(shrink $\mathcal{U}$  if necessary),   the sets $R_u^{1}(\theta),  R_u^{2}(1-\theta)$ are internal rays landing at same point 
$p_u$ which is pre-repelling (note that $N_u^{k}(p_u)=\infty$ for some positive integer $k$), and   their closures  
 $\overline{R_u^{1}}(\theta),  \overline{R_u^{2}}(1-\theta)$   move continuously with respect to $u\in \mathcal{U}$.
 

 \begin{figure}[!htpb]
  \setlength{\unitlength}{1mm}
  {\centering
  \includegraphics[width=62mm]{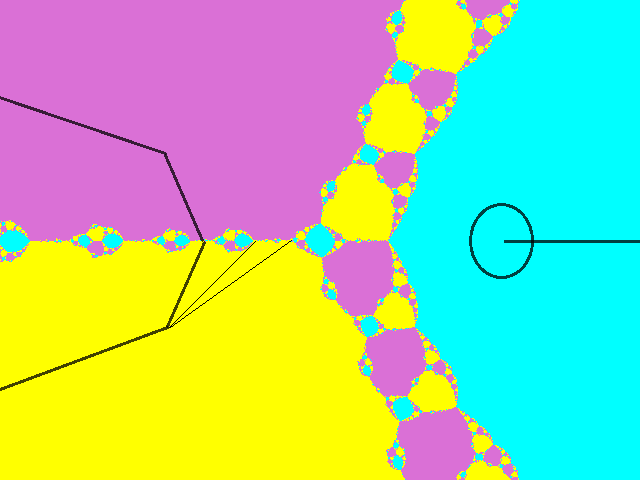}
  \includegraphics[width=62mm]{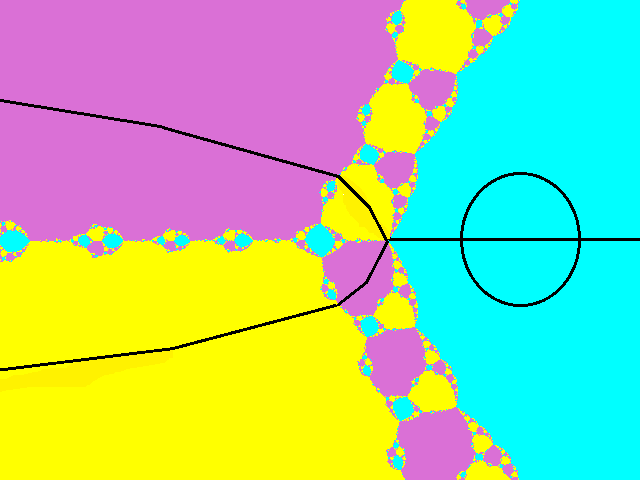}
   \put(-60,33){$0$}  \put(-60,13){$0$} \put(-47,33){$\bullet$} \put(-47,12){$\bullet$}\put(-50,37){$b_2(u)$} \put(-50,8){$b_1(u)$}   \put(-13,22.5){$\bullet$} \put(-3,24){$0$} \put(-23,25){$\frac{1}{2}$}  \put(-28,37){$W_u^2$}
 \put(-28,8){$W_u^1$}
 \put(-122,32){$0$}  \put(-122,12){$0$} \put(-110,31){$\bullet$} \put(-110,14){$\bullet$}\put(-112,35){$b_2(u)$} \put(-112,10){$b_1(u)$}    \put(-70,24){$0$}     \put(-78,22.3){$\bullet$}   \put(-110,18){$\theta$}  \put(-106,20){$\zeta$}     \put(-102,18){$\gamma$}   
  \put(-93,25){$V_u$} 
}
  \caption{The domains $V_u$ (left) and $W_u^\varepsilon$ (right).  Note that:
  the internal rays with angles $\zeta, \gamma=\boldsymbol{h}(u)(\geq \sigma)$ are  included (left), 
  a part of  $N_u^{-1}(G)$ is not included (right).
  } 
\end{figure}
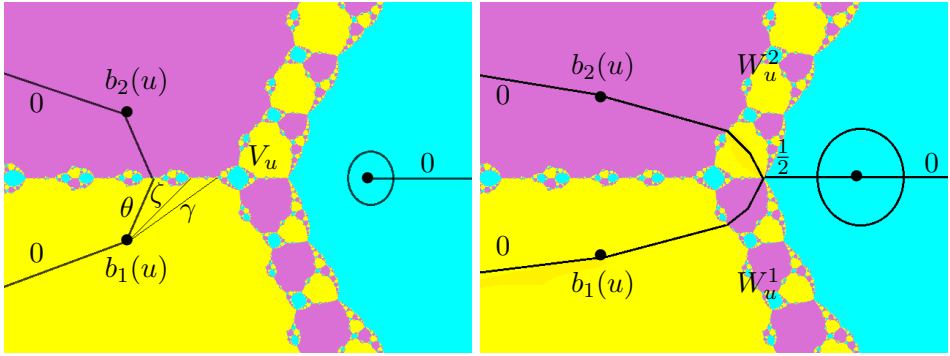

For any  $u\in \mathcal{U}$, we denote by  $V_u$ the connected component of $\mathbb{\widehat{C}}\setminus G$ where
$$G=\overline{R_u^1}(0)\cup \overline{R_u^2}(0)\cup \overline{R_u^1}(\theta)\cup
 \overline{R^2_u}(1-\theta)\cup \overline{R_u^3}(0)\cup E_u^3(1/2)$$ (here $E_u^3(1/2)=(\phi_u^3)^{-1}(\{|z|=1/2\})$ is the equipotential curve in $B_u^3$) which does not
contain the critical value $N_u(0)$. Since $V_u$ is a topological
disk, its preimage $N_u^{-1}(V_u)$ has three connected components and only one of
them intersects both $B_u^3$ and $B_u^\varepsilon$ ($\varepsilon=1,2$); we denote it by $W_u^\varepsilon$. One may see  that $W_u^\varepsilon\subset V_u$. We denote the inverse of  $N_u:W_u^\varepsilon\rightarrow V_u$ by $g_{u,\varepsilon}$. It's clear that $g_{u,\varepsilon}: V_u \rightarrow  W_u^\varepsilon$ extends continuously to 
$R_u^1(0)\cup R_u^2(0)\cup R_u^1(\theta)\cup R^2_u(1-\theta)\cup\{b_1(u), b_2(u)\}$, which is a part of the boundary $\partial{V_u}$. So the points $g_{u,\varepsilon}(b_{\varepsilon'}(u))$ with $\varepsilon,\varepsilon'\in\{1,2\}$ are well-defined.
%

  By Proposition \ref{density} and note that $\alpha$ is non-dyadic, we may  find   a non-dyadic rational angle
 $\zeta\in C_{\sigma}\cap(\sigma, \theta)$  satisfying the following conditions:
 
 \vspace{5pt}
 
 \textbf{C1.}  For all $u\in\mathcal{U}$ (shrink it if necessary), the  set
$$X_u=\overline{R_u^2}(-\zeta)\cup \overline{R_u^1}(\zeta)\cup \overline{R_u^1}(\zeta/2)\cup \overline{g_{u, 1}(R_u^2(-\zeta))}$$
 is a closed arc connecting $b_2(u)$ with $g_{u, 1}(b_2(u))$ (a well-defined point on $\partial W_u^1$) and $X_u$ avoids the free critical point $0$ of $N_u$.
 
When $u=\lambda$,  if $\lambda$ is {\it non-hyperbolic} (i.e. $N_\lambda$ is not hyperbolic), we further require  that  the set $X_\lambda$ avoids the  orbit of the critical point $0$ of
$N_\lambda$ (this is feasible because  there are many angles  $\zeta\in C_{\sigma}\cap(\sigma, \theta)$ for choice and at least one meets the requirement).

 \textbf{C2.} $X_u$ moves continuously with respect to $u\in \mathcal{U}$ in Hausdorff
 topology (by Fact \ref{con} and shrink $\mathcal{U}$ if necessary).
 



\
 
 Now we consider the continuous family of holomorphic  maps 
 $$H_u=g_{u, 2}\circ g_{u, 2}\circ g_{u, 1}: V_u\rightarrow V_u$$
 parameterized by $\mathcal{U}$. 
 Here are several observations: first,  $H_u$ maps $V_u$ to a proper subset of $V_u$ and $\overline{H_u(V_u)}\subset V_u$; second, the landing point of $R_u^3(1/7)$, denoted by $P_u$, is an attracting fixed point of $H_u$, and it is also the unique   fixed point  of $H_u$ in $V_u$  
 because $H_u$ contracts the hyperbolic metric of  $V_u$.



The following is a refinement of Theorem \ref{a-ray}, sufficient for our purpose: 

\begin{thm} [Articulated ray and local stability] \label{art-con}
Let $\lambda\in\Omega_0$ be non-hyperbolic and $\mathcal{U}, \sigma, \theta$, $\zeta\in C_{\sigma}\cap(\sigma, \theta)$ be chosen above. Let $Y_u=g_{u, 2}\circ g_{u, 2} (X_u)$.
Then there is a neighborhood $\mathcal{V} \subset\mathcal{U}$ of $\lambda$ satisfying that

1. for all $u\in\mathcal{V}$,  the set 
$$L_u(\zeta)=\{P_u\}\cup \bigcup_{k\geq 0} H_u^k(Y_u)$$
 is a closed arc 
 stemming from $B_u^2$ with angle $-\zeta/4$, converging to $P_u$ which is the  landing point of the internal  ray $R_u^3(1/7)$. Moreover, 
 
2. the  closed arc $L_u(\zeta)$ moves continuously in Hausdorff topology with respect to $u\in \mathcal{V}$. 
 \end{thm}
 
We still call the closed arc  $L_u(\zeta)$     an {\it articulated ray}.
 

 
\begin{proof}  The first statement is actually Theorem \ref{a-ray}, see \cite[Proposition 4.3]{Ro08} for a proof.  

 The idea of  proof of the second statement,  same as that of Fact \ref{con}, is to  decompose the set $L_u(\zeta)$ into two parts: 
$$\bigcup_{0\leq k\leq l} H_u^k(Y_u)  \text{ \ and \ }   \{P_u\}\cup \bigcup_{k\geq l} H_u^k(Y_u).$$
We only  sketch the proof. The assumption that  $X_\lambda$ avoids the free critical orbit of $N_\lambda$  guarantees that 
 there exist  a large integer $l>0$ (independent of  $\mathcal{U}$) and  a smaller neighborhood $\mathcal{V}\subset \mathcal{U}$ of $\lambda$, such  that   $\cup_{0\leq k\leq l} H_u^k(Y_u) $
  does not meet the set $\{H_u^{j}(0), 0\leq j\leq l\}$. Therefore $\cup_{0\leq k\leq l} H_u^k(Y_u)$
 is an arc and moves continuously in $u$;  the latter part $ \{P_u\}\cup (\cup_{k\geq l} H_u^k(Y_u))$ is contained in a linearized neighborhood of $P_u$ and therefore continuous in $u$.
Gluing them together, we get the continuity  of $L_u(\zeta)$ over $\mathcal{V}$.
\end{proof}

\subsection{Two graphs} \label{two-graph}
 Let $\lambda\in\Omega_0$ be non-hyperbolic  and $\mathcal{V}, \sigma, \theta, \zeta$ be chosen in Theorem \ref{art-con}. 
 Fix a rational angle $\kappa\in C_\sigma \cap (\sigma, \theta)$ so that $\kappa$ and $\zeta$ have disjoint orbit under the doubling map $\tau$ (such $\kappa$ always exists because of Proposition \ref{density}).  
 It follows from Fact \ref{con} and Theorem \ref{art-con} that  for all $u\in \mathcal{V}$ (shrink it if necessary),   the  following two graphs 
  $$G(u,\zeta)= \bigcup_{ j\geq 0} (N_u^j(L_u(\zeta)) \cup \overline{R_u^3}({2^{j}}/{7})),$$
$$I(u, \kappa)=\Big(\overline{R_u^1}(0)\cup \overline{R_u^2}(0)\cup \overline{R_u^3}(0)\Big)\bigcup\bigcup_{j\geq 0} \Big(\overline{R_u^1}(2^{j}\kappa)\cup \overline{R_u^2}(1-2^{j}\kappa)\Big),$$
 avoid the  critical point $0$ and they move continuously in Hausdorff topology  with respect to $u\in\mathcal{V}$.

 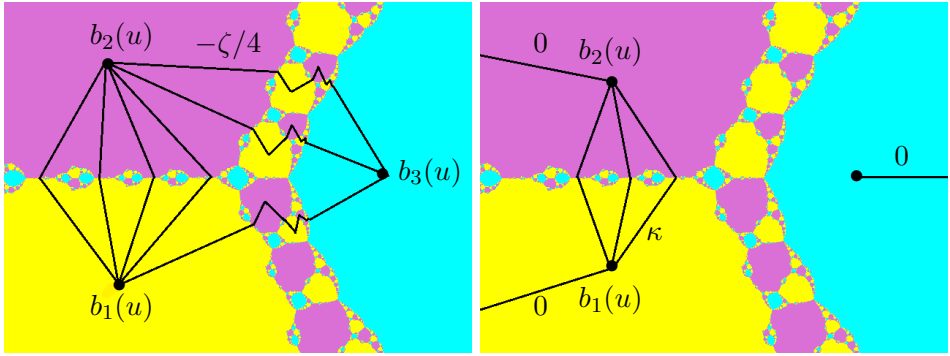
\begin{figure}[!htpb]
  \setlength{\unitlength}{1mm}
  {\centering
  \includegraphics[width=62mm]{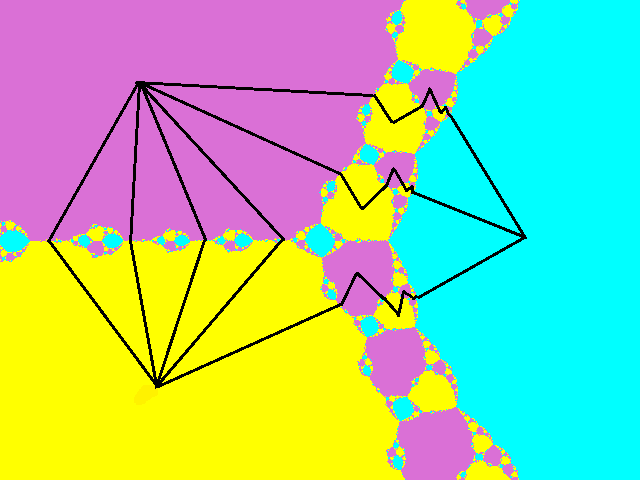}
  \includegraphics[width=62mm]{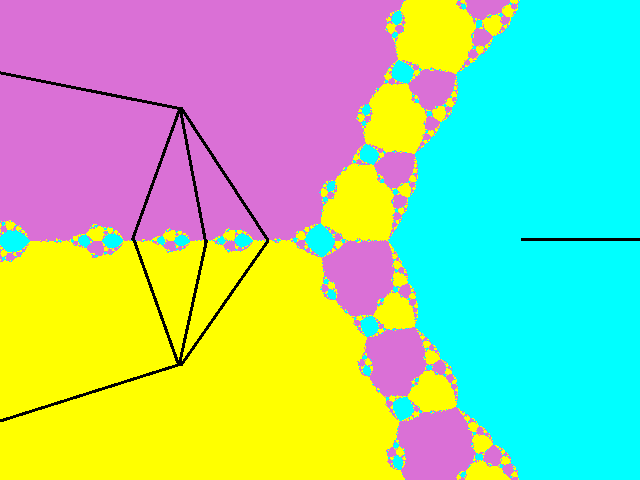}
  \put(-45.5,35){$\bullet$} \put(-45.5,10.5){$\bullet$}\put(-49,39){$b_2(u)$} \put(-49,6){$b_1(u)$}   \put(-13,22.5){$\bullet$}
  \put(-7,25){$0$}  \put(-55,40){$0$}  \put(-55,5){$0$}  \put(-40,15){$\kappa$}
      \put(-112.5,37.5){$\bullet$} \put(-114,41){$b_2(u)$} \put(-111,8){$\bullet$} \put(-114,5){$b_1(u)$}    \put(-73,23){$b_3(u)$}     \put(-76,22.8){$\bullet$}    \put(-100,40){$-\zeta/4$}
}
  \caption{The graphs $G(u,\zeta)$ (left) and $I(u, \kappa)$ (right).   } 
\end{figure}

 Note that  we have assumed that  $\lambda\in\Omega_0$  is  non-hyperbolic, this in  particular implies that  at least one of the graphs $G(\lambda,\zeta), I(\lambda, \kappa)$ avoids the free critical orbit (i.e. the orbit of $0$) of $N_\lambda$. 
 If $G(\lambda,\zeta)$ avoids the free critical orbit, we set $G_u=G(u,\zeta)$, else, we set  $G_u=I(u,\kappa)$.
 It's clear that in either case, $G_\lambda$ avoids the free critical orbit of $N_\lambda$.

 For the graphs $\{G_u\}_{u\in \mathcal{V}}$, we have  the following {\it local stability property}:

 \begin{fact} \label{stability}  
 Let  $\lambda\in\Omega_0$ be non-hyperbolic  and $\{G_u\}_{u\in \mathcal{V}}$   be the family of graphs defined above. 
 Then for any integer $q\geq 0$, there is a neighborhood $\mathcal{V}_q\subset \mathcal{V}$ ($\mathcal{V}_q$ depends on $q$) of $\lambda$, satisfying that 
 
a). for all  $u\in \mathcal{V}_q$,  the set  $N^{-q}_u(G_u)$ avoids the free critical point $0$,  and

b). the set  $N^{-q}_u(G_u)$  moves continuously  in Hausdorff topology  with respect to $u\in \mathcal{V}_q$.
 \end{fact}
 
 \begin{proof}  By the choice of the graphs $\{G_u\}_{u\in \mathcal{V}}$,  for any  integer $q\geq 0$, 
 the points $0, N_\lambda(0), \cdots, N^q_\lambda(0)$ avoid the graph $G_\lambda$. By the continuity of the graphs  $\{G_u\}_{u\in \mathcal{V}}$ (see Fact \ref{con} and Theorem \ref{art-con}) and the points $N_u^k(0), 1\leq k\leq q$, 
 we  may choose  a neighborhood of $\lambda$, say $\mathcal{V}_q\subset \mathcal{V}$, such that the points $0, N_u(0), \cdots, N^q_u(0)$ do not meet $G_u$, for all $u\in \mathcal{V}_q$.  The conclusion then follows.
\end{proof}
 
 Fact \ref{stability} will be useful in the proofs of  Lemmas \ref{land}, \ref{imp-2} and Theorem \ref{capture-jordan}.

\section{Characterization  of maps on  $\partial\mathcal{H}^\varepsilon_0$} \label{char-maps}

In  this section, we study the  properties of the maps on  $\partial\mathcal{H}^\varepsilon_0$.



\subsection{Characterization of $\partial\mathcal{H}^\varepsilon_0$}

\begin{thm}\label{pri} Let $\varepsilon\in \{1,2\}$. If  $\lambda\in \partial
\mathcal{H}_0^\varepsilon\cap\Omega$, then  $\partial B_\lambda^\varepsilon$   contains either  the free critical point
$0$ or a parabolic cycle of $N_\lambda$.
\end{thm}

We remark that with a little more effort, one can show that $\lambda\in\partial\mathcal{H}_0^\varepsilon\setminus\{0\}$ if and only if
$\partial B_\lambda^\varepsilon$   contains either  the free critical point
$0$ or a parabolic cycle of $N_\lambda$. However, this result is not necessary for our purpose in this paper, so we will not give the proof.
To prove Theorem \ref{pri}, we  need the following:

\begin{lem} \label{cripar} Let $\lambda\in \Omega_0$. If $\partial B_\lambda^\varepsilon$   contains neither  the free critical point
$0$ nor a parabolic cycle of $N_\lambda$, then there exist an integer
$K\geq 1$ and two topological disks $U_\lambda,V_\lambda$ with
${B_\lambda^\varepsilon}\Subset V_\lambda\Subset U_\lambda$, such that
$N^K_\lambda: V_\lambda\rightarrow U_\lambda$ is a polynomial-like
map of degree $2^K$.
\end{lem}

There are two proofs of Lemma \ref{cripar}. One is based on the Yoccoz puzzle technique, the other is based on a theorem of Ma\~n\'e. To avoid  technical argument, we prefer to use  Ma\~n\'e's Theorem in the proof. Here is the detail:

\begin{proof}   Write $N_\lambda^{-1}(B_\lambda^\varepsilon)=B_\lambda^\varepsilon \cup T_\lambda^\varepsilon$.  By the assumption that  $0\notin\partial B_\lambda^\varepsilon$,  we know that  $B_\lambda^\varepsilon$ and $T_\lambda^\varepsilon$  have disjoint closures.  So there is a disk neighborhood $U$ of $\overline{B_\lambda^\varepsilon}$, without intersection with $T_\lambda^\varepsilon$. It is easy to see that $V=N_\lambda(U)$ is also a neighborhood of $\overline{B_\lambda^\varepsilon}$.

Take a Riemann mapping $h: \mathbb{\widehat{C}}\setminus \overline{B_\lambda^\varepsilon}\rightarrow \mathbb{\widehat{C}}\setminus{\overline{\mathbb{D}}}$. The  map
$g=hN_\lambda h^{-1}: h(U\setminus \overline{B_\lambda^\varepsilon}) \rightarrow  h(V\setminus \overline{B_\lambda^\varepsilon})$ is a holomorphic map.  By the Schwarz reflection principle, we can extend $g$ to a holomorphic map $G$, mapping a neighborhood of the circle $\partial\mathbb{D}$ to another neighborhood of $\partial\mathbb{D}$. By assumption,  the restriction  $G|_{\partial\mathbb{D}}:\partial\mathbb{D}\rightarrow 
\partial\mathbb{D}$ has neither critical point nor non-repelling cycles.  By Ma\~n\'e's Theorem \cite{Mane},  the circle $\partial\mathbb{D}$ is a hyperbolic set of $G$.  This means, there exist  constants 
$C>0, \rho>1$,  such that for all $k\geq1$ and all $z\in\partial\mathbb{D}$, 
$$|(G^{k})'(z)|\geq C\rho^k.$$
Then one can find an integer $K\geq 1$ and   two annular neighborhoods $X, Y$ of $\partial\mathbb{D}$  with $X\Subset Y\subset U$, such that $G^K: X\rightarrow Y$ is a proper map of degree $2^K$ (the sets $X,Y$ can be  constructed by hand, see for example the proof of \cite[Proposition 6.1]{QWY}).
By pulling back $X\setminus{\overline{\mathbb{D}}}, Y\setminus{\overline{\mathbb{D}}}$ via $h$,  we  get a polynomial-like map $N^K_\lambda: V_\lambda\rightarrow U_\lambda$,
  where $$V_\lambda=h^{-1}(X\setminus{\overline{\mathbb{D}}})\cup  \overline{B_\lambda^\varepsilon}, \ U_\lambda=h^{-1}(Y\setminus{\overline{\mathbb{D}}})\cup  \overline{B_\lambda^\varepsilon}.$$ 
  
  This completes the proof. \end{proof}

\

\noindent{\it Proof of Theorem \ref{pri}.}
By Lemma \ref{cripar}, there is a neighborhood  $\mathcal{U}$ of $\lambda$, such
that for all $u\in \mathcal{U}$, the map  $N_u^K$ has only one  critical value  in $\overline{U_\lambda}$.
 This critical value is nothing but $b_\varepsilon(u)$.
Thus the
component $V_u$ of  $N^{-K}_u(U_\lambda)$ that contains $b_\varepsilon(u)$ is
a disk. Since $\partial V_u$ moves holomorphically with respect to
$u\in \mathcal{U}$, we may shrink $\mathcal{U}$ a little bit
 so that $V_u\Subset U_\lambda$ for all  $u\in \mathcal{U}$. Set $U_u=U_\lambda$. In this way, we get a polynomial-like map $N^K_u: V_u\rightarrow
 U_u$  of degree $2^K$ for all $u\in
 \mathcal{U}$.

 However when $u\in\mathcal{U}\cap\mathcal{H}_0^\varepsilon$, the basin $B_u^\varepsilon$ contains two critical points of $N_u$ and the degree of
 $N^K_u: V_u\rightarrow
 U_u$ is  $3^K$. This is a contradiction.
\hfill $\Box$

\subsection{Parameter ray vs dynamical ray}
Recall that $\mathcal{R}_0^1(t)$ is the parameter ray in $\mathcal{H}_0^1$, defined in Section \ref{parameter-ray}. 
For any $t\in[0,\frac{1}{2}]$, the  {\it impression } $\mathcal{I}_t$ of $\mathcal{R}_0^1(t)$  is defined as the  intersection of  the shrinking   closed sectors
$\overline{\mathcal{S}_k(t)}$, where
$$\mathcal{S}_k(t)=(\Phi_0^1)^{-1}(\{re^{2\pi i \theta}; r\in (1-1/k,1), \ \theta\in(t-1/k,t+1/k)\cap [0,1/2]\}).$$


\begin{lem} \label{land} For any  $t\in [0, \frac{1}{2}]$ and any  $\lambda\in
\mathcal{I}_t\cap \Omega$,

1. if $N_\lambda$ has no parabolic cycle, then
$R^1_\lambda(t)$ lands at $0$;

2. if $N_\lambda$ has a parabolic cycle, then  $R^1_\lambda(t)$ lands at a parabolic point.
\end{lem}

\begin{proof} For any  $\lambda\in
\mathcal{I}_t\cap \Omega$, 
it follows from Theorem \ref{pri} that either $0\in
\partial B^1_\lambda$ or $\partial B^1_\lambda$ contains a parabolic cycle.
If  $\partial B^1_\lambda$ contains a parabolic cycle, then by \cite[Lemma 6.5]{Ro08},
there exist an integer $p\geq 1$
and two disks $U$ and $V$ containing  the free critical point $0$, such that  $N_\lambda^p:
U\rightarrow V$ is a quadratic-like map satisfying the following two properties:

(i). $N_\lambda^p: U\rightarrow V$ is hybrid equivalent to  $q(z)=z^2+1/4$, and

(ii).  The  filled Julia set $K$ of $N_\lambda^p:
U\rightarrow V$ intersects
$\partial B_\lambda^1$ at exactly one point.  This point  is a parabolic fixed point of $N_\lambda^p$, say $\beta_\lambda$.

 It is known from Theorem \ref{roesch-jordan} that $\partial B^1_\lambda$ is a Jordan curve.  This implies that the free critical point $0$  (in the case that $0\in \partial B^1_\lambda$), or the parabolic point
$\beta_\lambda$ (in the case that $\partial B^1_\lambda$ contains a parabolic cycle)
is necessarily a landing point of some  internal ray, say $R^1_\lambda(t')$.

In the following, we show $t'=t$.    The proof is based on the  local stability property, as stated in Fact \ref{stability}.
We assume by contradiction that $t'\neq t$. 
Then by Fact \ref{stability}, 
there is a graph $G_\lambda$ avoiding the free critical orbit of $N_\lambda$, an integer $q\geq0$  and a neighborhood $\mathcal{V}'$ of $\lambda$, satisfying that

(a).   $G_u$ is well-defined and continuous when  $u$ ranges over $\mathcal{V}'$;
 
(b).  $N_u^{-q}(G_u)$ avoids the free critical point $0$ for  all  $u\in\mathcal{V}'$;

(c).   when $u=\lambda$, the graph $N_\lambda^{-q}(G_\lambda)$ separates  $R_\lambda^1(t')$ and $R_\lambda^1(t)$.

The third property (c) implies that $N_\lambda^{-q}(G_\lambda)$ also separates $R_\lambda^1(t')$ and a sector neighborhood of   $R_\lambda^1(t)$.  The sector neighborhood can be chosen as follows. We may first  choose  rational angles $t_1, t_2$ such that

(1).   $t_1<t<t_2$  in counter clockwise order.
 
 (2).  The internal rays  with angles $t_1,t,t_2$  are in the same component of $\mathbb{\widehat{C}}\setminus N_\lambda^{-q}(G_\lambda)$.
 

(3). In a neighborhood $\mathcal{V}_0\subset \mathcal{V}'$ of $\lambda$, the sets $R_u^1(\theta)$ with $\theta\in\{t_1,t_2\}$ are   internal rays,  avoiding  the points
 $N^j_u(0)$ with $0\leq j\leq q$,  and their closures  move continuously (this is guaranteed by suitable choices of the angles and  Fact \ref{con}).

Let $S_{u}(t_1, t_2)$  be the open sector containing $(\phi_u^1)^{-1}\{(0,1/2)e^{2\pi i t}\}$ and  bounded by $\partial B_u^1$, $\overline{R_u^1}(\theta), \theta\in\{t_1,t_2\}$. 
By the continuity  of $N_u^{-q}(G_u)$, we see that for  all $u\in \mathcal{V}_0$,  the free critical point  $0$ and $(\phi_u^1)^{-1}\{(0,1/2)e^{2\pi i t'}\}$ are contained in the same component, say $D_u$,  of  $\mathbb{\widehat{C}}\setminus N_u^{-q}(G_u)$,  and $D_u\cap S_{u}(t_1, t_2)=\emptyset$.
However, by  the assumption $\lambda\in \mathcal{I}_{t}$  and the definition of $\mathcal{I}_{t}$, we know  that 
 when $u\in \mathcal{S}_k(t)\cap \mathcal{V}_0$ with ${1}/{k}<\min\{|t_1-t|, |t_2-t|\}$, the free critical point $0\in S_{u}(t_1, t_2)$. This is  a contradiction.
\end{proof}




\section{Yoccoz puzzle theory revisited}\label{yoccoz}

\subsection{The Yoccoz puzzle theory}


 Let $X,X'$ be connected open subsets of $\mathbb{\widehat{C}}$ with finitely
many smooth boundary components and such that $X'\Subset X\neq \mathbb{\widehat{C}}$. A holomorphic
map $f : X'\rightarrow X$ is called a {\it rational-like map}
if it is proper and has finitely many critical points  in $X'$. We denote by ${\rm deg}(f)$ the topological degree
of $f$ and by $K(f)=\bigcap_{n\geq 0}
f^{-n}(X)$ the
 {\it filled Julia set}, by $J(f)=\partial K(f)$ the {\it Julia set}. A rational-like map $f : X'\rightarrow X$ is called {\it simple} if its  filled Julia set $K(f)$ contains only one critical point, and with multiplicity one.

Although we do not use here, we remark that  an analogue of Douady-Hubbard's straightening
theorem \cite{DH}  holds for rational-like maps. That is, a rational-like map  $f : X'\rightarrow X$
 is always  {\it hybrid equivalent} to a rational map $R$ of degree
${\rm deg}(f)$; if $K(f)$ is connected, such $R$
is unique up to M\"obius conjugation,
provided we further require that $R$ is  post-critically finite
outside  its filled Julia set, see \cite[Theorem 7.1]{W}.

A finite, connected graph $\Gamma$ is called a {\it puzzle} of $f$ if it
satisfies the conditions:  $\partial X\subset \Gamma$, $f(\Gamma\cap X')\subset \Gamma$,  and the orbit of each critical point  of $f$ avoids $\Gamma$.

The {\it puzzle pieces} $P_n$ of depth $n$ are the connected components of $f^{-n}(X\setminus \Gamma)$
and the one containing the point $x$ is denoted by $P_n(x)$.  For any  $x\in J(f)$, let ${\rm orb}(x)=\{x,f(x),f^2(x),\cdots\}$ be the forward orbit of $x$. For $n\geq0$,
let $P_n^*(x)=\overline{P_n(x)}$ if ${\rm orb}(x)\cap \Gamma=\emptyset$, and $P_n^*(x)=\bigcup_{x\in \overline{P_n}}\overline{P_n}$
if ${\rm orb}(x)\cap \Gamma\neq\emptyset$.
The  {\it impression}  ${\rm Imp}(x)$  of $x$   is defined by
\begin{equation*}
{\rm Imp}(x)=
 \bigcap_{n\geq0} {P^*_n(x)}.
\end{equation*}

A puzzle $\Gamma$ is said to be $k${\it-periodic} at a critical point $c$ if $f^k(P_{n+k}(c))=P_n(c)$ for any  $n\geq 0$, where $k\geq1$ is  some smallest integer.

 A puzzle $\Gamma$ is said {\it admissible} if satisfies the conditions: 
(a). for each critical point $c\in K(f)$, there is an integer $d_c\geq0$ such that  $P_{d_c}(c)\setminus\overline{ P_{d_c+1}(c)}$ is a non-degenerate annulus;
 (b). each puzzle piece is a topological disk. 


 The following result is fundamental and well-known:

 \begin{thm}[Branner-Hubbard \cite{BH}, Roesch \cite{Ro99}, Yoccoz \cite{M2, H}]\label{Y-P}
Let $f:X'\rightarrow X$ be a simple rational-like map, with critical point $c\in K(f)$. Suppose that $\Gamma$ is an admissible puzzle.

1. If $\Gamma$ is not periodic at $c$, then $K(f)=J(f)$ and for any $x\in J(f)$, the impression ${\rm Imp}(x)=\{x\}$.

2. If $\Gamma$ is $k$-periodic at $c$, then $f^k:P_{n+k}(c)\rightarrow P_{n}(c)$ for some large $n$ defines a
quadratic-like map with filled Julia set ${\rm Imp}(c)$. Moreover,
\begin{equation*}
{\rm Imp}(x)=\begin{cases}
 \text{a conformal copy of ${\rm Imp}(c)$},\ &\text{ if }x\in \bigcup_{k\geq 0}f^{-k}({\rm Imp}(c)),\\
\{x\},\ &\text{ if }x\in K(f)-\bigcup_{k\geq 0}f^{-k}({\rm Imp}(c)).
\end{cases}
\end{equation*}
 \end{thm}


In general, Theorem \ref{Y-P} is used to study  the topology  (e.g. connectivity and local connectivity) of the Julia set. By applying complex analysis especially some distortion results, one can further  study the analytic property (e.g. Lebesgue measure and Hausdorff dimension)
 of the Julia set. One of the fundamental result is due to Lyubich and Shishikura:

  \begin{thm}[Lyubich, Shishikura]\label{Lyu}
Let $f:X'\rightarrow X$ be a simple rational-like map, with critical point $c\in K(f)$. Suppose that $\Gamma$ is an admissible puzzle. 

1. If $\Gamma$ is not periodic at $c$, then the Lebesgue measure of $J(f)$ is zero.

2. If $\Gamma$ is periodic at $c$, then  Lebesgue measure of $J(f)$ is zero if and only if the  Lebesgue measure of $\partial {\rm Imp}(c)$ is zero.
 \end{thm}

  Theorem \ref{Lyu} is slightly stronger than Lyubich's original result \cite{L}, but the proof works well without any problem. See \cite{L} or \cite[Theorem 9.1]{QRWY} for a proof.

\subsection{Yoccoz puzzle for cubic Newton maps}

The previous subsection provides the  basic  machinery of the Yoccoz puzzle theory. Further developments of these techniques  can be found in 
\cite{KSS},\cite{KS},\cite{QY} and the references therein.  Now, we  concentrate   on the settings of cubic Newton maps and state a theorem for our later use.

In \cite{Ro08},   to apply the Yoccoz puzzle theory, two kinds of graphs are constructed. We  briefly recall the constructions here.
Let $\lambda\in \Omega_0$, define
$$X_\lambda=\mathbb{\widehat{C}}\setminus \bigcup_{1\leq\varepsilon\leq 3}(\phi_\lambda^{\varepsilon})^{-1}(\mathbb{D}_{1/2}).$$

Fix a  rational angle  $\eta\in \Theta_\lambda$, it can induce a graph: 
$$Z_\lambda(\eta)=\bigcup_{j\geq 0} (\overline{R_\lambda^1}(2^j \eta)\cup \overline{R_\lambda^2}(1-2^j \eta)).$$

The graphs defining the Yoccoz puzzles are the refinement of the two graphs $G(\lambda, \zeta), I(\lambda, \kappa)$  in Section \ref{two-graph}.  They are defined as follows:
$$Y^I_\lambda(\kappa)=\partial X_\lambda \cup (X_\lambda\cap I(\lambda, \kappa)),$$
$$Y^{II}_\lambda(\zeta, \eta)=\partial X_\lambda \cup (X_\lambda\cap (G(\lambda, \zeta)\cup Z_\lambda(\eta))),$$
here, the rational angle $\eta$ is chosen so that $\zeta$ and $\eta$ have disjoint orbits under the angle doubling map.

  \begin{figure}[!htpb]
  \setlength{\unitlength}{1mm}
  {\centering
  \includegraphics[width=62mm]{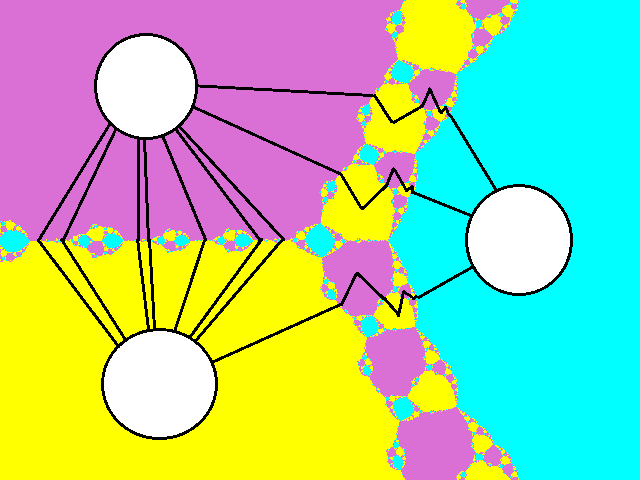}
  \includegraphics[width=62mm]{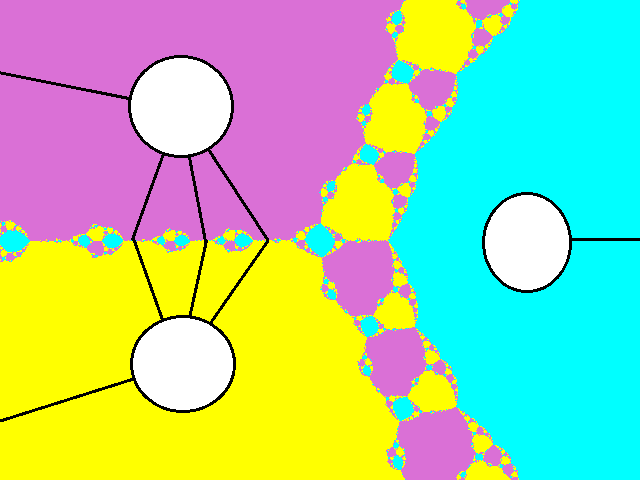}
}
  \caption{The graphs $Y^{II}_\lambda(\zeta, \eta)$ (left) and $Y^I_\lambda(\kappa)$ (right).   } 
\end{figure}

\begin{thm} \label{newton-puzzle}  1.   Fix $\lambda\in \Omega_0$, then with suitable choices of  $\zeta, \eta, \kappa$, 
at least one of the graphs $Y^I_\lambda(\kappa), Y^{II}_\lambda(\zeta, \eta)$ is an admissible puzzle.

2.  If we further assume that $N^k_\lambda(0)\in \cup_{\varepsilon=1}^3\partial B_{\lambda}^{\varepsilon}$ for some $k\geq 0$, then 
the Lebesgue measure of the Julia set $J(N_\lambda)$ is zero,   and with respect to the admissible puzzle, we have ${\rm Imp}(x)=\{x\}$ for each point $x\in J(N_\lambda)$.
Moreover, the intersection of any shrinking closed puzzle pieces is a singleton.
\end{thm}

The first statement is proven  in  \cite[Proposition 5.4]{Ro08}, the second statement follows from Theorems \ref{Y-P}, \ref{Lyu}. 


 




  

\section{Rigidity and boundary regularity of $\partial\mathcal{H}^\varepsilon_0$}\label{rigidity}

The aim of this section is to show that all the boundaries $\partial \mathcal{H}_0^\varepsilon$ are Jordan curves. By the symmetry of  the parameter space $\mathcal{X}$, it suffices to prove the results in the fundamental domain $\mathcal{X}_{FD}$. 
By Remark \ref{tran},  our task is further reduced to
  show that $\partial \mathcal{H}_0^1\cap\overline{\Omega}$ is a Jordan arc.
  
  The main ingredient of the proof is the following rigidity theorem:

\begin{thm} \label{rigidity1}  Given two parameters 
$\mathbf \lambda_1, \lambda_2\in \Omega$. Assume that   the internal rays 
$R^1_{\lambda_1}(t), R^1_{\lambda_2}(t)$ both  land at the  free critical point $0$ in the corresponding dynamical planes. Then we have $\lambda_1= \lambda_2$.
 \end{thm}
 
 When $t$ is rational, the maps $N_{\lambda_1}$ and $N_{\lambda_2}$ are post-critically finite.
  In this  case, the proof is same as that of Lemma \ref{pre-rigidity} (applying  Thurston's Theorem). We omit the details.

 Our main effort is to treat the technical case: $t$ is irrational.  In this case, the maps are  post-critically infinite and Thurston's Theorem is not available.
  Instead,  an important role of the Yoccoz puzzle theory will emerge. 
An  essential  step  in the proof of Theorem \ref{rigidity1} is a  {\it QC-criterion}, due to Kozlovski, Shen and van Strien \cite{KSS}.    
Let $U\subsetneq\mathbb{C}$ be a simply connected planar domain and
$z\in U$. The {\it shape} of $U$ about $z$ is defined by:
$${S}(U,z)=\sup_{x\in \partial U}|x-z|/\inf_{x\in \partial U}|x-z|.$$

  \begin{lem}[QC-criterion]\label{qc-c} Let $\phi:\Omega\rightarrow \tilde{\Omega}$  be a  homeomorphism
between two Jordan domains, $k\in(0,1)$ be a constant.  Let $X$ be a subset of $\Omega$ such that both $X$ and $\phi(X)$ have zero Lebesgue measures. Assume the following holds:

1. $|\bar{\partial}\phi|\leq k|{\partial}\phi|$ a.e. on  $\Omega\backslash X$.

2.  There is a constant $M>0$ such that for all  $x \in X$, there is a sequence of open topological disks $D_1 \Supset  D_2 \Supset  \cdots$ containing $x$, satisfying that

(a). $\bigcap_j \overline{D_j}=\{x\}$, and 

(b).  $\sup _jS(D_j , x) \leq M$, \   $\sup_j S(\phi(D_j ), \phi(x))< \infty.$ 

Then $\phi$ is a $K$-quasi-conformal map, where $K$ depends on $k$ and $M$.
\end{lem}

We remark that this  QC-criterion is a simplified version of \cite[Lemma 12.1]{KSS}, with a slightly difference in the second assumption (that is, we replace {\it a sequence of
 round disks} in \cite{KSS} by
{\it a sequence of disks with uniformly bounded shape}),  and the original proof goes through without any problem.

\vspace{5pt}

\noindent{\it Proof of Theorem \ref{rigidity1}  when $t$ is irrational.}  The proof consists of three steps, and the Yoccoz puzzle theory plays an important role in the proof.

\vspace{5pt}

{\textbf{Step 1 (Same Head's angle)}} {\it  We first show that $\boldsymbol{h}(\lambda_1)=\boldsymbol{h}(\lambda_2)$.}

\vspace{5pt}

If not,  without loss of generality, we assume   $0< \boldsymbol{h}(\lambda_1)<\boldsymbol{h}(\lambda_2)< 1/2$.  We claim $( \boldsymbol{h}(\lambda_1), \boldsymbol{h}(\lambda_2))\cap \Xi\neq \emptyset$. This is because if $(\boldsymbol{h}(\lambda_1), \boldsymbol{h}(\lambda_2))\cap \Xi=\emptyset $, then by Lemma \ref{char-S}, we see that   $\boldsymbol{h}(\lambda_1)$ takes the form 
$p/2^q$ and  $\boldsymbol{h}(\lambda_2)$ takes the form $p/(2^q-1)$. By Theorem \ref{H-tongue1} and Corollary  \ref{H-tongue2}, the parameter ray $\mathcal{R}^1_0(\boldsymbol{h}(\lambda_1))$ lands at $\lambda_1$.
By Lemma \ref{land}, the internal ray $R_{\lambda_1}^1(\boldsymbol{h}(\lambda_1))$ would land at $0\in \partial{B_{\lambda_1}^1}$.  This would imply that  $t=\boldsymbol{h}(\lambda_1)$ is rational, which is impossible by assumption.

Therefore, by Lemma \ref{char-S}, the open interval  $(\boldsymbol{h}(\lambda_1), \boldsymbol{h}(\lambda_2))$ contains a component of $(0,1/2)\setminus \Xi$, say $(t_1,t_2)\Subset(\boldsymbol{h}(\lambda_1), \boldsymbol{h}(\lambda_2))$.
   It's obvious that  $\lambda_1, \lambda_2$ are contained in two impressions $\mathcal{I}_{\alpha}, \mathcal{I}_{\beta}$ with $\alpha\leq t_1<t_2\leq \beta$, respectively. 
By Lemma \ref{land}, in the $\lambda_1$-dynamical plane, the internal ray $R_{\lambda_1}^1(\alpha)$ lands at $0$;  in the $\lambda_2$-dynamical plane, the internal ray $R_{\lambda_2}^1(\beta)$ lands at $0$. This contradicts our assumption.

\vspace{5pt}

{\textbf{Step 2 (Topological conjugacy)}} {\it There is  a topological conjugacy $\psi$ between $N_{\lambda_1}$ and $N_{\lambda_2}$,  which is holomorphic in the Fatou set of $N_{\lambda_1}$.} 

\vspace{5pt}

  The  construction of   $\psi$ is  based on the Yoccoz puzzle theory,  as follows:  

It's known from Step 1 that $\boldsymbol{h}(\lambda_1)=\boldsymbol{h}(\lambda_2)$ and $\Theta_{\lambda_1}=\Theta_{\lambda_2}$. By the construction of the articulated rays (see Theorem \ref{a-ray}), 
we know that if $L_{\lambda_1}(\zeta)$ is an articulated ray for $N_{\lambda_1}$, then  $L_{\lambda_2}(\zeta)$ is an articulated ray for $N_{\lambda_2}$, and vice versa.
As a consequence,  we can define the same type of Yoccoz puzzles  $Y^I_\lambda(\kappa)$ and $Y^{II}_\lambda(\zeta, \eta)$ for $\lambda=\lambda_1, \lambda_2$. 
The assumption that $R^1_{\lambda_1}(t), R^1_{\lambda_2}(t)$ both  land at $0$  implies that if $Y^I_{\lambda_1}(\kappa)$ (resp. $Y^{II}_{\lambda_1}(\zeta, \eta)$) is admissible for $N_{\lambda_1}$, then   $Y^I_{\lambda_2}(\kappa)$ (resp. $Y^{II}_{\lambda_2}(\zeta, \eta)$) is admissible for $N_{\lambda_2}$, and vice versa.

Now we fix a pair of  admissible puzzles, say  $(Y^\nu_{\lambda_1},Y^\nu_{\lambda_2})$.  For the reader's convenience, we recall some definitions from Section \ref{yoccoz}.
The {\it puzzle piece} of depth $d\geq0$, denoted by $P_{d}^\lambda$,  is a connected    component of 
$\mathbb{\widehat{C}}\setminus N_{\lambda}^{-d}(Y^\nu_{\lambda})$. 
For any  $z\in J(N_{\lambda})$, no matter whether the orbit the critical point $0$ meets the puzzle $Y^\nu_{\lambda}$, the set  
$$P_{d, \lambda}^*(z)=\bigcup_{z\in \overline{P_{d}^\lambda}} \overline{P_{d}^\lambda}$$
is always a closed topological disk.

We first construct a homeomorphism $\psi_0$ from the $\lambda_1$-dynamical plane  to the $\lambda_2$-dynamical plane in the following way:

Set  $\widehat\psi_0|_{B_{\lambda_1}^\varepsilon}= (\phi_{\lambda_2}^{\varepsilon})^{-1}\circ\phi_{\lambda_1}^{\varepsilon}, \varepsilon=1,2,3$.  This $\widehat\psi_0$ 
can be extended to $\overline{B_{\lambda_1}^\varepsilon}$ because  the boundaries $\partial B_{\lambda_1}^\varepsilon$  are Jordan curves. 
Then we are able to extend $\widehat\psi_0$ to  the  complementary  set $\mathbb{\widehat{C}}\setminus( \cup_\varepsilon \overline{B_{\lambda_1}^\varepsilon} )$, which consists  of countably many disk components.  This extension  can be made by interpolation 
 since we have already known the boundary information. 
In this way, we get an extension of $\widehat\psi_0$, say $\psi_0$.  We make an  additional  plausible requirement for $\psi_0$, that is, $\psi_0(Y^\nu_{\lambda_1})=Y^\nu_{\lambda_2}$.

Then, we can lift $\psi_0$ to $\psi_1$ so that $\psi_0\circ N_{\lambda_1}=N_{\lambda_2}\circ \psi_1$ and $\psi_1|_{\cup_\varepsilon B_{\lambda_1}^\varepsilon}=\psi_0|_{\cup_\varepsilon B_{\lambda_1}^\varepsilon}$. The lift  process is feasible 
because in the corresponding  dynamical planes, the itineraries   of  the critical point $0$ are  same  with respect to the corresponding Yoccoz puzzles.  Moreover,  we can lift infinitely many times and get a sequence of 
homeomorphisms $\psi_k$, satisfying that $\psi_k\circ N_{\lambda_1}=N_{\lambda_2}\circ \psi_{k+1}$
and  $\psi_{k+1}|_{N_{\lambda_1}^{-k}(\cup_\varepsilon B_{\lambda_1}^\varepsilon)}=\psi_k |_{N_{\lambda_1}^{-k}( \cup_\varepsilon B_{\lambda_1}^\varepsilon)}$
 for all $k\geq 0$. Note that  the  requirement $\psi_0(Y^\nu_{\lambda_1})=Y^\nu_{\lambda_2}$ implies that  $\psi_k$ preserves the puzzle pieces up to depth $k$
 (namely, $\psi_k (N_{\lambda_1}^{-k}(Y^\nu_{\lambda_1}))=N_{\lambda_2}^{-k}(Y^\nu_{\lambda_2})$).   
By  Theorem \ref{newton-puzzle}, for every point  $z\in J(N_{\lambda_1})$, the   shrinking sequence of  closed disks 
$$P_{1, \lambda_1}^*(z)\supset P_{2, \lambda_1}^*(z)\supset P_{3, \lambda_1}^*(z)\supset  \cdots$$
has impression ${\rm Imp}_{\lambda_1}(z)=\bigcap_k P_{k, \lambda_1}^*(z) =\{z\}$.
The sequence $$\psi_1(P_{1, \lambda_1}^*(z)),  \psi_2(P_{2, \lambda_1}^*(z)),  \psi_3(P_{3, \lambda_1}^*(z)), \cdots$$ is a  sequence of shrinking closed disks  in the
$\lambda_2$-dynamical plane. Again by Theorem \ref{newton-puzzle}, the intersection $\bigcap_k \psi_k(P_{k, \lambda_1}^*(z))$ consists of a single point, which is denoted by 
$\rho(z)$.

  We define a  map  $\psi:\mathbb{\widehat{C}}\rightarrow \mathbb{\widehat{C}}$ by

\begin{equation*}
\psi(z)=\begin{cases}
 \lim_{k\rightarrow \infty} \psi_k(z),\ &\text{ if }z \in \mathbb{\widehat{C}}\setminus J(N_{\lambda_1 }),\\
\rho(z),\ &\text{ if } z \in   J(N_{\lambda_1 }).
\end{cases}
\end{equation*}

It's easy to see that $\psi$ is holomorphic in the Fatou set  of $N_{\lambda_1}$. 
 The continuity and injectivity   of  $\psi$ follows from Theorem \ref{newton-puzzle}.   It's also clear that $\psi$ preserves the puzzle pieces of all depths and  $\psi$ is surjective.
 So $\psi$ is a homeomorphism. Since  $\psi$ is  conjugacy between $N_{\lambda_1}$ and $N_{\lambda_2}$  in the Fatou set, it is actually a global conjugacy, by continuity.
 
\vspace{5pt}
  
  {\textbf{Step 3 (Rigidity)}} {\it The conjugacy $\psi$ is a quasi-conformal map.} 

\vspace{5pt}

By Theorem \ref{newton-puzzle},  we know that the Julia sets $J(N_{\lambda_1})$ and $J(N_{\lambda_2})$ both have zero Lebesgue measures. To show that  $\psi$ is quasi-conformal,  
 by Lemma \ref{qc-c}, it suffices to show that there is a constant $M>0$ such that  for any point $z\in J(N_{\lambda_1})$, there is a sequence of  open topological disks $D_1 \Supset  D_2 \Supset  \cdots$ containing $x$, satisfying that

(a). $\bigcap_j \overline{D_j}=\{x\}$, and 

(b).  $\sup _jS(D_j , x) \leq M$, \   $\sup_j S(\psi(D_j ), \psi(x))< \infty.$

The proof is very similar to  \cite[Section 5.4]{QRWY}, with a  slight difference.
 For completeness and the reader's convenience, we include the details here.

We decompose the Julia set $J(N_\lambda)$ (here $\lambda$ can be either $\lambda_1$ or $\lambda_2$) into three disjoint sets:
\bess
J_\lambda^0&=&J(N_\lambda)\cap \big(\cup_{k\geq0} N_{\lambda}^{-k} (Y_\lambda^\nu)\big),\\
J_\lambda^1&=&\{z\in J(N_\lambda)\setminus J_\lambda^0;  \  0\notin \omega(z) \}, \\
J_\lambda^2&=&\{z\in J(N_\lambda)\setminus J_\lambda^0;  \  0\in \omega(z) \},
\eess
where  $\omega(z)$ is the $\omega$-limit set of $z$, defined
as  $\{y\in  J(N_\lambda);  \text{ there exist } n_k\rightarrow\infty  \text{ such that }
N_{\lambda}^{n_k}(z)\rightarrow y\}$.

\textbf{Case 1:  Points in $J_\lambda^0$. }  By the construction of the Yoccoz puzzle $Y_\lambda^\nu$, 
the set $J_\lambda^0$ is countable, forward invariant (that is $N_\lambda(J_\lambda^0)\subset J_\lambda^0$), and each point in $J_\lambda^0$ is preperiodic. 
Note also $J_\lambda^0$ contains only finitely many periodic  points, all are repelling.

Let  $z\in J_\lambda^0$ be a repelling periodic point with period say $p$, then there are two small topological disks $U, V$ in a linearizable  neighborhood of $z$, both containing $z$ so that 
$N_{\lambda}^p(U)=V$ and $U\Subset V$.  By pulling back $V$ via $N_{\lambda}^{kp}$,  
we get  a sequence of neighborhoods of $z$, say
$$V_0(z)\Supset V_1(z)\Supset V_2(z)\Supset \cdots$$
with $V_0(z)=V$, such that $N_{\lambda}^{kp}: V_k(z)\rightarrow V_0(z)$ is  a conformal map for all $k\geq 1$.
By the shape  distortion \cite[Lemma 6.1]{QWY},
$$S(V_k(z), z)\leq C(m_z) \cdot S(V_1(z),z), \ \forall  k\geq 2, $$
where $C(m_z)$ is a constant depending only on $m_z={\rm mod}(V_0(z)\setminus \overline{V_1}(z))$.

For any aperiodic point $z\in J_\lambda^0$, there is a smallest number $\ell\geq1$ so that $N_\lambda^\ell(z)$ is periodic. 
Then by pulling back the sequence of neighborhoods of $V_k(N_\lambda^\ell(z))$ of $N_\lambda^\ell(z)$, we get a sequence of neighborhoods $V_k(z)$ of $z$.
Again by the shape distortion, we get 
$$S(V_k(z), z)\leq C(m_{N_\lambda^\ell(z)})  \cdot S(V_1(N_\lambda^\ell(z)), N_\lambda^\ell(z)).$$
 For any periodic point $z\in J_{\lambda_1}^0$, set  $m_z^*={\rm mod}\psi((V_0(z)\setminus \overline{V_1}(z))$. 
Take  
\bess & M_0= \max \{ C(m_z) \cdot S(V_1(z),z);  z\in J_{\lambda_1}^0 \text{is  periodic}\},& \\
 &{M}_0^*=\max\{ C(m_z^*) \cdot S(\psi(V_1(z)),\psi(z));  z\in J_{\lambda_1}^0 \text{ is  periodic}\},&
\eess
then  for all $z\in J_{\lambda_1}^0$, 
 $$\sup _jS(V_j(z) , z) \leq M_0, \   \sup_j S(\psi(V_j(z) ), \psi(z))\leq {M}^*_0<\infty.$$

 \textbf{Case 2:  Points in $J_\lambda^1$. }  In this case,  there is a integer   $d_0\geq 0$ such that $N_\lambda^k(z)\notin P_{d_0}^\lambda(0)$ for all $k\geq 1$.  Then 
 we can find a sequence of integers $k_j$ and  a point $w\in J(N_\lambda)\setminus J_\lambda^0$ so that $N_{\lambda}^{k_j}(z)\rightarrow w$ as $j\rightarrow \infty$.
 By  passing to a subsequence,  we assume that $N_{\lambda}^{k_j}(z)\in P_{d_0}^\lambda(w)$ for all $j$.  It's clear that the degree of 
 $N_\lambda^{k_j}: P_{d_0+k_j}^\lambda(z)\rightarrow P_{d_0}^\lambda(w)$ is at most two.
 Take a small number $r>0$ so that ${\rm mod}(P_{d_0}^\lambda(w)\setminus\overline{D_w^r})\geq 1$,  where
$D_w^r$ is the open Euclidean disk centered at $w$ with radius $r$. When $j$ is large, we have $N_{\lambda}^{k_j}(z)\in D_w^{r/2}$, let $V_j(z)$ be the component of 
 $N_{\lambda}^{-k_j}(D_w^{r})$ containing $z$. Then by the shape distortion \cite[Lemma 6.1]{QWY},  for large $j$, 
 $$S(V_j(z),z)\leq C \cdot S(D_w^r, N_{\lambda}^{k_j}(z))\leq C\cdot \frac{r+r/2}{r-r/2}=3C,$$
 here $C$ is a universal constant.  
  
  Note that the above shape control holds when  $\lambda=\lambda_1$, and the upper bound is a universal constant. When $\lambda=\lambda_2$, the upper bound  of the shape distortion depends on $z$, this can be seen from the following estimate:
  \bess S(\psi(V_j(z)), \psi(z))&\leq&  C_1 \cdot S(\psi(D_w^r), \psi(N_{\lambda_1}^{k_j}(z)))\\
  &\leq& C_1 \sup_{\zeta\in \psi(D_w^{r/2})}S(\psi(D_w^r), \zeta)<\infty,\eess
  where $C_1$ depends on the modulus of $\psi(D_w^r\setminus\overline{D_w^{r/2}})$, which turns out to be related
  to  $z\in J_{\lambda_1}^1$.

   \textbf{Case 3:  Points in $J_\lambda^2$. }     We first look at the free critical point $0$. If $0\notin \omega(0)$, 
   then  we may choose a sequence of topological disks  $V_j(0)$  with  uniformly  bounded shape as Case 2. 
    For any $z\in J_{\lambda}^2$,   just as the proof of Case 2, by pulling back the sequence of $V_j(0)$,
    we  get a sequence of disks $V_j(z)$ surrounding $z$ with uniformly bounded shapes.

   The case $0\in \omega(0)$ is more delicate and it is our main focus. In this case, for each $d\geq 0$, let $\tau_Y(d)\in [0,d]$  be first integer 
   $k$ such that $0\in N_{\lambda}^{d-k}(P_d^\lambda(0))$.    The function $\tau_Y$ is called  {\it the Yoccoz $\tau$-function}. It satisfies $\tau_Y(d+1)\leq
   \tau_Y(d)+1$.
   It's clear that the assumption $0\in \omega(0)$ implies that $\limsup_d \tau_Y(d)=\infty$.
   There are two possibilities for $\liminf_d \tau_Y(d)$, either $\liminf_d \tau_Y(d)\leq L<\infty$ or 
   $\liminf_d \tau_Y(d)=\infty$. We treat these two cases separately. 
   
   \textbf{Case 3.1 $\liminf_d \tau_Y(d)\leq L<\infty$. }      Choose $n_2>n_1>n_0\geq L$ such that $P_{n_2}^\lambda(0)\Subset 
   P_{n_1}^\lambda(0)\Subset P_{n_0}^\lambda(0)$.  We see that $\tau_Y^{-1}(n_0)$ is an infinite set, and we write
   $\tau_Y^{-1}(n_0)=\{k_1, k_2, \cdots \}$. By the definition of $\tau_Y$, the map $N_{\lambda}^{k_j}: P_{n_0+k_j}^\lambda(0)\rightarrow
  P_{n_0}^\lambda(0)$ has degree two, for all $j$.
  For each $j$, let $l_j\geq 0$ be the smallest integer such that $N_\lambda^{l_j}(N_\lambda^{k_j}(0))\in
   P_{n_2}^\lambda(0)$. Then for all $j$, the degree of $N_{\lambda}^{k_j+l_j}: P_{n_0+k_j+l_j}^\lambda(0)\rightarrow 
    P_{n_0}^\lambda(0)$  is at most $2\cdot 2^{n_2-n_0}$. 
    For any $z\in J_\lambda^2$,  and any $j\geq 0$, let $m_j\geq 0$ be the first integer such that $N_{\lambda}^{m_j}(z)\in P_{n_2+k_j+l_j}^{\lambda}(0)$, with the same discussion as above, we see that  
    $${\rm deg}(N_{\lambda}^{m_j}: P_{n_0+k_j+l_j+m_j}^\lambda(z)\rightarrow 
     P_{n_0+k_j+l_j}^\lambda(0))\leq 2\cdot 2^{n_2-n_0}.$$
     This implies that, for all $j$,  
    $${\rm deg}(N_{\lambda}^{m_j+k_j+l_j}: P_{n_0+k_j+l_j+m_j}^\lambda(z)\rightarrow 
     P_{n_0}^\lambda(0))\leq 4\cdot 4^{n_2-n_0}.$$
     By the shape distortion \cite[Lemma 6.1]{QWY},
     $$S(P_{n_1+k_j+l_j+m_j}^\lambda(z), z)\leq C_2 \cdot S(P_{n_1}^\lambda(0), 0)<+\infty,$$
     here $C_2$ depends on ${\rm mod}(P_{n_0}^\lambda(0)\setminus\overline{P_{n_1}^\lambda(0)})$, independent of
      $z\in J_{\lambda}^2$ and $j$.
     
       \textbf{Case 3.2 $\liminf_d \tau_Y(d)=\infty$. }  By a very technical construction 
       of enhanced nests,  we have the following property: 
       
       {\it There exist a constant $m>0$ and  a sequence of 
       critical puzzle pieces  
       $$P_{d_1}^\lambda(0)\Supset P_{l_1}^\lambda(0)
       \Supset P_{k_1}^\lambda(0)\Supset P_{d_2}^\lambda(0)\Supset P_{l_2}^\lambda(0)\Supset P_{k_2}^\lambda(0)\cdots$$
       satisfying that 
       
       (a) $\bigcap  \overline{P_{d_j}^\lambda(0)}=\{0\}$;
       
       (b) ${\rm mod}(P_{d_j}^\lambda(0)\setminus \overline{P_{l_j}^\lambda(0)}) \geq m, \ \
        {\rm mod}(P_{l_j}^\lambda(0)\setminus \overline{P_{k_j}^\lambda(0)})  \geq m, \ \forall j \geq 1$ ;
        
        (c) Both $P_{d_j}^\lambda(0)\setminus \overline{P_{l_j}^\lambda(0)}$ and
         $P_{l_j}^\lambda(0)\setminus \overline{P_{k_j}^\lambda(0)}$ avoid the free critical orbit, for all $j$.}
         
         See \cite[Section 8]{KSS} for the enhanced nest construction in a more general setting, see \cite{KS} and 
         \cite{QY} for the proof of the complex bounds.
         
         Then by \cite[Proposition 1]{YZ}, there is a constant $M>0$ such that $S(P_{l_j}^\lambda(0), 0)\leq M$ for all $j$.
         For any $z\in J_\lambda^2$ and any $j\geq 1$, let $m_j\geq 0$ be the first  integer such that 
         $N_{\lambda}^{m_j}(z)\in P_{k_j}^\lambda(0)$. It follows  that ${\rm deg}(N_{\lambda}^{m_j}: 
         P_{m_j+k_j}^\lambda(z)\rightarrow P_{k_j}^\lambda (0))\leq 2$. 
         By   the above property (c),  we have
         $${\rm deg}(N_{\lambda}^{m_j}: 
         P_{m_j+d_j}^\lambda(z)\rightarrow P_{d_j}^\lambda (0))\leq 2.$$
         By the shape distortion, we have 
         $$S(P_{m_j+l_j}^\lambda(z),z)\leq C_3 \cdot S(P_{l_j}^\lambda(0),0)\leq C_3M,$$
         here the constant $C_3$ depends on $m$ but independent of $z$ and $j$.
         
         \
   
   This completes the proof of Step 3. Since $\psi$ is conformal on the Fatou set and the
    Julia set has zero Lebesgue measure (see Theorem \ref{newton-puzzle}), this  $\psi$ is a M\"obius map. Note that $\lambda_1, \lambda_2$ are in the 
    fundamental   domain,   we have  $\lambda_1=\lambda_2$, completing the proof of the theorem. 
    \hfill $\Box$

\begin{thm} \label{jordan} The boundary $\partial \mathcal{H}_0^1$ is a Jordan curve.
\end{thm}
\begin{proof}   We first show that  $\partial \mathcal{H}_0^1$ is locally connected. For $t\in[0,1/2]$,  recall that $\mathcal{I}_t$  is the impression  of the parameter ray
$\mathcal{R}_0^1(t)$.  If $\mathcal{I}_t\cap \Omega\neq \emptyset$, then we take two parameters $\lambda_1, \lambda_2\in \mathcal{I}_t\cap \Omega$ (if any)
 so that $N_{\lambda_1}$ and $N_{\lambda_2}$ have no parabolic cycles. By Lemma \ref{land}, the internal rays $R_{\lambda_1}^1(t), R_{\lambda_2}^1(t)$ both land at the 
 free critical point $0$, in the corresponding dynamical planes. By Theorem \ref{rigidity1}, we see that $\lambda_1=\lambda_2$. Therefore  
 \bess \mathcal{I}_t&=&(\mathcal{I}_t\cap{\partial \Omega}) \cup(\mathcal{I}_t\cap \Omega)\\ 
 &\subset&\{0, \sqrt{3}i/2\}\cup\{\text{parabolic parameters}\}\cup\{\text{singleton}\},
 \eess
which means the continuum  $\mathcal{I}_t$ is contained in  a countable set.  So  $\mathcal{I}_t$ is a singleton.  
Since  $t\in[0,1/2]$ is arbitrary, this means that $\partial \mathcal{H}_0^1\cap \overline{\Omega}$ (hence $\partial \mathcal{H}_0^1$) is locally connected.

 In the following, we will show that if two 
parameter rays $\mathcal{R}_0^1(t_1), \mathcal{R}_0^1(t_2)$ with $t_1,t_2\in [0,1/2]$  land  at the same point $\lambda$,  then $t_1=t_2$. This would
 imply that $\partial \mathcal{H}_0^1$ is a Jordan curve.

We first show that if one of $t_1,t_2$ is $0$ (resp. $1/2$), then the other would be  $0$ (resp. $1/2$). To see this, we only consider the case $t_1=0$, 
the same discussion works for the other case.  
Note that $0$ is an accumulation point of the set $\partial(\Xi)$ (see Fact \ref{angle-Xi} for its definition).  If $t_2\in (0,1/2]$, we can find an angle $t_*\in \partial(\Xi)$ lying in between $0$ and $t_2$. By Theorem \ref{H-tongue1}, the parameter rays $\mathcal{R}_0^1(0), \mathcal{R}_0^1(t_2)$ are contained in different components of $\mathbb{C}\setminus ([-1/2,1/2]\cup \overline{\mathcal{R}_0^1}(t_*)\cup \overline{\mathcal{R}_0^2}(1-t_*))$, contradicting that $\mathcal{R}_0^1(0)$ and $\mathcal{R}_0^1(t_2)$ land at the same point. So we must have $t_2=t_1=0$.

Now, it  suffices to assume that   $t_1,t_2\in (0,1/2)$.  This assumption implies that $\lambda\in\Omega$.  By  Lemma \ref{land}, we know that in the dynamical plane, the internal rays $R_\lambda^1(t_1)$ and $R_\lambda^1(t_2)$ would land at the same point. 
It follows from  Theorem \ref{roesch-jordan} that $t_1=t_2$, completing the proof.
\end{proof}







\section{Boundaries of Capture domains} \label{cap-jordan}
Let $\mathcal{H}\subset \Omega$ be a capture domain of level $k\geq2$. That is, it is a component of $\mathcal{H}_k^\varepsilon$ for some $\varepsilon\in\{1,2,3\}$.
By Theorem \ref{1b}, the map  $\Phi_\mathcal{H}:\mathcal{H}\rightarrow{\mathbb{D}} $ defined by
$\Phi_\mathcal{H}(\lambda)=\phi_\lambda^\varepsilon(N_\lambda^{k}(0))$  parameterizes  $\mathcal{H}$.
The {\it parameter ray} $\mathcal{R}_{\mathcal{H}}(t)$  in $\mathcal{H}$,  with angle $t\in\mathbb{S}=\mathbb{R}/\mathbb{Z}$,  is defined by $$\mathcal{R}_\mathcal{H}(t)=\Phi_\mathcal{H}^{-1}(\{re^{2\pi i t}; 0<r<1\}).$$

For any $t\in\mathbb{S}$ and any integer $j\geq1$, we define  an open sector  as follows:
$$\mathcal{S}_{\mathcal{H},j}(t)=\Phi_\mathcal{H}^{-1}(\{re^{2\pi i \theta}; r\in(1-1/j,1), \theta\in(t-1/j,t+1/j)\}).$$

The {\it impression}  $\mathcal{I}_{\mathcal{H}}(t)$  of the parameter ray  $\mathcal{R}_{\mathcal{H}}(t)$  is defined by
$$\mathcal{I}_{\mathcal{H}}(t)=\bigcap_{j\geq1}\overline{\mathcal{S}_{\mathcal{H}, j}(t)}.$$
It's a standard fact that $\mathcal{I}_{\mathcal{H}}(t)$ is a connected and compact set.  Our goal in this section is to show that $\mathcal{I}_{\mathcal{H}}(t)$ is a singleton, which implies that $\partial{\mathcal{H}}$ is locally connected.

Before further  discussion, we give an observation for $\mathcal{H}$. 
 By Corollary \ref{H-tongue2},  all the maps in $\mathcal{H}$ have the same Head's angle $\alpha$ of the form ${p\over2^n-1}$, 
 where $(\beta, \alpha)$ is some connected component of $(0,1/2)\setminus\Xi$.  Therefore $\mathcal{H}$ is contained in 
 $V(\beta, \alpha)\cap \Omega_0$, see Lemma \ref{Head-bounds}.  In particular, we have $\mathcal{\overline{H}}\subset\Omega$.


 When  $\lambda\in\mathcal{H}$, we define the  set $U_\lambda$ to be  the Fatou component of $N_\lambda$ containing the free critical point $0$.  Clearly,  the {\it center} $c_\lambda$ of  $U_\lambda$, defined as  the unique point $c_\lambda\in  U_\lambda$ satisfying $N_\lambda^{k}(c_\lambda)=b_\varepsilon(\lambda)$,   moves continuously with respect to   $\lambda\in\mathcal{{H}}$.  It's also obvious that the center map $\lambda\mapsto c_\lambda$ has a continuous extension to $\partial\mathcal{H}$. Therefore,  when $\lambda\in \partial\mathcal{H}$, the point $c_\lambda$ is well-defined and  we define $U_\lambda$  to be the unique Fatou component of $N_\lambda$ containing $c_\lambda$.
 In this way,  the set $U_\lambda$ is designated  for all $\lambda \in\mathcal{\overline{H}}$.

  \begin{lem}\label{5b} For any $t\in[0,1)$ and any $\lambda\in\mathcal{I}_{\mathcal{H}}(t)\setminus(\partial\mathcal{H}_0^1\cup
  \partial\mathcal{H}_0^2)$,  we have $0\in \partial U_\lambda$
and the internal ray $R_{U_\lambda}(t)=(N_\lambda^k|_{U_\lambda})^{-1}(R_\lambda^\varepsilon(t))$ lands at $0$.
\end{lem}
\begin{proof} Note that for any $\lambda\in\mathcal{I}_{\mathcal{H}}(t)\setminus(\overline{\mathcal{H}_0^1\cup \mathcal{H}_0^2})\subset\Omega$,  there is a disk neighborhood  $\mathcal{U}$ of $\lambda$ contained in  $\Omega\setminus(\overline{\mathcal{H}_0^1\cup \mathcal{H}_0^2})$.

 We first claim that for any $\varepsilon\in \{1,2,3\}$, the boundary $\partial B_u^\varepsilon$ moves holomorphically  with respect to $u\in\mathcal{U}$. To see this, fix  some $u_0\in \mathcal{U}$, we define a map $h: \mathcal{U}\times B_{u_0}^\varepsilon\rightarrow \mathbb{\widehat{C}}$ by
 $h(u, z)=(\phi_{u}^\varepsilon)^{-1}\circ \phi_{u_0}^\varepsilon(z)$.  It satisfies:

 (1). Fix any $z\in B_{u_0}^\varepsilon$, the map $u\mapsto h(u, z)$ is holomorphic;

 (2). Fix any $u\in \mathcal{U}$, the map $z\mapsto h(u, z)$ is injective;

 (3). $h(u_0, z)=z$ for all $z\in B_{u_0}^\varepsilon$.

 These properties imply that $h$ is a holomorphic motion parameterized by $\mathcal{U}$,  with base point $u_0$. By the Holomorphic Motion Theorem (see \cite{GJW} or \cite{Sl}), there is a holomorphic motion $H:\mathcal{U}\times
\mathbb{\widehat{C}}$ extending $h$ and for any $u\in \mathcal{U}$,
we have $H(u,\partial B_{u_0}^\varepsilon)=\partial B_u^\varepsilon$. Therefore  $\partial B_u^\varepsilon$ moves holomorphically  with respect to $u\in\mathcal{U}$. The claim is proved.

It follows that  for any $p\geq 0$, the set $N_u^{-p}(\partial B_u^\varepsilon)$ moves continuously in Hausdorff topology
 with respect to $u\in\mathcal{U}$.
  By the assumption  $\lambda\in\mathcal{I}_{\mathcal{H}}(t)\setminus(\partial\mathcal{H}_0^1\cup
  \partial\mathcal{H}_0^2)$, there exist a sequence of parameters $\{u_n\}$ in $\mathcal{H}$ and  a sequence of angles $\{t_n\}$, such that
  $$u_n\rightarrow\lambda, \ t_n\rightarrow t \text{ as } n\rightarrow\infty,  \text{ and } N_{u_n}^k(0)\in R_{u_n}^\varepsilon(t_n)\subset B_{u_n}^\varepsilon \text{ for all } n.$$


By the continuity of $u\mapsto \partial B_u^\varepsilon$  (which also implies the continuity of the internal rays with respect to the parameter), we have $N_\lambda^{k}(0)\in \partial B_\lambda^\varepsilon$
and the internal ray $R_{\lambda}^\varepsilon(t)$ lands at $N_\lambda^{k}(0)$.   By the continuity of $u\mapsto  N_u^{-k}(\partial B_u^\varepsilon)$
 for $u\in \mathcal{U}$  and the fact $N_\lambda^{k}(R_{U_\lambda}(t))=R_{\lambda}^\varepsilon(t)$, we have that $0\in \partial U_\lambda$ and
 $R_{U_\lambda}(t)$ lands at $0$.
\end{proof}

To show that $\partial{\mathcal{H}}$ is a Jordan curve, we need  two lemmas, whose proofs   are very technical.

\begin{lem} \label{imp-h} For each $t$, let  $\mathcal{I}_{\mathcal{H}}^*(t)= \mathcal{I}_{\mathcal{H}}(t)\setminus(\partial\mathcal{H}_0^1\cup\partial\mathcal{H}_0^2)$. Then 
$\mathcal{I}_{\mathcal{H}}^*(t)$  is either empty or a singleton.
\end{lem}
\begin{proof}  It's clear that $\mathcal{I}_{\mathcal{H}}^*(t)\subset \overline{\mathcal{H}} \subset \Omega_0=\Omega\setminus(\mathcal{H}_0^1\cup\mathcal{H}_0^2)$.  Recall that  $\mathcal{H}$ is a  component of $\mathcal{H}_k^\varepsilon$ for some $\varepsilon\in\{1,2,3\}$.

We shall prove the lemma by contradiction.   If it is not true, then there exist   a connected and compact subset $\mathcal{E}$ of $\mathcal{I}_{\mathcal{H}}^*(t)$  containing at least two points.
  By Lemma \ref{5b}, the internal ray $R_{U_\lambda}(t)$ lands at $0$ for all $\lambda\in \mathcal{E}$.
It is worth observing that for all $\lambda\in \mathcal{E}$, we have
$N_\lambda^{k-1}(0)\notin \partial B_\lambda^\varepsilon$ and
 $N_\lambda^{k}(0)\in  \partial B_\lambda^\varepsilon$. (To see this, if $N_\lambda^{k-1}(0)\in \partial B_\lambda^\varepsilon$, then $N_\lambda^{k-1}(0)$ would be a common  boundary point of  $N_\lambda^{k-1}(U_\lambda)$ and $N_\lambda^{k}(U_\lambda)=B_\lambda^\varepsilon$.
 It turns out that $N_\lambda^{k-1}(0)$ is a critical point of $N_\lambda$.  Necessarily, we have $N_\lambda^{k-1}(0)=0$. Contradiction.) So by  continuity, there is  disk neighborhood
 $\mathcal{D}\subset \Omega_0$ of $\mathcal{E}$ such that for all $\lambda\in \mathcal{D}$, we have $N_\lambda^{k-1}(0)\notin  \overline{ B_\lambda^\varepsilon}$.

Now take two different parameters $\lambda_1,\lambda_2\in \mathcal{E}$.   Suppose that  $\{1,2,3\}$ can be rewritten as $\{\varepsilon,\varepsilon_1,\varepsilon_2\}$. For $l=1,2$,  we define a subset  $Z_{\lambda_1}^l$ of
 $B^{\varepsilon_l}_{\lambda_1}$ to be $\{z\in B^{\varepsilon_l}_{\lambda_1}; |\phi^{\varepsilon_l}_{\lambda_1}(z)|<1/2\}.$
Let
 $$J=\{N_{\lambda_1}^j(0); 0\leq j \leq k-1\}\cup {B_{\lambda_1}^\varepsilon}\cup Z_{\lambda_1}^1\cup Z_{\lambda_1}^2.$$ It's clear that its closure $\overline{J}$ contains the post-critical set of $N_{\lambda_1}$.
We define a continuous map $h: \mathcal{D}\times J\rightarrow \mathbb{\widehat{C}}$ satisfying:


 1.  $h(\lambda, z)=(\phi_\lambda^\varepsilon)^{-1}\circ \phi_{\lambda_1}^\varepsilon(z)$ for all $(\lambda,z)\in \mathcal{D}\times {B_{\lambda_1}^\varepsilon}$;

 2.  $h(\lambda, z)=(\phi_\lambda^{\varepsilon_l})^{-1}\circ \phi_{\lambda_1}^{\varepsilon_l}(z)$ for all $(\lambda,z)\in \mathcal{D}\times Z_{\lambda_1}^l, \ l=1,2$;

 3.  $h(\lambda, N_{\lambda_1}^j(0))=N_{\lambda}^j(0)$  for all $\lambda\in \mathcal{D}$ and  $0\leq j\leq k-1$.

By definition, $z\mapsto h(\lambda_1, z)$ is the identity map. The map $h$ is a holomorphic motion parameterized by $\mathcal{D}$, with base point $\lambda_1$. By the Holomorphic Motion  Theorem (see \cite{GJW,Sl}), there is a holomorphic motion $H:\mathcal{D}\times
\mathbb{\widehat{C}}\rightarrow\mathbb{\widehat{C}}$ extending $h$. We consider the restriction $H_0=H|_{\mathcal{E}\times
\mathbb{\widehat{C}}}$ of $H$.
  Note that fix any  $\lambda\in \mathcal{E}$, the map  $z\mapsto H(\lambda,z)$ sends
  the post-critical set of $N_{\lambda_1}$ to that of $N_{\lambda}$, preserving the dynamics on this set.
  By the  lifting property, there is a
  unique continuous map  $H_1: \mathcal{E}\times \mathbb{\widehat{C}}\rightarrow \mathbb{\widehat{C}}$ such that
$N_\lambda(H_1(\lambda,z))=H_0(\lambda, N_{\lambda_1}(z))$ for all $(\lambda,z)\in\mathcal{E}\times\widehat{\mathbb{C}}$
and
   $H_1(\lambda_1,\cdot)\equiv id$.

 Set $\psi_0=H_0(\lambda_2, \cdot)$ and $\psi_1=H_1(\lambda_2, \cdot)$.
  Both  $\psi_0$ and $\psi_1$ are quasi-conformal maps, satisfying  $N_{\lambda_2}\circ \psi_{1}=\psi_0\circ N_{\lambda_1}$.
   One may verify that $\psi_0$ and $\psi_1$ are isotopic rel $\overline{J}$.
Again by the lifting property, there is a sequence of quasi-conformal maps $\psi_j$ satisfying that

 (a).  $N_{\lambda_2}\circ \psi_{j+1}=\psi_j\circ N_{\lambda_1}$
for all $j\geq0$,

(b). $\psi_{j+1}$ and $\psi_{j}$ are isotopic rel $N_{\lambda_1}^{-j}(\overline{J})$.


The maps $\psi_j$ form a normal family since their dilatations are
uniformly bounded above.  Let  $\psi_\infty$ be the limit map of  $\psi_j$.
It is quasi-conformal in $\mathbb{\widehat{C}}$,  holomorphic in the Fatou set $F(N_{\lambda_1})=\bigcup_{j\geq 0} N_{\lambda_1}^{-j}(B_{\lambda_1}^\varepsilon \cup Z_{\lambda_1}^1\cup Z_{\lambda_1}^2)$ and  satisfies
$N_{\lambda_2}\circ \psi_{\infty}=\psi_\infty\circ N_{\lambda_1}$ in $F(N_{\lambda_1})$. By continuity, we have $N_{\lambda_2}\circ \psi_{\infty}=\psi_\infty\circ N_{\lambda_1}$ in $\mathbb{\widehat{C}}$.
By Theorem \ref{newton-puzzle},
 the  Lebesgue measure of  the Julia set $J(N_{\lambda_1})$ is zero.
Therefore  $\psi_\infty$ is a
M\"obius map.
Since both  $\lambda_1$ and $\lambda_2$ are  contained in  the fundamental domain $\mathcal{X}_{FD}$, we have
$\lambda_1=\lambda_2$.
This  contradicts the assumption that  $\lambda_1\neq \lambda_2$.  
\end{proof}

It follows from Lemma \ref{imp-h}  that  the impression $\mathcal{I}_{\mathcal{H}}(t)$ is either a singleton or
 contained in
   $\partial\mathcal{H}_0^1\cup\partial\mathcal{H}_0^2$.  To analyze the latter case, we  shall prove

\begin{lem} \label{imp-2} If $\mathcal{I}_{\mathcal{H}}(t)\subset \partial\mathcal{H}_0^1\cup\partial\mathcal{H}_0^2$, then $\mathcal{I}_{\mathcal{H}}(t)$ is a singleton. 
\end{lem}

\begin{proof}  
Let $\lambda \in \mathcal{I}_{\mathcal{H}}(t)$ be a parameter so that $N_\lambda$ has no parabolic cycle. By assumption, 
 either $\lambda \in \partial\mathcal{H}_0^1$ or  $\lambda \in \partial\mathcal{H}_0^2$. 
 
 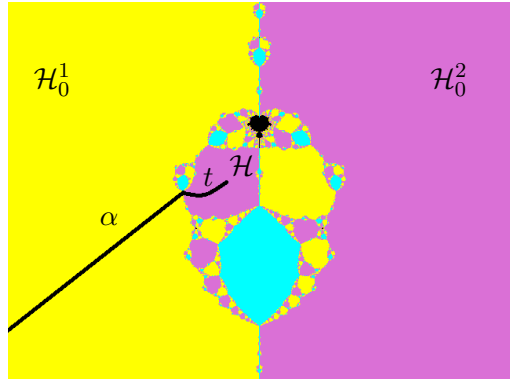
\begin{figure}[h]
\begin{center}
\includegraphics[height=5cm]{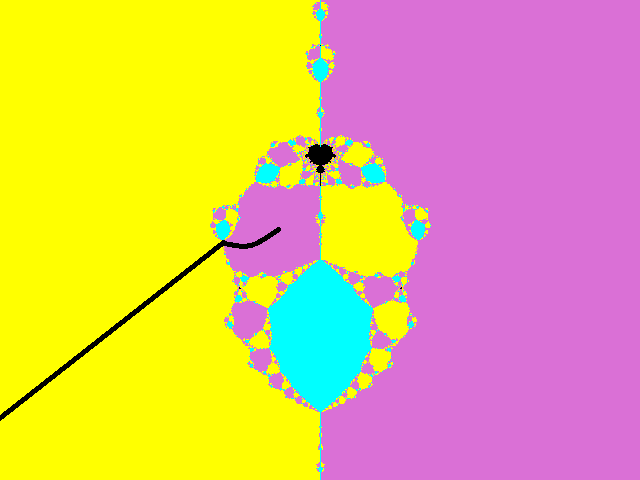}
 \put(-155,58){$\alpha$}  \put(-116,73){$t$}  \put(-106,77){$\mathcal{H}$} 
\put(-180,110){$\mathcal{H}_0^1$}    \put(-30,110){$\mathcal{H}_0^2$}       
 \caption{A case that $\mathcal{I}_{\mathcal{H}}(t)\cap \partial\mathcal{H}_0^1\neq \emptyset$. }
\end{center}\label{f5}
\end{figure}

 If $\lambda \in \partial\mathcal{H}_0^1$, 
 by Theorem \ref{pri},  the free critical point $0$ of $N_\lambda$  is  on $\partial B_\lambda^1$. 
Let  $R_\lambda^1(\alpha)$ be the internal ray  landing at $0$, we will show that 
$$2^k\alpha= t  \  {\rm mod} \ \mathbb{Z},$$
where we recall that  $k$ is the level of $\mathcal{H}$.

We prove the assertion by contradiction, mimicking the proof of Lemma \ref{land}. If  the above equation is not true,  then by Fact \ref{stability}, 
there is a graph $G_\lambda$ avoiding the free critical orbit of $N_\lambda$, an integer $q\geq0$  and a neighborhood $\mathcal{V}'$ of $\lambda$, satisfying that

(a).   $G_u$ is well-defined and moves continuously when   $u\in\mathcal{V}'$;
 
(b).  $N_u^{-q-k}(G_u)$ avoids the free critical point $0$ for all $u\in\mathcal{V}'$;

(c).   when $u=\lambda$, the graph $N_\lambda^{-q}(G_\lambda)$ separates  $R_\lambda^1(2^k\alpha)$ and $R_\lambda^1(t)$.

The third property (c) implies that $N_\lambda^{-q}(G_\lambda)$ also separates $R_\lambda^1(2^k\alpha)$ and a sector neighborhood of   $R_\lambda^1(t)$.  The sector neighborhood can be chosen in the following way: we can first  choose  rational angles $t_1, t_2$ such that

1.   $t_1<t<t_2$ are in counter clockwise order.
 
 2.  The internal rays  with angles $t_1,t,t_2$  are in the same component of $\mathbb{\widehat{C}}\setminus N_\lambda^{-q}(G_\lambda)$.

3. The closure of the  internal rays  $R_u^1(\theta)$ with $\theta\in\{t_1,t_2\}$ are well-defined,  avoiding  $0, N_u(0),\cdots, N_u^{q+k}(0)$  and moves continuously   in a neighborhood $\mathcal{V}_0\subset \mathcal{V}'$ of $\lambda$ (by suitable choices of the angles and  Fact \ref{con}).

Let $S_{B_u^1}(t_1, t_2)$  be the open sector which  contains $(\phi_u^1)^{-1}\{(0,1/2)e^{2\pi i t}\}$ and  bounded by $\partial B_u^1$, $\overline{R_u^1}(\theta), \theta\in\{t_1,t_2\}$. 
By the continuity  of $N_u^{-q-k}(G_u)$, we see that  for all $u\in \mathcal{V}_0$,  the free critical point  $0$ and $(\phi_u^1)^{-1}\{(0,1/2)e^{2\pi i \alpha}\}$ are contained in the same component, say $D_u$,  of  $\mathbb{\widehat{C}}\setminus N_u^{-q-k}(G_u)$,  and $D_u\cap N_u^{-k}(S_{B_u^1}(t_1, t_2))=\emptyset$.
However, the assumption $\lambda\in \mathcal{I}_{\mathcal{H}}(t)$ implies that when $u\in \mathcal{H}\cap \mathcal{V}_0\cap \mathcal{S}_{\mathcal{H},j}(t)$ with $1/j
<\min\{|t_1-t|, |t_2-t|\}$, the free critical point $0\in N_u^{-k}(S_{B_u^1}(t_1, t_2))$. Contradiction. This proves  that $2^k\alpha= t  \  {\rm mod} \ \mathbb{Z}$.

If $\lambda \in \partial\mathcal{H}_0^2$, again by Theorem \ref{pri},  the free critical point $0$ of $N_\lambda$  is  on $\partial B_\lambda^2$. Let  $R_\lambda^2(\beta)$ be the internal ray  landing at $0$, using the same argument as above,  we can show that 
$$2^k\beta= t  \  {\rm mod} \ \mathbb{Z}.$$


 

Therefore each  $\lambda\in \mathcal{I}_{\mathcal{H}}(t)$ either corresponds to   a map $N_\lambda$   having  a parabolic  cycle  or is contained in  the following finite  set
$$\{ \text{the landing point of the parameter ray } \mathcal{R}_0^\varepsilon(\alpha); 2^k\alpha= t, \ \varepsilon=1,2\}.$$
So  $\mathcal{I}_{\mathcal{H}}(t)$ is at most a countable set.  The connectivity of  $\mathcal{I}_{\mathcal{H}}(t)$ implies that it is a singleton. 
  \end{proof}
Now we are ready to prove the main result of this section:

\begin{thm} \label{capture-jordan} $\partial \mathcal{H}$ is a Jordan curve.
\end{thm}
\begin{proof}  We know from the previous two lemmas  that  $\partial \mathcal{H}$ is locally connected.  Assume by contradiction that 
there are two  parameter rays  $\mathcal{R}_{\mathcal{H}}(t_1),  \mathcal{R}_{\mathcal{H}}(t_2)$  with $t_1\neq t_2$, landing at the same point $\lambda\in\partial \mathcal{H}$. Let's look at the dynamical plane of $N_\lambda$, it follows from Theorem \ref{roesch-jordan}  that the internal rays $R_{U_\lambda}(t_1)$ and $R_{U_\lambda}(t_2)$ would land at two different points.  Using the same idea of  proof as  Lemma \ref{imp-2},  by Fact \ref{stability},
there is a graph $G_\lambda$ avoiding the free critical orbit of $N_\lambda$, an integer $q\geq0$  and a neighborhood $\mathcal{V}'$ of $\lambda$ satisfying that

(a).   $G_u$ is well-defined and continuous for   $u\in\mathcal{V}'$;
 
(b).  $N_u^{-q}(G_u)$ avoids the free critical point $0$;

(c).   when $u=\lambda$, $N_\lambda^{-q}(G_\lambda)$ separates  $R_{U_\lambda}(t_1)$ and $R_{U_\lambda}(t_2)$. 
 
 The continuity of $N_u^{-q}(G_u)$ implies that it would separated two sector neighborhoods of the internal arcs  with angles  $t_1$ and $t_2$ in $U_u$ for $u\in\mathcal{V}'$ (here, the  {\it internal arc} with angle $\theta\in \{t_1, t_2\}$ refers to the section  of  the set  $R_{U_u}(\theta)$  that is closed  to the center $c_u$). 
This 
  implies  that $\lambda$ can not be the landing points of  the parameter rays $\mathcal{R}_{\mathcal{H}}(t_1)$ and $\mathcal{R}_{\mathcal{H}}(t_2)$ simultaneously.   
  It contradicts our assumption.
\end{proof}

\section{Proof of  Theorems \ref{main3} and  \ref{main2}} \label{final-s}

In this section, we will prove Theorems \ref{main3} and \ref{main2}. 
To prove these results, it suffices to work in the fundamental domain $\mathcal{X}_{FD}$.








Up to now, we have shown  that $\partial\mathcal{H}_0^1$ and $\partial\mathcal{H}_0^2$ are Jordan curves.  This implies that  each  parameter ray   $\mathcal{R}_0^1(t)$  with $t\in[0, 1/2]$ (resp. $\mathcal{R}_0^2(\theta)$  with $\theta\in[1/2,1]$)  converges to a point on  $\partial{\mathcal{H}_0^1}$  (resp. $\partial{\mathcal{H}_0^2}$).

  Let $\nu_1(t), \nu_2(\theta)$ be the landing points of the parameter rays 
$\mathcal{R}_0^1(t), \mathcal{R}_0^2(\theta)$ respectively.   By Caratheodory's Theorem,  the maps
$$\nu_1: [0,1/2]\rightarrow \partial\mathcal{H}_0^1\cap  \overline{\Omega},  \ \ 
\nu_2: [1/2, 1]\rightarrow \partial\mathcal{H}_0^2\cap  \overline{\Omega}$$
both are  homeomorphisms.
By Theorem \ref{H-tongue1}, we have
$$\nu_1(t)=\nu_2(1-t),\  \forall \ t\in \partial(\Xi).$$


It's known from Lemma \ref{char-S} that 
$\Xi\cup\{0\}$ is the accumulation set of  $\partial(\Xi)$.  Therefore by the continuity of $\nu_\varepsilon$, we get 
$$\nu_1(t)=\nu_2(1-t),\ \forall \ t  \in \Xi\cup\{0\}.$$

Let's look at the Head's angles for the maps in $\mathcal{X}_{FD}\setminus (\mathcal{H}_0^1\cup \mathcal{H}_0^2)=\Omega_0\cup\{\sqrt{3}i/2\}$. 
By Corollary \ref{H-tongue2}, we see that 
$$\boldsymbol{h}(\nu_1(t))=t,  \ \forall \ t\in \partial(\Xi).$$
By Corollary \ref{head-semi-con}, we see that $\boldsymbol{h}$ is continuous at the points $\nu_1(t)$ with $t\in \Xi\setminus\partial(\Xi)$. Therefore by continuity and Fact \ref{head-special} we have
$$\boldsymbol{h}(\nu_1(t))=t,  \forall \  t\in \Xi.$$
This equality implies that the map $\boldsymbol{h}: \Omega_0\cup\{\sqrt{3}i/2\}\rightarrow \Xi$ is surjective. 

To finish, we show that this map is also monotone. In fact, by Corollary \ref{H-tongue2} and Fact \ref{head-special}, we know that 

(a) If  $\theta\in \Xi\cap\Theta_{dya}$, then $\boldsymbol{h}^{-1}(\theta)$ is a singleton.
 
 (b) If $\theta\in \Xi\cap\Theta_{per}$, then   $\boldsymbol{h}^{-1}(\theta)$ is a closed disk $D$ minus a boundary point $w$.  Assume that $\theta$ takes the form $\frac{p}{2^q-1}$, then this closed disk $D$ is  bounded exactly by the curve 
$$\nu_1\Big(\Big[\frac{p}{2^q}, \frac{p}{2^q-1}\Big]\Big)\bigcup \nu_2\Big(\Big[1-\frac{p}{2^q}, 1-\frac{p}{2^q-1}\Big]\Big)$$
and  the boundary point  $w$ is exactly $\nu_1(\frac{p}{2^q})$.   
 
 (c)  If  $\theta\in \Xi\setminus\partial(\Xi)$, we see from the above discussion that  $\boldsymbol{h}^{-1}(\theta)$ is also a singleton.   
 
The discussions in this section can be summarized as follows: 

{\it  The fiber  $\boldsymbol{h}^{-1}(\theta)$ over $\theta\in \Xi$ is a singleton if and only if  $\theta$ is not $\tau$-periodic.
If $\theta$ is $\tau$-periodic, then  $\boldsymbol{h}^{-1}(\theta)$ is homeomorphic to a closed disk minus a boundary point.} 

This completes the proof of Theorem  \ref{main2}, hence also Theorem \ref{main3}.

\end{document}